\documentclass[a4paper,11pt]{amsart}

\usepackage{graphicx} % Required for inserting images

\title{Subconvexity Problem on $\gelle_3$ over number fields: the twist aspect}
\author{Filippo Berta}
\address{
EPFL, Station 8, 1015, Lausanne 
}
\email{filippo.berta@epfl.ch}

\date{}
\usepackage{amsmath}
\usepackage{amssymb}
\usepackage{amsthm}
\usepackage{graphicx}
\usepackage{enumerate}
\usepackage{mathtools}
\usepackage{tikz-cd}
\usetikzlibrary{arrows}
\usepackage{ bbold }
\usepackage{footmisc}
\usepackage{mathrsfs}
\usepackage{yfonts}
\usepackage[margin=30mm,paperwidth=216mm, paperheight=303mm]{geometry}

\newtheorem{theorem}{Theorem}[section]
\newtheorem{proposition}[theorem]{Proposition}
\newtheorem{lemma}[theorem]{Lemma}
\newtheorem{corollary}[theorem]{Corollary}

\theoremstyle{definition}
\newtheorem{definition}[theorem]{Definition}
\newtheorem{remark}[theorem]{Remark}

\newcommand{\N}{\mathbb{N}}
\newcommand{\A}{\mathbb{A}}
\newcommand{\R}{\mathbb{R}}

\newcommand{\Z}{\mathbb{Z}}
\newcommand{\Q}{\mathbb{Q}}

\newcommand{\gelle}{\operatorname{GL}}

\newcommand{\cl}{\mathcal{C}\ell}

\newcommand{\imag}{\operatorname{im}}
\newcommand{\real}{\operatorname{re}}

\newcommand{\fin}{\text{fin}}
\newcommand{\CC}{\mathbb{C}}

\newcommand{\mods}[1]{\,(\mathrm{mod}\,{#1})}

\usepackage{hyperref}
\definecolor{linkcolour}{rgb}{0,0.2,0.6}
\hypersetup{colorlinks,breaklinks,urlcolor=linkcolour, linkcolor=linkcolour}

\begin{document}

\begin{abstract} Let $F$ denote a number field and let $\mathfrak{q}\subset O_F$ traverse a sequence of prime ideals with norm $N(\mathfrak{q}) \to \infty$ and for each $\mathfrak{q}$, let $\chi \in \widehat{F^{\times}\setminus \A^\times}$ be a finite order character of conductor $\mathfrak{q}$. For a fixed unitary cuspidal automorphic representation $\pi$ of $\gelle_3/F$ we show that 
\begin{equation*}
L(\pi \otimes \chi,\tfrac{1}{2})\ll \ N(\mathfrak{q})^{3/4-\kappa}.\end{equation*} holds for all $\kappa< \frac{1}{36}$.
\end{abstract}
\maketitle

\section{Introduction } 

The subconvexity problem for higher rank automorphic forms has a recent history. 
The first result is due to Li. In \cite[Corollary 1.2]{Li2011GL3GL2} Li shows, among other results, that for any \emph{self-dual} Hecke--Mass form $\pi$ for $\operatorname{SL}_3(\Z)$,
\[L(\pi,\tfrac{1}{2}-t) \ll (1+\vert t\vert)^{3/4-\delta}, \]
for any $\delta < \frac{1}{16}$.
Subsequently, Blomer proved a similar result for the non-archimedean twists. More precisely, in \cite{blomersubconvexity} Blomer shows, among other results, that for any \emph{self-dual} Hecke--Mass form $\pi$ for $\operatorname{SL}_3(\Z)$ and any primitive quadratic Dirichlet character $\chi \mods{q}$
\[L(\pi\otimes \chi,\tfrac{1}{2}) \ll q^{3/4-\delta}, \]
for any $\delta < \frac{1}{8}$.

The first result for non necessarily self-dual cusp forms is by Munshi. In \cite{munshioriginal}, Munshi shows that for any fixed Hecke--Mass cusp form $\pi$ over $\gelle_3(\Z)$ that satisfies the Ramanujan-Selberg conjecture and for a sequence of Dirichlet characters of prime conductor $q$ that tends to infinity, the subconvex bound
\[L\left(\pi \otimes \chi, \frac{1}{2}\right)\ll q^{3/4-\delta} \]
holds for any $\delta < \frac{1}{1612}$. In a subsequent work \cite{Munshi1}, Munshi  was able to improve the exponent to $\frac{1}{308}$ and remove the Ramanujan-Selberg conjecture. All these results were possible using a so called $\gelle(2)$-delta symbol, whereby one detects the equality $n_1=n_2$ using the Petersson trace formula. The arguments leading to these results were carefully studied by Holowinsky and Nelson \cite{nelsonholowinsky}, who were able to drastically simplify the proof and obtain the exponent $\delta = \frac{1}{36}$.

In this paper, we deal with the following problem: consider $\pi$ a fixed automorphic representations on $\gelle_3/F$, where $F$ is any number field and let $\chi$ be a sequence of characters in $\widehat{F^{\times}\setminus\A^{\times}}$ of finite order of conductor $N(\mathfrak{q})$, where $\mathfrak{q}$ is an integral prime ideal with norm tending to infinity. The subconvexity problem in the twist aspect is to show the existence of $\kappa_0>0$ such that
\begin{equation}\label{eq:Theorem}L(\pi\otimes \chi,\tfrac{1}{2}) \ll N(\mathfrak{q})^{\frac{3}{4}-\kappa}. \end{equation}for all $\kappa < \kappa_0$. 
The corresponding problem (and other aspects as well) for representations over $\gelle_2$ has been completely solved in \cite{michelvenkatesh} (see also \cite{sparse-equidistribution-venkatesh} for the twist case).

For $\gelle_3$, recent progress has been made by Qi, \cite{Qi2019}, \cite{Qi2024}, in the case where $\pi$ is a self-dual cuspidal representation. Qi succeeded in generalizing Blomer's work \cite{blomersubconvexity} for a general number field $F$ using his theory of complex Bessel and Hankel transforms (see \cite{qi2021theory}). Unfortunately, this argument cannot deal with a non-self-dual automorphic representation $\pi$. In this paper, we generalize the methods of \cite{nelsonholowinsky} to number fields and obtain. 

\begin{theorem}
    The subconvex bound \eqref{eq:Theorem} holds for any $\kappa_0 < \frac{1}{36},$
    provided that the ideal conductor $\mathfrak{C}(\pi)$ of $\pi$ is coprime to the discriminant ideal of $F$. Moreover, the implicit constant depends at most polynomially on $\mathfrak{C}(\pi)$ and on the invariants of $F$.
\end{theorem}

\begin{remark}Since our main focus is not on the conductor of $\pi$, we deliberately assume that $\mathfrak{C}(\pi)$ and the discriminant ideal of $F$ are coprime. This simplifies computations for the Voronoi summation formula. However, it should be possible to perform the computations even in the case when the two ideals are not coprime.  
\end{remark}
\begin{remark}
    The assumption that $\chi$ is of finite order is taken for simplicity. Our method is also applicable for general $\chi$ and one would obtain \eqref{eq:Theorem} with a constant depending on the archimedean conductor of $\chi$. To obtain a hybrid subconvex bound in both the archideman and non-archimdean aspect following the method in this work (and \cite{schumacherphdthesis}), one has to use the theory of complex bessel transforms, developed by Zhi Qi. See Remark \ref{rmk:complexbesseltransform} for more details.
\end{remark}

   Given a finite order character $\chi \colon F^\times \setminus \A^\times \to \CC^\times$ of the conductor $\mathfrak{q}$, one obtains a character on $(O_F/\mathfrak{q})^\times$. This character is an instance of an $\ell$-adic trace function over $O_F/\mathfrak{q}$. From this perspective, if we consider subconvexity \eqref{eq:Theorem} as non-correlation results, it is natural to ask whether other trace functions also show non-correlation with automorphic forms. Fouvry, Kowalski and Michel studied in \cite{fouvry2015algebraic} the non-correlation of $\gelle_2/\Q$ automorphic forms with trace functions over $\Z/(q)$ for $q$ a prime number, and this was extended by Nadarajan in \cite{nadarajan2024algebraictwistsgl2automorphic} to the number field setting. Kowalski, Lin, Michel and Sawin \cite{kowalski2020periodic} have already used the strategy discovered in \cite{nelsonholowinsky} to show non-corellation of trace functions (more generally periodic twists) with automorphic forms over $\gelle_3/\Q$. It seems reasonable to expect that trace functions over $O/\mathfrak{q}$ and automorphic forms over $\gelle_3/F$ also show non-corellation properties. We hope to return to this question in future.

\subsection{Strategy for the proof and structure of the paper} 
We briefly discuss the strategy of the proof.
For $\phi \in \pi^{\infty}$ we consider the Jacquet-- Piateski-Shapiro-- Shalika adelic period $$ I(\phi,\chi,s) = \int_{F^\times\setminus \A^\times}\mathbb{P}\phi\begin{pmatrix}
    y&\\&1
\end{pmatrix}\chi(y)\vert y \vert ^{s-1/2}\ \mathrm{d}^\times y,$$
where $\mathbb{P}\phi$ is some non-complete global Whittaker function, see equation \eqref{eq:definitionPphi} for the precise definition. We choose $\phi_0$ via the local Whittaker models to correspond to the local new-vectors for finite places and to some nice smooth functions for the archimedean places. We choose then $\phi$ to be the right shift of $\phi_0$ by a unipotent element, more precisely $\phi = \phi_0^{u_{12}(\zeta_\mathfrak{q})}$. See Section \ref{sec:approximatefunctionalequation} for details about the new-form and Section \ref{sec:Notation and basic set-up} for the notations $u_{12}$ and $\zeta_\mathfrak{q}$. We can show similarly to \cite{sparse-equidistribution-venkatesh} and \cite{michelvenkatesh} that 
\[N(\mathfrak{q})^{-\frac{1}{2}}L(\pi \otimes \chi,s) \sim I(\phi,\chi,s),\]
and so it suffices to show the existence of some $\delta>0$ such that
\[ I(\phi,\chi,\tfrac{1}{2}) \ll N(\mathfrak{q})^{1/4-\delta}.\]
The sign $\sim$ can be interpreted as equal up to absolutely bounded non-oscillating archimedean factors.
We then write $I(\phi,\chi,s) = \int_{0}^\infty I_t(\phi,\chi)t^{s-1/2}\ \mathrm{d}^\times t$ (again following \cite{sparse-equidistribution-venkatesh} and also \cite{MR4370537}), where $I_t(\phi,\chi)$ is a period over $F^{\times}\setminus\A^{(1)}$, and we see that using an approximate functional equation (see Proposition \ref{lem:convexitybound} and Proposition \ref{prop:approximatefunctionalequation} for a refined version) it suffices to bound 
\[\int_{\substack{ N(\mathfrak{q})^{-3/2-\kappa}<t<N(\mathfrak{q})^{-3/2+\kappa}}}I_{t}(\phi,\chi) \ \mathrm{d}^{\times}t\]
for some $\kappa >0$ small enough. 
We then choose a fundamental domain $\bigsqcup_{I}I(\widehat O^\times E)$, where $I$ runs over fixed representatives of the class group, for integration over $F^{\times}\setminus \A^{(1)}$, and let $(O_F/\mathfrak{q})^\times$ act on this domain and we write

\begin{equation*}\label{eq:intro2}I_t(\phi,\chi) = \sum_{I}\chi(I)\sum_{a\in (O_F/\mathfrak{q})^\times}\chi(\iota_\mathfrak{q}(a))\int_{\widehat O^\times_1(\mathfrak{q})E}\mathbb{P}\phi(a_1(I\iota_\mathfrak{q}(a)yz_t))\chi(y_\infty)\ \mathrm{d}^\times y,\end{equation*}
see Section \ref{sec:Notation and basic set-up} for the definitions of $\widehat O^\times$ and $\widehat O_1^\times(\mathfrak{q})$ and for the explicit choice of $E$.

This allows us to write down a key identity similar to \cite[(3.5)]{nelsonholowinsky}, which is roughly as follows. In what follows we pretend that $\mathfrak{q}=(\alpha_\mathfrak{q})$ is principal. See \eqref{eq:keyidentity} for a more precise formulation.
\begin{align*}
    &I_t(\phi,\chi) = \mathcal{F}_t(\phi,\chi) - \mathcal{O}_t(\phi,\chi), \end{align*}where 
    \begin{equation}\mathcal{F}_t(\phi,\chi)=\frac{N(\mathfrak{q})}{\vert\Delta\vert}\sum_{I}\chi(I)\sum_{\delta}\chi(\iota_\mathfrak{q}(\delta))V_0\left(\frac{\delta}{\Delta}\right)\int_{\widehat{O}_1^\times(\mathfrak{q})E}\mathbb P\phi(a_1(I\iota_\mathfrak{q}(\delta^{-1})yz_t))\chi(y_\infty) \ \mathrm{d}^\times y\label{eq:intro3}\end{equation}and  \begin{align}
        \mathcal{O}_t(\phi,\chi)&=\sum_{\delta \neq 0}\widehat{V_0}\left(\frac{\delta\Delta}{\alpha_\mathfrak{q}^{-1}}\right)\sum_{I}\chi(I)\nonumber\\&\times\sum_{a\in (O_F/\mathfrak{q})^\times}\chi(\iota_\mathfrak{q}(a))\psi_\mathfrak{q}\left(\frac{\iota_\mathfrak{q}(a\delta)}{\alpha_\mathfrak{q}}\right)\int_{\widehat{O}^\times_1(\mathfrak{q}) E} \mathbb P\phi(a_1(I\iota_\mathfrak{q}(\delta^{-1})yz_t))\chi(y_\infty) \ \mathrm{d}^\times y. \label{eq:intro4}
\end{align}
for some compactly supported smooth function $V_0\in C_c^\infty(\A)$ and a parameter $\Delta \in \R^{S_\infty}_{>0}\subset \A^\times$ that we want to choose appropriately. We proceed by bounding \eqref{eq:intro3} and \eqref{eq:intro4} separately, see Propositions \ref{prop:splitFandO}, \ref{prop:boundF} and \ref{prop:boundO}.  To bound \eqref{eq:intro3} we do not use the oscillation $\chi(\iota_\mathfrak{q}(\delta))$ but rather we bound $\mathbb{P}\phi(a_1(\iota_\mathfrak{q}(\delta^{-1}z_t))$ non-trivially, here we ignore $I$ and $y$ for this introduction. We choose $\Delta$ and $V_0$ so that if $V_0\left(\frac{\delta}{\Delta}\right)\neq0$, then $N(\delta) \asymp \vert \Delta \vert < N(\mathfrak{q})$. We use additive reciprocity to write, for $\delta$ so that \eqref{eq:intro3} is non-zero, \[\mathbb P\phi_0^{u_{12}(\zeta_\mathfrak{q})}(a_1(\iota_\mathfrak{q}(\delta^{-1})z_t)) \sim \mathbb{P}\phi_0^{u_{12}(\zeta_\delta)}(a_1(\iota_\delta(\alpha_\mathfrak{q}^{-1})z_t))\]
again $\sim$ means that the two differ by absolutely bounded non-oscillating archimedean factors. An application of the Voronoi summation formula developed by Ichino and Templier in \cite{IchinoTemplier} allows us to obtain a non-trivial bound for $\mathbb{P}\phi_0^{u_{12}(\zeta_\delta)}(a_1(\iota_\delta(\alpha_\mathfrak{q}^{-1})z_t))$. See Remark \ref{rmk:heuristic} for more details.

Second, $\Delta$ must be so that the Fourier transform $\widehat{V_0}\left(\frac{\delta \Delta}{\alpha_\mathfrak{q}^{-1}}\right)$ decreases very rapidly as soon as the norm of $\delta$ is larger than the norm of $\mathfrak{q}$, that is $\vert \Delta \vert > N(\mathfrak{q})^\epsilon$. Choosing $\Delta$ optimally gives us the convexity bound, and therefore to reach subconvexity we need an additional amplification argument (see Section \ref{sec:amplification}). In particular, we consider the character $\chi$ as a signed measure on $(O/\mathfrak{q})^\times$, and we amplify this measure using (principal) prime ideals coprime to $\mathfrak{q}$.

The proof now follows closely \cite{nelsonholowinsky}, with the only difference that it is formulated using the adèlic language.

The structure of the paper is the following. \begin{itemize}
    \item In section \ref{sec:Notation and basic set-up} we fix the notations and the basic notions. 
    \item In section \ref{sec:Counting units}  we collect consequences of Dirichlet's unit theorem that are relevant to our analysis later on, for instance, the Fourier-Whittaker expansion of an automorphic form is a sum over elements in $F$, whereas the arithmetic $L$-function is Dirichlet series over ideals. 
\item In section \ref{sec:gausskloosterman} we prove the necessary local results, including bounds on Kloosterman sums and correlation bounds of Kloosterman sums with smooth functions. 
\item In section \ref{sec:integralrep} we recall the basics on the integral representation of $L$-functions. \item In section \ref{sec:voronoi} we recall and elaborate on the Voronoi summation formula of Ichino-Templier \cite{IchinoTemplier}, and deduce from it the formula we need in our situation. \item In section \ref{sec:keyidentity} we state the Key identity which does not depend on our fixed automorphic representation, but it is really just a statement on the characters of $(O_F/\mathfrak{q})^\times$. \item In section \ref{sec:approximatefunctionalequation} we state and prove the necessary approximate functional equations, and in section \ref{sec:amplification} we introduce an amplification. \item In sections \ref{sec:Bounds for FT} and \ref{sec:Bounds for OT} we conclude. These last two sections are very similar to \cite{nelsonholowinsky}, and it might be useful to first read the sections $4$ and $5$ of this work. 
\end{itemize}

\subsection*{Acknowledgments}
We would like to thank Philippe Michel for numerous corrections and valuable feedback that improved the presentation of the paper. We also thank Yongxiao Lin, Paul D. Nelson and Zhi Qi for useful comments on the paper.

\section{Notation and basic set-up}\label{sec:Notation and basic set-up} Let $F$ denote a number field and $O_F$ its ring of integers. By \emph{place} of $F$ we mean an identity class of discrete valuations of $F$. By Ostrowski's Theorem, we identify the set of (identity classes of) non-archimedean places $v$ of $F$ with $\operatorname{spec}(O_F)$. We denote the underlying prime ideal by $\mathfrak{p}_v \subset O_F$, in particular for $x\in O_F$ the valuation $v(x) \in \Z$ is the biggest integer so that $x\in \mathfrak{p}_v^{v(x)}$. Viceversa, for a prime ideal $\mathfrak{p} \subset O_F$ we denote by $v_{\mathfrak{p}}$ the corresponding discrete valuation.

 For any  non-archimedean place $v$ of $F$ fix the following data:\begin{itemize}\item A completion $F\subset F_v$ of $F$ with respect to the absolute value $\vert{\bullet}\vert_v$ induced by $v$.  \item The ring of $v$-adic integers $O_{v} = \{x \in F_v,\ \vert{x}\vert_v\leq 1\} \subset F_v$ and its subgroup of units $O_{v}^{\times}=\{x \in F_v,\ \vert{x}\vert_v=1\}.$
    \item A uniformizer $\varpi_v \in O_v$. 
\end{itemize}

The set of archimedean places, denoted here by $S_{\infty}$, is identified with the union of real embeddings, that is, $\rho_v \colon F \hookrightarrow \CC$ so that $\rho_v(F)\subset \R$, and a choice of one of the two conjugates $\rho_v,\overline{\rho_v}\colon F \hookrightarrow \CC$, when $\rho_v(F) \nsubseteq \R$. They will be called \emph{real} embeddings and \emph{complex} embeddings, respectively. Let $d_{\infty} = \vert{S_{\infty}}\vert$. For a place $v$ of $F$ we will usually write the condition $v\notin S_{\infty}$ as $v<\infty$.

Let $\A$ denote the adèle ring of $F$ and $\A^{\times}$ the idèle group of $F$. We embed, as usual, $F\subset \A$ and $F^{\times}\subset \A^{\times}$ diagonally. For any finite set of places $S$ of $F$ we denote $\A^{S} = \prod'_{v\notin S}(F_v,O_v)$ and $(\A^{\times})^{S} = \prod_{v\notin S}'(F_v^{\times},O_v^{\times})$ the adèle ring and the idèle group, respectively, defined omitting those places $v\in S$ from the definition. We will eventually use the notation $\A_{\fin} = \A^{S_{\infty}}$ and call it the ring \emph{of finite adèles}. For a finite set of places $R$ we will denote $\A_R = \prod_{v\in R}F_v$ and $\A_{R}^{\times} = \prod_{v\in R}F_v^{\times}$. We write $\widehat{O} = \prod_{v<\infty}O_{v}\subset \A_{\fin}$ and $\widehat{O}^{\times} = \prod_{v<\infty}O_{v}^{\times}\subset \A_{\fin}^{\times}$. 
 For an integral ideal $\mathfrak{a}$ we denote \[\widehat{O}^{\times}_{1}(\mathfrak{a}) =  \prod_{v|\mathfrak{a}}\left(1+\varpi_v^{v(\mathfrak{a})}O_{v}\right) \times \prod_{\substack{v<\infty\\v\nmid \mathfrak{a}}}O_v^{\times} \subset \A_{\fin}^{\times}.\]
 We also set $F_{\infty} = \prod_{v\in S_\infty}F_v$ and $F^{+}_{\infty} = \prod_{v\ \text{real}}\R_{>0} \times \prod_{v\ \text{complex}}\CC^{\times}$.

For functions $f\colon \A \to \CC^{\times}$, that is split as a product $f = \prod_vf_v$ for functions $f_v\colon F_v\to \CC^{\times}$ such that for all but finitely many $v<\infty$ it holds that $f_v|_{O_v}\equiv 1$, we set $f^S = \prod_{v\notin S}f_v$, and $f_S = \prod_{v\in S}f_v$.

We define the adèlic norm map \[\vert{\cdot}\vert\colon \A \to \R_{\geq 0};\ (y_v)_v\mapsto \prod_v \vert{y_v}\vert_v,\]
and the norm of a (finite) idèle 
\[ N\colon \A^{\times}\to \R_{>0};\ (y_v)_v \mapsto \prod_{v<\infty}\vert{y_v}\vert^{-1}_v.\]
The first induces a direct product decomposition $A_1\A^{(1)} = \A^{\times},$ where $\A^{(1)}$ is the kernel of the norm map restricted to $\A^{\times}$ and $A_1\simeq \R_{>0}$: this can and will be seen via the map $\R_{>0} \hookrightarrow \prod_{v \in S_{\infty}}F_v^{+};\ t \mapsto z_t=(t^{\frac{1}{d_v}})_v$, where $d_v=d_{\infty}$ if $v$ is real and $d_v=2d_{\infty}$ if $v$ is complex. 

We denote \[\vert{\cdot }\vert_{\infty} \colon F_{\infty} \to \R; x \mapsto \prod_{v\in S_{\infty}}\vert{x}\vert_v,\] we set
$F_{\infty}^{(1)}$ as its kernel. Also, we set $F_{\infty}^{(1),+}$ to be the intersection $F^{(1)}_{\infty}\cap F_{\infty}^+$. 

It is true that $F^{\times} \subset \A^{(1)}$.
We embed $\A^{\times}_{\fin} \hookrightarrow \A^{(1)}$ by setting for $x\in \A^{\times}_{\fin}$ the archimedean 'coordinate' $x_{\infty} = z_{N(x)}\in F_{\infty}^{\times}$. 

For a prime integral ideal $\mathfrak{p}$  (or alternatively a finite place $v$ depending on the context) of $F$ we denote $\iota_{\mathfrak{p}}\colon F_{\mathfrak{p}}^{\times}\to \A^{\times}_{\fin}\subset \A^{(1)}$ that maps to the finite idèle which is $1$ for each place $v\neq v_{\mathfrak{p}}$ and to $x$ at the place $v_{\mathfrak{p}}$. In particular, note that for $v\in S_{\infty}$ the value of $\iota_{\mathfrak{q}}(x)_{v}$ is $\vert{x}\vert_{v_{\mathfrak{q}}}^{-\frac{1}{d_v}}$. For an integral ideal $\mathfrak{a}$ we define $\iota_{\mathfrak{a}} = \prod_{v|\mathfrak{a}}\iota_v$. 

The gothic letters $\mathfrak{a,b,p,q}$ etc. are reserved to integral ideals. We identify an ideal $\mathfrak{a}$ with the idèle $(\varpi_v^{v(\mathfrak{a)}})_v \in \A^{\times}$. This identification will be used only for objects that are invariant under the action of $\widehat{O}^{\times}$.  

\subsection{Asymptotic notation}
For two quantities $A$ and $B$ we write $A\ll B$, or $A = O(B)$, if there exists a constant $C$, independent of $A$ and $B$, so that $\vert{A}\vert\leq C\vert{B}\vert$. We write $A\asymp B$ if $A\ll B$ and $B\ll A$. If we write $A\ll_j B$, to highlight that the constant $C$ depends on a variable $j$. We say that a quantity $A$ is negligible small, if $A\ll_B N(\mathfrak{q})^{-B}$ for all $B>0$.

Let $\epsilon >0$ denote a small positive number. We adopt the convention that $\epsilon$ can change from line to line, but it is a fixed variable.
In general, we allow constants to depend implicitly on fixed constants such as $\epsilon$ or the field $F$.

\subsection{Choice of fundamental domain for $F^{\times}\setminus\A^{(1)}$}\label{subsec:2.choice of fundamental domain}
First, we fix a set of representatives $
    \mathcal{S}_{\cl}=(I_i)_{1\leq i \leq \vert{\cl}\vert} \subset \A_{\fin}^{\times}$ for the Class group. We do so that for each place $v$ we have $v(I_i)\leq 0$. That is, we are choosing them to be the "opposite" of integral, the reason is to simplify some computations afterwards. In addition, we assume that each $I \in \mathcal{S}_{\cl}$ is coprime to $\mathfrak{C}(\pi)$ and $\mathfrak{q}$.

For the archimedean side we define $E_{\infty}\subset F_{\infty}^{\times}$ as follows: first fix a place $v_0 \in S_{\infty}$. Consider the map \begin{equation*}l \colon F_{\infty}^{(1)}\to \R^{d_{\infty}-1}; (x_v)_v \mapsto (\log\vert{x_v}\vert_v)_{v\neq v_0}.\end{equation*} Let $\{\epsilon_1,\dots,\epsilon_{d_{\infty}-1}\}$ denote a basis for $O_F^{\times}/\mu_F$, where $\mu_F$ is the finite set of roots of unity of $F$. Let $P\subset \R^{d_{\infty}-1}$ denote the convex polytope generated by $\{l(\epsilon_1),\dots,l(\epsilon_{d_{\infty}-1})\}$. We define \begin{equation}\label{eq:definitionE}
    E = \left\{b\in l^{-1}(P),\ 0\leq \operatorname{arg}b_{v_0} \leq \frac{2\pi}{\vert{\mu_F}\vert}\right\}.\end{equation}
 Then \begin{equation*}\label{eq:fundamentaldomain2}\mathcal{E} = \bigsqcup_{I\in \mathcal{S}_{\cl}} I(\widehat{O}^{\times}\times E)\end{equation*}is a fundamental domain for $F^{\times}\setminus
\A^{(1)},$ see \cite[Chapter VII §3,\ Theorem 6]{lang1986algebraic}.

\subsection{Special elements}
Given an ad\`ele  $x=(x_v)_v \in \A$ we will denote by $\zeta_x \in \A$ the (finite) ad\`ele obtained by setting $(\zeta_x)_v = 0$ whenever $x_v \in O_{v}$ or $v\in S_{\infty}$, and $(\zeta_x)_v = x_v$ otherwise. 
We also fix the following notation for matrices (with the convention that empty spaces are $0$) in $\gelle_3$:
\begin{align}
    a_1(y) &= \begin{pmatrix}
        y&&\\&1&\\&&1
    \end{pmatrix},\qquad y\neq 0,\quad 
    a(x,y)= \begin{pmatrix}
        xy&&\\&x&\\&&1
    \end{pmatrix}, \qquad x,y\neq 0,\nonumber\\
    u_{12}(x) &=\begin{pmatrix}
        1&x&\\&1&\\&&1
    \end{pmatrix},\quad
    \sigma_{23}=\begin{pmatrix}
        1&&\\&&1\\&1&
    \end{pmatrix}, \qquad \sigma_{13} = \begin{pmatrix}
        &&1\\&1&\\1&&
    \end{pmatrix}.\label{eq:u12}
\end{align}
For any $F$-algebra $R$ and any integer $n\geq 2$ we denote
\begin{align*}
    U_n(R) = \left\{\begin{pmatrix}
        1&x_{12}&\cdots &x_{1n}\\&\ddots&\ddots &\vdots\\&&1&x_{n-1n}\\&&&1 
    \end{pmatrix}, \ x_{ij} \in R\right\}
\end{align*}

\subsection{Additive character}Let $\A_{\Q}$ be the ring of adèles of $\Q$. Let $\psi_{\Q} \colon \Q\setminus \A_{\Q} \to \CC$ the standard additive character, that is $\psi_{\Q} = \prod_{p}\psi_{\Q_p} \times \psi_{\Q,\infty}$ and the local characters are as follows: 
$\psi_{\Q,\infty}(x) = e^{-2\pi ix}$, $x\in \R$, and for any finite prime number $p$ we have $\psi_{\Q_p}(x) = e^{2\pi i\{x\}}$, $x \in \Q_p$, where $\{x\}$ is defined so that $x-\{x\} \in \Z_p$.

The standard character $\psi \colon F\setminus\A \to \CC$ is defined as $\psi_{\Q}\circ \operatorname{tr}_{\A_F/\A_\Q}$. It satisfies $\psi = \prod_{v<\infty} \psi_v \times \prod_{v\in S_{\infty}}\psi_v$, where for $v< \infty$ let $p$ the prime under $v$ so that $F_v/\Q_p$ is a finite (possibly trivial) extension, then $\psi_v = \psi_{\Q_p}\circ \operatorname{tr}_{F_v/\Q_p}$ for $v< \infty$. For $v\in S_{\infty}$ let $\psi_v(x) = e^{-2\pi i\operatorname{tr}_{F_v/\R}(x)}.$   

\begin{remark}  We denote $\mathfrak{D}\subset O_F$ the discriminant ideal of $F/\Q$.
The standard character $\psi = \prod_{v}\psi_v$ we just defined is ramified at those prime ideals $\mathfrak{p}$ which are ramified over $\Q$: that is the prime ideals that divide $\mathfrak{D}$.
\end{remark}

\subsection{Measures}
The choices of Haar measures are not so crucial as long as we remain consistent, since we are interested only in upper bounds.  

For any place $v$ we endow $F_v$ with the additive Haar measure $\mathrm{d}x_v$, which is self-dual with respect to the Fourier transform with respect to $\psi_v$. In particular, or finite $v$ we have $\operatorname{vol}(O_v) = \vert{\mathfrak{D}}\vert_v^{1/2}$.
We endow $\A$ with the Haar measure $\mathrm{d}x$ normalized so that for any measurable function $f = \prod_{v}f_v$ with $f_v = \mathbb{1}_{O_v}$ for all but finitely many $v< \infty$ and $f_v$ is measurable for every $v$, then 
\[\int_\A f(x)\ \mathrm{d}x = \prod_{v}\int_{F_v}f_v(x_v)\ \mathrm{d}x_v.\]

Similarly, we endow $F_v^{\times}$ with the Haar measure $\mathrm{d}^{\times}x_v$ that assigns to $O_v^{\times} \subset F_v^{\times}$ volume $1$, when $v<\infty$ and it assigns volume $1$ and we set $\mathrm{d}^\times x=\frac{\mathrm{d}x}{\vert{x}\vert}$ for $v\in S_{\infty}$. Notice that for $v<\infty$ we have $d^{\times}x_v = \operatorname{vol}(\mathfrak{D}_v)^{-1/2}(1-\frac{1}{N(\mathfrak{p}_v)})^{-1}\frac{\mathrm{d}x_v}{\vert{x}_v\vert}.$ In particular, $\mathrm{d}^{\times} x = \frac{\mathrm{d}x}{\vert{x}\vert}$. 

We endow $\R_{>0}$ with the Haar measure $\mathrm{d}^{\times }t= \frac{\mathrm{d}t}{t}$ and push forward to $A_1$. We endow then $\A^{(1)}$ with the measure $\mathrm{d}^\times x$ so that \[\int_{\A^{\times}}f(x)\ \mathrm{d}^{\times}x= \int_{\R_{>0}}\int_{\A^{(1)}}f(z_tx)\ \mathrm{d}^{\times}x\ \mathrm{d}^\times t\]
for any integrable function $f$.

\subsection{Newforms}\label{subsect:newforms} Let $K$ be a non-archimedean local field with ring of integers $O$ and unique maximal ideal $\mathfrak{p}$.
We recall briefly the local theory of new-vectors for irreducible generic representations of  $\gelle_n/K$ (see \cite{conducteurgln} and \cite{matringe2012essential}).

Any character $\psi$ on $K$ induces a character, which we still denote $\psi$ on $U_n(K)$: 
\begin{equation}\label{eq:psiextended} \psi\left(\begin{pmatrix}
    1&x_{12}&\cdots&x_{1n}\\&\ddots&\ddots&\vdots\\&&1&x_{n-1n}\\&&&1
\end{pmatrix}\right) = \psi(x_{12}+\cdots + x_{n-1n}).\end{equation}

Let $\rho$ be an irreducible unitary generic representation of $\gelle_3(K)$ and $\psi$ an unramified character on $K$. Let $\mathcal{W}(\rho,\psi)$ denote the Whittaker model of $\rho$. By the theory of local new-vectors, developed by Jacquet e Shalika, \cite{conducteurgln}, and with a significant contribution of Matringe \cite{matringe2012essential}, there is a unique vector $W_\rho^\circ \in \mathcal{W}(\rho,\psi)$ satisfying the following properties:

\begin{itemize}
\item\label{itm:new-vector-normalization} $W_\rho^{\circ}(I_3) = 1$
\item $W_\rho^{\circ}\left(g\begin{pmatrix}k&\\&1\end{pmatrix}\right) = W_\rho^{\circ}(g)$  for all $g\in \gelle_3(K)$ and $k\in \gelle_2(O)$.
    \item For any unramified irreducible generic unitary representation $\sigma$ of $\gelle_2(K)$ let $W_{\sigma} \in \mathcal{W}(\sigma,\overline \psi)$ denote the spherical vector of $\sigma$ as in \cite{shintani} (normalized to take value $1$ at the identity $I_2$). Then, for $\real(s)>1$,  \begin{align*}\int_{U_{2}(K)\setminus\gelle_2(K)}W_\rho^{\circ}\begin{pmatrix}
        h&\\&1
    \end{pmatrix}W_{\sigma}(h)\vert{h}\vert_v^{s-\frac{1}{2}} \ \mathrm{d}h = L(\rho\times \sigma,s).\label{eq:newform1}\end{align*}
     \end{itemize}
If $\rho$ is itself unramified, then the unique spherical vector (provided that it takes value $1$ at the identity element) of \cite{shintani} is also the new-vector.

It is shown in \cite{matringe2012essential} that the new-vector satisfies also the following property: for any unramified character $\eta$ of $K^\times$ we have for $\real(s)$ big enough
\begin{equation*}
    \int_{K^\times}W_\rho^\circ(a_1(y))\eta(y)\vert y\vert^{s-1} \mathrm{d}^\times y = L(\rho \times \eta,s).
\end{equation*}

It is shown in \cite{conducteurgln} that the new-vector is right-invariant under the action of 
\[K_1(\mathfrak{p}^{a(\rho)}) = \left\{\begin{pmatrix}
    k&\omega\\\nu^t&u
\end{pmatrix}\in \gelle_3(O) |  \ \nu \in \mathfrak{p}^{a(\rho)}\times \mathfrak{p}^{a(\rho)}, u\equiv 1\mods{\mathfrak p^{a(\rho)}} \right\}, \]
where $a(\rho)$ denotes the exponent conductor of $\rho$. Moreover, this property characterizes uniquely, up to normalization, the new-vector.

We now consider the number field $F$ and the fixed cuspidal automorphic representations $\pi$ and its contragredient $\tilde{\pi}$. By Flath's theorem we have isomorphisms $\pi \simeq \otimes_v'\pi_v$, $\tilde{\pi} \simeq \otimes_v' \tilde{\pi}_v$ as restricted tensor product of local representations and $\tilde{\pi}_v$ is the contragredient representation of $\pi_v$ for every $v$. For any finite place $v<\infty$ let $\psi_v^\circ(x) = \psi_v(\varpi_v^{-v(\mathfrak{D})}x)$. We denote \begin{align*}W_v^\circ \in \mathcal{W}(\pi_v,\psi_v^\circ), \quad U_v^\circ \in \mathcal{W}(\tilde{\pi}_v,\overline{\psi_v^\circ})\end{align*}
the corresponding new-vectors.

Let $v<\infty$ be a place so that $v(\mathfrak{D})>0$, that is, $\psi_v$ is ramified, consider the following functions, obtained by left translation of the new-vectors in $W^\circ_{v}\in \mathcal{W}(\pi_v,\psi_v^\circ)$ (and $U_v^\circ \in \mathcal{W}(\tilde{\pi}_v,\overline{\psi_v^\circ})$ respectively) by $a(\varpi_v^{v(\mathfrak{D})},\varpi_v^{v(\mathfrak{D})})$. For instance, we have
    \[\ell\left(a(\varpi^{v(\mathfrak{D})},\varpi_v^{v(\mathfrak{D})})\right)W_{v}^{\circ} \colon g\mapsto  W_{v}^{\circ}\left(\begin{pmatrix}
    \varpi_v^{2v(\mathfrak{D})}&&\\&\varpi_v^{v(\mathfrak{D})}\\&&1
    \end{pmatrix}g \right)\]
    and similarly for $\ell\left(a(\varpi^{v(\mathfrak{D})},\varpi_v^{v(\mathfrak{D})})\right)U_{v}^{\circ}$.
    Then \begin{align*}\ell\left(a(\varpi_v^{v(\mathfrak{D})},\varpi_v^{v(\mathfrak{D})})\right)W_{v}^{\circ} \in \mathcal{W}(\pi_v,\psi_v), \quad \ell\left(a(\varpi^{v(\mathfrak{D})},\varpi_v^{v(\mathfrak{D})})\right)U_{v}^{\circ} \in \mathcal{W}(\tilde{\pi}_v,\overline{\psi_v})
    \end{align*} 
    and it is right-$\gelle_2(O_v)$ (or even $\gelle_3(O_v)$ if $\pi_v$ is unramified) invariant.

    \begin{remark}\label{rmk:newform}
        From now on, abusing notation, we will denote $\ell(a(\varpi_v^{v(\mathfrak{D})},\varpi_v^{v(\mathfrak{D}))})W_v^{\circ}$ also by $W_v^{\circ}$ and refer to it as the new-vector in $\mathcal{W}(\pi_v,\psi_v)$. Similarly, we denote $$\ell(a(\varpi_v^{v(\mathfrak{D})},\varpi_v^{v(\mathfrak{D}))}) U_v^\circ \in \mathcal{W}(\tilde{\pi}_v,\overline{\psi_v})$$ by $U_v^\circ$. We will always specify the Whittaker model in which we consider the new-vectors to live in order to make the statement clear.
        Note that for an unramified character $\eta \colon F_v^{\times}\to \CC^{\times}$ and $W_v^{\circ}\in \mathcal{W}(\pi_v,\psi_v)$ the newform we have for any $\real(s)$ large enough that
        \[\int_{F_v^{\times}}W_v^{\circ}(a(\varpi_v^{-v(\mathfrak{D})},y\varpi_v^{-v(\mathfrak{D})}))\eta(y)\vert{y}\vert^{s-1}\ \mathrm{d}^{\times}y = L(\pi\times\eta,s). \]
    \end{remark}
    
\begin{remark}
    The new-vector $W_v^{\circ}\in \mathcal{W}(\pi_v,\psi_v^{\circ})$ satisfies 
    \[ W_v^{\circ}\begin{pmatrix}
        \varpi_v^{l_1}&&\\&\varpi_v^{l_2}&\\&&\varpi_v^{l_3}
    \end{pmatrix}=0\]
    if $l_2>l_1$ or $l_3>l_2$. This can be deduced directly from the right $\gelle_2(O_v)$-invariance or by the formulae in \cite{Shalika} (for the unramified case) and \cite{matringe2012essential} (for the ramified case).
    Hence, the new-vector $W_v^{\circ} \in \mathcal{W}(\pi_v,\psi_v)$ satisfies 
    \[W_v^{\circ}\begin{pmatrix}
        \varpi_v^{l_1}&&\\&\varpi_v^{l_2}&\\&&\varpi_v^{l_3}
    \end{pmatrix}=0,\]
    whenever $l_2>l_1+v(\mathfrak{D})$ or $l_3>l_2+v(\mathfrak{D})$.
\end{remark}

\section{Counting units}\label{sec:Counting units}
A source of difficulty when working with an extension of $\Q$ that is neither trivial nor an imaginary quadratic field is that the set of units $O_F^{\times}$ is infinite. 

The classical result that one uses is the Dirichlet's unit theorem (see \cite[Unit Theorem]{lang1986algebraic}, with some changes in the notation):
\begin{theorem}[Dirichlet's unit theorem]
 Consider the $\log$ map 
    \[\operatorname{Log}_{S_\infty}\colon F^{\times} \to \prod_{v\in S_\infty}\R;\ x \mapsto \operatorname{log}(\vert{x}\vert_v)_v.\]
    Then $\operatorname{Log}_{S_\infty}(O_F) \subset \R^{d_{\infty}}$ is a lattice of rank $d_{\infty}-1$ contained in the hyperplane $\{x \in \R^{d_\infty}\ | \ x_1+\dots + x_{d_\infty} = 0\}.$
    Moreover, the kernel of $\operatorname{Log}_{S_{\infty}}$ is $\mu_F$.
\end{theorem}

We can state the following three consequences that will be useful.
\begin{lemma}\label{lem:unithteorem1}Let $v_0 \in S_{\infty}$ and let $S_{\infty}' = S_{\infty}\smallsetminus \{v_0\}$. Then for any pair of tuples $(T_v)_{v}, (t_v)_v \in \R_{>0}^{d_{\infty}-1}$  of positive numbers it holds that 
    \[\vert{\{o\in O_F^{\times}\ | \ \forall v \in S'_{\infty} : t_v\leq \vert{o}\vert_v \leq T_v\}}\vert \ll 1+\prod_{v\in S'_{\infty}}\max(\log T_v-\log t_v,0)\]
    and that 
     \[\vert{\{o\in O_F^{\times}\ | \ \forall v \in S'_{\infty} : t_v\leq \vert{o}\vert_v \leq T_v\}}\vert \gg \prod_{v\in S'_{\infty}}\max(\log T_v-\log t_v,0),\]
     where the implicit constants depend only upon $F$. 
\end{lemma}
\begin{proof}The statement is trivially true for $S_{\infty} = \{v_0\}$ or $t_v\geq  T_v$ for some $v\in S_{\infty}$. So suppose these are not the case.
    Let $(Y_v)_{v\neq v_0},(y_v)_{v\neq v_0} \in \R^{d_{\infty}-1}$ be so that $y_v < Y_v$ for each $v$. 
    By Dirichlet's unit theorem $L = \operatorname{Log}_{S_{\infty}}(O^{\times}_F)$ is a lattice of rank $d_{\infty}-1$.
    
    Then by constructing spheres of appropriate radius around each lattice point we deduce that there exists constants $c_L,C_L >0$ so that  
    \[c_L\prod_{v\in S_{\infty}'}(Y_v-y_v) \leq \vert{\{x \in L\ | \ \forall v \in S_{\infty}':y_v\leq x \leq Y_v \}}\vert \leq C_L\prod_{v\in S'_{\infty}}(Y_v-y_v),\]
    that translates to
    \begin{align*}&c_L\prod_{v\in S_{\infty}'}(Y_v-y_v)\leq  \vert{\{o \in O_F^{\times}\ | \  \forall v\in S'_{\infty}: y_v\leq \log\vert{o}_v\vert \leq  Y_v\}}\vert \\&\leq C_L\vert{\mu(F)}\vert \prod_{v\in S_{\infty}'}(Y_v-y_v).\end{align*}
    Choosing $Y_v = \log T_v$ and $y_v = \log t_v$ shows the assertion.
\end{proof}

\begin{lemma}\label{lem:countingunits}
    Let $(T_v)_v \in \R_{>0}^{d_{\infty}}$ and let $T = \prod_{v}T_v$. Suppose that for each $v$ there exists $\delta >0$ and constants $C_{\delta}>0$ so that for all $v\in S_{\infty}$ \[  T_v\leq C_{\delta}T^{\frac{1}{d_{\infty}}+\delta}.\] Then
    \begin{equation*}\label{eq:unittheorem1}
        \vert{\{o \in O_F^{\times}\ | \  \forall v \in S_{\infty}: \vert{o}\vert_v\leq T_v}\}\vert \ll (\log C_{\delta}T^{\frac{1}{d_{\infty}}+\delta})^{d_{\infty}-1}, \end{equation*} 
    where the implicit constant only depends upon $F$. 
\end{lemma}
\begin{proof} Let $o\in O_F^{\times}$ so that $\vert{o}\vert_v \leq T_v$ for all $v\in S_{\infty}$. Suppose that there exists $v_0$ so that $\vert{o}\vert_{v_0} < C_{\delta}^{^{1-d_{\infty}}}T^{-(d_{\infty}-1)(\frac{1}{d_{\infty}}+\delta)}$. Then 
\[  C_{\delta}^{d_{\infty}-1}T^{(d_{\infty}-1)(\frac{1}{d_{\infty}}+\delta)}\geq \prod_{v\in S_{\infty}, v\neq v_0}\vert{o}\vert_v  = \vert{o}\vert_{v_0}^{-1} > C_{\delta}^{d_{\infty}-1}T^{(d_{\infty}-1)(\frac{1}{d_{\infty}}+\delta)}, \]
which is absurd. Hence, we see that the cardinality on the left-hand side of \eqref{eq:unittheorem1} is bounded by  
\[\vert{\{ o\in O_F^{\times}\ | \  \forall v \in S_{\infty}: C_{\delta}^{^{1-d_{\infty}}}T^{-(d_{\infty}-1)(\frac{1}{d_{\infty}}+\delta)}\leq \vert{o}\vert_v \leq T_v \}}\vert\]
and latter is bounded by the previous corollary by
\[\min_{v_0\in S_{\infty}}\prod_{v\in S_{\infty}\smallsetminus\{v_0\}}(\log T_v-(1-d_{\infty})\log C_{\delta}T^{\frac{1}{d_{\infty}}+\delta},0)\ll \log (C_{\delta}T^{\frac{1}{d_{\infty}}+\delta})^{d_{\infty}-1}, \]
where the implicit constant depends only upon $F$.
\end{proof}

\begin{lemma}\label{lem:choicegenerator}
    Let $(\gamma) \subset O_F$ be a non-zero integral principal ideal, $T=(T_v)_v \in \R_{>0}^{d_{\infty}}$ so that $\prod_{v\in S_{\infty}}\vert{T_v}\vert_v = N(\gamma)$. There exists a constant $C>0$ depending only on $F$ so that for each $\frac{1}{d_{\infty}}>\delta > \frac{C}{\log N(\gamma)}$, there exists a unit $o\in O_F^{\times}$ satisfying 
    \[\forall v \in S_{\infty}: \vert{T_v}\vert_vN(\gamma)^{-\delta(d_{\infty}-1)}\leq \vert{o\gamma}\vert_v \leq \vert{T_v}\vert_vN(\gamma)^{\delta(d_{\infty}-1)}.  \]
\end{lemma}
\begin{proof} Let $v_0 \in S_{\infty}$ and let $S'_{\infty} = S_{\infty}\smallsetminus\{v_0\}$. We may assume that $S'_{\infty} \neq \varnothing$, otherwise the statement is clear.
 Let \[B = \{o \in O_F^{\times}\ | \  \forall v \in S_{\infty}' : \vert{T_v}\vert_vN(\gamma)^{-{\delta}} \leq \vert{\gamma o}\vert_v \leq \vert{T_v}\vert_vN(\gamma)^{\delta}\}. \]
 By Lemma \ref{lem:unithteorem1} we have that
 \[\vert{B}\vert \gg (2\delta\log N(\gamma))^{d_{\infty}-1}, \]
 where the implicit constant depends only upon $F$. 
 For $o\in B$ we have then $\vert{o}\vert_{v_0} = \prod_{v\neq v_0}\vert{o}\vert_ v$ and so \[\vert{T_{v_0}}\vert_{v_0} N(\gamma)^{-(d_{\infty}-1)\delta} \leq \vert{\gamma o}\vert_{v_0} \leq \vert{T_{v_0}}\vert_{v_0}N(\gamma)^{(d_{\infty}-1)\delta}.\]
 
\end{proof}

\section{Gauss sums and Kloosterman sums}\label{sec:gausskloosterman}

\subsection{Gauss sums}
Let $K$ be a non-archimedean local field (of characteristic $0$) with finite residue field and denote by $\vert \cdot \vert$ the attached absolute value. Let $O$ be its ring of integers, $\mathfrak{p}$ be the maximal ideal of $O$ and $O^{\times} = O\smallsetminus\mathfrak{p}$ be the group of units. Fix a generator $\varpi$ of $\mathfrak{p}$ and let $p = \vert O/\mathfrak{p}\vert$.

Let $\psi \colon K \to \mathbb{C}^{(1)}$ be an additive character and  $\eta \colon K^{\times}\to \mathbb{C}^{(1)}$ be a  multiplicative character. Since $K^{\times} = \varpi^{\Z}O^{\times}$ we can write $\eta = \eta_0\vert{\ }\vert_v^{it}$ for some $\eta_0 \in \widehat{O^{\times}}$ and $t \in \R/2\pi i \Z \log p.$ Denote by $\mathfrak{p}^{r}$ and $\mathfrak{p}^{s}$ the respective conductors, which means that $r$ is the smallest positive integer so that $\eta$ is trivial in $1+\mathfrak{p}^{r}$, or $r=0$ if $\eta_0$ is trivial and $s$ is the smallest integer so that $\psi$ is trivial in $\mathfrak{p}^{s} = \varpi^s O_F$. Notice that $s$ can be negative. The existence of conductors is ensured by continuity. We denote $r$ by $a(\eta)$ and $s$ by $a(\psi)$.

Fix a Haar measure $\mathrm{d}x$ on $K$ a Haar measure on $K^{\times}$. Then $\mathrm{d}^{\times}{x} = c\frac{\mathrm{d}x}{\vert{x}\vert}$ for some $c>0$. The Gauss sum associated to $\eta$ and $\psi$ is the following integral.
\begin{equation*}
    G(\eta,\psi) = \int_{O^{\times}}\eta(y)\psi(y)\ \mathrm{d}^{\times}y.
\end{equation*}
From \cite[7-4 Lemma]{fourieranalysisnumberfields}:
\begin{lemma}\label{lem:gausssum}We have
    \begin{enumerate}
        \item If $a(\psi)<a(\eta)$, then $G(\eta,\psi) = 0$.
        \item If $a(\eta)=a(\psi)$, then \[\vert{G(\eta,\psi)}\vert^2 = c\operatorname{vol}(O,\mathrm{d}x)\operatorname{vol}(1+\mathfrak{p}^{a(\eta)},d^\times x)\]
        \item If $a(\psi)>a(\eta)$, then \[\vert{G(\eta,\psi)}\vert^2 = c\operatorname{vol}(O,\mathrm{d}x)\left(\operatorname{vol}(1+\mathfrak{p}^{a(\eta)},\mathrm{d}^\times x)-p^{-1}\operatorname{vol}(1+\mathfrak{p}^{a(\psi)-1},\mathrm{d}^\times x)\right).\]
    \end{enumerate}
    
\end{lemma}

Suppose now that $\eta$ is a ramified character, that is, $a(\eta)>0$.
Let $W \colon \gelle_3(K) \to \CC$ be a continuous function and suppose that $W$ is right invariant by the action of $a_1(O^\times)$, that is, $W\left(ga_1(u)\right) = W(g)$ for all $g\in \gelle_3(K)$ and $u\in O^{\times}$, whenever it converges one has 
\begin{equation*}\label{eq:local0newform}\int_{K^{\times}}W(a_1(y))\eta(y) \ \mathrm{d}^{\times}y = \sum_{k\in \Z}W(a_1(\varpi^k))\eta(\varpi)^k\int_{O^{\times}}\eta(u) \ \mathrm{d}^{\times}u= 0. 
\end{equation*}

As in \eqref{eq:psiextended} we extend $\psi$ to a character on $U_3(K)$ and suppose that $W$ is such that $$W(ug) = \psi(u)W(g)$$ for any $u\in U_3(K)$ and $g\in \gelle_3(K)$. Then we have the following Lemma 
\begin{lemma}\label{lem:shiftnewform}Let $\psi,\eta$ and $W$ as before. Suppose, moreover, that $W(a_1(y))= 0$ for all $y \notin O$ and $W$ is right-invariant by the action of $a_1(O^\times)$. Then, whenever the integral converges, we have
\[\int_{K^\times}W(a_1(y)u_{12}(\varpi^{-a(\eta)+a(\psi)}))\eta(y)\vert y\vert^{s}\ \mathrm{d}^\times y = G(\eta,\psi(\varpi^{-a(\eta)+a(\psi)}\cdot)) W(I_3).\]
\end{lemma}
\begin{proof} For any $y\in K^\times$ we have $W(a_1(y)u_{12}(\varpi^{-a(\eta)+a(\psi)}))= \psi(y\varpi^{-a(\eta)+a(\psi)})W(a_1(y))$. Therefore, supposing that the integral is convergent, we see that it is equal to
\begin{align*}
    \sum_{k=0}^\infty W(a_1(\varpi^k))\eta(\varpi)^k p^{-ks}\int_{O^{\times}}\psi(u\varpi^{k-a(\eta)+a(\psi)})\eta(u)\ \mathrm{d}^\times u.\end{align*}
    The conductor of $\psi( \varpi^{k-a(\eta)+a(\psi)})$ is $a(\eta)-k \leq a(\eta)$, so the inner integral is always $0$, unless $k=0$ and we have the desired result.
\end{proof}

\subsection{Kloosterman sums}
Let $\psi$ be an additive character of conductor $a(\psi)\in \Z$. Let also $l\in \Z_{>0}$, $\omega\in \mathfrak{p}^{-2l+a(\psi)}$, $o \in \mathfrak{p}^{a(\psi)}$. Also, let $\eta\in \widehat{K^{\times}}$ be a multiplicative character with $a(\eta)\leq l$. We define the Kloosterman sum
\begin{align*}\label{eq:Kloostermansumdefinition}
    S_{\eta}(o,\omega;l,\psi) &= \int_{\varpi^{-l}O^{\times}/O}\eta(\varpi^lx)\psi(xo)\psi(x^{-1}\omega)\ \mathrm{d}x,\\ 
\end{align*}
    where $O$ acts on $\varpi^{-l}O^{\times}$ by additive translation. It is important to remark that the measure is the additive measure $\mathrm{d}x$ and not the multiplicative measure $\mathrm{d}^{\times}x$.
    If $\eta = 1$, we simply write $S(o,\omega;l,\psi)$.

   We also denote $Kl_{\eta}(\omega;l,\psi) = \frac{1}{p^{l/2}}S_{\eta}(\varpi_v^{a(\psi)},\omega;l,\psi)$ the \emph{normalized}, twisted by $\eta$, Kloosterman sum. As before, for $\eta=1$, we simply write $Kl(\gamma;l,\psi)$.  We have the obvious isomorphism
    \[\varpi^{-l}O^{\times}/O\to (O/\mathfrak{p}^l)^{\times}; x \mapsto \varpi^lx, \]
    so that 
    \[S_{\eta}(o ,\omega;l,\psi) = \sum_{u\in (O/\mathfrak{p}^l)^{\times}}\eta(u)\psi(\varpi^{-l}uo)\psi(\varpi^lu^{-1}\omega). \]  
    Observe that if $\omega\in \mathfrak{p}^{-2l+a(\psi)+1}$, then we have the following. 
     \begin{equation}\label{eq:trivialitykloosterman}    Kl(\omega;l,\psi) = Kl(0;l,\psi) = 
        \frac{1}{p^{l/2}}\sum_{u\in (O/\mathfrak{p}^l)^{\times}}\psi(\varpi^{-l+a(\psi)}u) 
        = \begin{cases}    0 &l>1\\-p^{-l/2}&l=1.\end{cases}  \end{equation}
  The case $l=1$ is shown by orthogonality of characters, while for $l>1$ let $x\in O/\mathfrak{p}$ act on $(O/\mathfrak{p}^l)^{\times}$ by additive translation: $x\cdot u = u+x\varpi^{l-1}$. Let $U$ be a set of representatives for $(O/\mathfrak{p}^l)^{\times}/(O/\mathfrak{p})$. Then \begin{align*} Kl(0;l,\psi) &= \sum_{u\in U}\sum_{u'\in O/\mathfrak{p}}\psi(\varpi^{-l+s}(u+u'\varpi^{l-1})) \\&= \sum_{u\in U}\psi(\varpi^{-l+s}u)\sum_{u'\in O/\mathfrak{p}}\psi(\varpi^{s-1}u') = 0.  \end{align*}
    Suppose now $l_1,l_2 $ are two non-negative integers. Let $\omega_1\in \mathfrak{p}^{-2l_1+a(\psi)}$, $\omega_2 \in \mathfrak{p}^{-2l_2+a(\psi)}$, $\gamma \in \mathfrak{p}^{-\max(l_1,l_2)+a(\psi)}$ and $\eta\in \widehat{K^{\times}}$ with conductor $a(\eta)\leq \min(l_1,l_2)$. 
    We define the correlation sum
    \[\mathcal{C}_{\eta}(\omega_1,\omega_2,\gamma;l_1,l_2,\psi) = \sum_{a\in O/\mathfrak{p}^{\max(l_1,l_2)}}Kl_{\eta}(a\omega_1;l_1,\psi)\overline{Kl_{\eta}(a\omega_2;l_2,\psi)}\psi(\gamma a). \] 
    As before, we omit $\eta$ from the notation when $\eta =1$. Next, we prove, with similar methods as in \cite[Appendix A]{nelsonholowinsky}, non-trivial bounds for some of these correlation sums.
 \begin{lemma}[Correlation of Kloosterman sums I]\label{lem:correlationkloosterman}   Consider the data $\psi,l_1,l_2,\omega_1,\omega_2,\gamma,\eta$ as above and assume that $l_1=l_2=l$. Assume that $\omega_1,\omega_2\in \varpi^{-2l+a(\psi)}O^{\times}$. Let $x\geq 0$ be so that $\omega_1-\omega_2 \in \varpi^{-2l+a(\psi)+x}O^{\times}$ if $\omega_1-\omega_2 \neq0$, otherwise let $x=\infty$. Similarly, let $y\geq 0$ be such that $\gamma\in \varpi^{-l+a(\psi)+y}O^{\times}$ if $\gamma \neq 0$ and if $\gamma = 0$ let $y=\infty$.  Then we have the following:  \begin{align*}              \mathcal{C}_{\eta}(\omega_1,\omega_2,0;l,\psi) &\ll p^{l/2}p^{\min(l,x)/2}\\
   \mathcal{C}(\omega_1,\omega_2,\gamma;l,\psi) &\ll       p^{l/2}p^{\min(l,x,y)/2}.   
 \end{align*}
 In the latter case, if $l>y>x$, then the correlation is $0$.

 \end{lemma}
    \begin{proof} After changing $\psi$ with $\psi(\varpi^{a(\psi)}\cdot)$ we may assume that $a(\psi)=0$.
    Suppose first $l=1$. Then both $Kl(\cdot \omega_1;l,\psi), {Kl(\cdot \omega_2;l,\psi)}\overline{\psi}(\cdot \gamma)$ are trace functions attached to a (geometric irreducible) $\ell$-adic sheaf, that we denote respectively $[\times \omega_2]^*\mathscr{K}\ell$ and $([\times \omega_2]^*\mathscr{K}\ell)\otimes \psi_{\gamma}$. In this case, the statement follows from the fact that the two sheaves are geometrically isomorphic if and only if $x,y\geq 1$.

    Suppose now that $l >1$, we use the $\mathfrak{p}$-adic stationary phase, explained in \cite[§12.3]{iwaniec2004analytic}. Opening the Kloosterman sum and performing the $a$-sum, we have  \begin{align*}\mathcal{C}_{\eta}(\omega_1,\omega_2,\gamma;l,\psi)&=\sum_{\substack{u_1,u_2\in (O/\mathfrak{p}^{l})^{\times}\\u_1^{-1}\varpi^l\omega_1-u_2^{-1}\varpi^l\omega_2+\gamma \in O}}\eta(u_1)\eta(u_2^{-1})\psi(\varpi^{-l}u_1-\varpi^{-l}u_2)\end{align*}
    The above sum is $0$ if $\gamma+ u_1^{-1}\varpi^{l}\omega_1 \in \mathfrak{p}^{-l+1}$ or $\gamma - u_2^{-1}\varpi^{l}\omega_2 \in \mathfrak{p}^{-l+1}$. So, suppose that these are not the case. Then we get
    \begin{align*}
        \mathcal{C}_{\eta}(\omega_1,\omega_2,\gamma;l,\psi)&=\sum_{u\in (O/\mathfrak{p}^{l})^{\times}}\eta(u\varpi^{-l}\omega_2^{-1}(\varpi^lu^{-1}\omega_1+\gamma))
        \psi(\varpi^{-l}u-\omega_2(\gamma+\varpi^{l}u^{-1}\omega_1)^{-1})\\
        &=\eta(\omega_1\omega_2^{-1})
        \sum_{u\in (O/\mathfrak{p}^l)^{\times}}\eta(1+u(\varpi^{l}\gamma)(\varpi^{2l}\omega_1)^{-1})\psi(\varpi^{-l}g(u)),\end{align*}
        where $g(u) = u\frac{u\varpi^l\gamma+\varpi^{2l}(\omega_1-\omega_2)}{u \varpi^l\gamma+\varpi^{2l}\omega_1}$. 
        We call
 $a = \varpi^{l}\gamma \in \varpi^yO^{\times}$ (here we denote $\mathfrak{p}^{\infty}=\{0\}$), and $b=\varpi^{2l}(\omega_1-\omega_2)\in \varpi^xO^{\times}$ and $d=\varpi^{2l}\omega_1 \in O^{\times}$. Suppose $x,y\geq 1$. Then denoting $g_1(u) = u \frac{a\varpi^{-1}u+b\varpi^{-1}}{au+d}$ we have    \[\sum_{\substack{u\in (O/\mathfrak{p}^l)^{\times}}}\eta(1+uad^{-1})
 \psi(\varpi^{-l}g(u)) = p\sum_{u\in (O/\mathfrak{p}^{l-1})^{\times}}\eta(1+uad^{-1})
 \psi(\varpi^{-(l-1)}g_1(u)).\] By induction, we have  \begin{align*}&\sum_{\substack{u\in (O/\mathfrak{p}^l)^{\times}}}\eta(1+uad^{-1})
 \psi(\varpi^{-l}g(u)) = p^{\min(l,x,y)}\sum_{\substack{u\in (O/\mathfrak{p}^{l^*})^{\times}}}\eta(1+uad^{-1}))
 \psi(\varpi^{-l^*}g^*(u)),\end{align*}where $l^* = l-\min(l,x,y)$, $g^*(u ) = u\frac{a^*u + b^*}{au+d}$, $a^* = a\varpi^{-\min(l,x,y)}$ and $b^*=b\varpi^{-\min(l,x,y})$. Notice that $l^* = 0$ or $a^* \in O^{\times}$ or $b^* \in O^{\times}$. 
 
  Now we focus on our cases of interest. In both cases (corresponding to $\eta=1$ or $a=0$), we see that the absolute value of the correlation sum is bounded by 
 \[p^{\operatorname{min}(l,x,y)}\left\vert \sum_{(u\in O/\mathfrak{p}^{l^*})^\times}\psi(\varpi^{-l^*}g^*(u)) \right\vert,  \]
 so that now we only consider the sum $\sum_{u\in (O/\mathfrak{p}^{l^*})^\times}\psi(\varpi^{-l^*}g^*(u))$. If $l^* = 0$ (that is, $l=\min(l,x,y)$) we only have the trivial bound $\ll p^{l}$. If $l^* = 1$, then we can apply the case $l=1$ we discussed above by rewriting

 \[ \sum_{u\in (O/\mathfrak{p})^{\times}}
 \psi(\varpi^{-1}g^*(u)) = C(\omega_1^*,\omega_2^*,\gamma^*;1,\psi),\]
 where $\omega_i^* = \varpi^{2l-2}\omega_i$ and $\gamma^* = \varpi^{l-1}\gamma$.
 Hence, suppose $l^*\geq 2$. We apply the $\mathfrak{p}$-adic stationary phase, let $\alpha= \lfloor{l^*/2}\rfloor$ \begin{align*} \sum_{\substack{u\in (O/\mathfrak{p}^{l^*})^{\times}\\\varpi^{l}\gamma u\neq-\varpi^{2l}\gamma\mod\mathfrak{p}}}\psi\left(\varpi^{-l^*}g^*(u)\right) &\ll p^{l^*/2}\sum_{\substack{u\in (O/\mathfrak{p}^{\alpha})^{\times}\\\varpi^{l}\gamma u\neq -\varpi^{2l}\gamma \mod \mathfrak{p}\\\substack{(g^*)'(u)}=0 \mod \mathfrak{p}^{\alpha}}}p^{\frac{1}{2}\gamma_{(g^*)''(u)\in \mathfrak{p}}} \end{align*} We have $${g^*}'(u) = \frac{{a^*}^2u^2+2a^*du+b^*d}{(au+d)^2}\quad \text{and} \quad {g^*}''(u) = \frac{2a^*+2c(g^*)'(u)}{au+d}.$$ Notice that if $a^* \in \mathfrak{p}$ then $b^* \in O^{\times}$ and so ${g^*}'(u) \neq 0 \mod \mathfrak{p}^{\alpha}$ for every $u$. Hence, the sum is $0$. If $a^* \in O^{\times}$ and $u\in (O/\mathfrak{p}^{l^*})^{\times}$ is so that ${g^*}'(u)\equiv 0 \mod \mathfrak{p}^{\alpha}$ and $(g^*)''(u) \equiv 0 \mod \mathfrak{p}$, then $2a \equiv 0 \mod \mathfrak{p}$.
        \end{proof}

\begin{lemma}[Correlation of Kloosterman sums II]
     Consider the data $\psi,l_1,l_2,\omega_1,\omega_2,\gamma$ as above and assume that $l_1\neq l_2$. 
     Then we have the following:  \[    \vert{\mathcal{C}(\omega_1,\omega_2,\gamma;l_1,l_2)}\vert \ll p^{\frac{\max(l_1,l_2)}{2}}.\]
     Moreover, it is $0$ if $\gamma \in \mathfrak{p}^{-\max(l_1,l_2)+a(\psi)+1}$.
\end{lemma}

\begin{proof} As before, we may assume without loss of generality that $a(\psi)=0$. We consider only the case $l_1>l_2$, the other being analogous.
Suppose first $\gamma \in \mathfrak{p}^{-2l_1+1}$. Then $Kl(a\omega_1;l_1,\psi) = 0,$ unless $l_1=1$ and $l_2=0$.  In the latter case $Kl(a\omega_1;l_1,\psi)$ is equal to $\frac{-1}{p^{1/2}}$and the correlation sum becomes 
\[-\frac{1}{p^{1/2}}\sum_{a \in O/\mathfrak{p}}\psi(\gamma a), \]
which is $0$ if $\gamma \in \varpi^{-1}O^{\times}$ and $-p^{1/2}$ if $\gamma \in O$.
Assume therefore that $\gamma \in \varpi^{-2l_1}O^{\times}$. 
We open the Kloosterman sums and get:
    \begin{equation}\label{eq:correlation2}\mathcal{C}(\omega_1,\omega_2,\gamma;l_1,l_2)
    =\frac{p^{l_1/2}}{p^{l_2/2}}\sum_{\substack{u_1\in (O/\mathfrak{p}^{l_1})^{\times}\\u_2\in (O/\mathfrak{p}^{l_2})^{\times}\\\varpi^{l_1}u_1^{-1}\omega_1-\varpi^{l_2}u_2^{-1}\omega_2+\gamma \in O}}\psi(\varpi^{-l_1}u_1-\varpi^{-l_2}u_2)
    \end{equation}
    The congruence condition is empty, unless $u_1^{-1} \varpi^{2l_1}\gamma +\varpi^{l_1}\gamma \in \mathfrak{p}^{l_1-l_2}$ and in particular is empty if $\gamma \in \mathfrak{p}^{-l_1+1}$. Assume then that $\gamma \in \varpi^{-l_1}O^{\times}$. Given $u_2$ and $u_1$ so that the sum is non-zero we have $u_1^{-1} = \omega_1^{-1}\varpi^{-l_1}(\varpi^{l_2}u_2^{-1}\omega_2-\gamma)  \mods{\mathfrak{p}^{l_1}}$. In particular, $u_1$ is uniquely defined by $u_2$ and the given data.  Denote $x = u_2\gamma\varpi^{l_1}+\varpi^{l_1+l_2}\omega_2$. Then \[u_2 = (\varpi^{l_1}\gamma)^{-1}(x+\varpi^{l_1-l_2}(\varpi^{2l_2}\omega_2)), \quad u_1 =x^{-1}(x+\varpi^{l_1-l_2}(\varpi^{2l_2}\gamma))(-\varpi^{2l_1}\omega_1)(\varpi^{l_1}\gamma)^{-1}\]The right hand side of \eqref{eq:correlation2} becomes:
    \begin{align*}&\mathcal{C}(\omega_1,\omega_2,\gamma;l_1,l_2)=
    \\& p^{(l_1-l_2)/{2}}\psi(-\gamma^{-1}(\omega_1+\omega_2))\sum_{x \in (O/\mathfrak{p}^{l_2})^{\times}}\psi(-\varpi^{-l_2}(\varpi^{l_1}\gamma)^{-1}x)\psi(-\varpi^{l_2}\omega_2(\varpi^{2l_1}\omega_1)(\varpi^{l_1}\gamma)^{-1})x^{-1})=
    \\ &p^{\frac{l_1}{2}}\psi(-\gamma^{-1}(\omega_1+\omega_2))Kl(\omega_2(\varpi^{2l_1}\omega_1)(\varpi^{l_1}\gamma)^{-2};l_2,\psi)\ll p^{\frac{l_1}{2}}.\end{align*}
\end{proof}

    \subsection{Correlation of Kloosterman sums in smooth boxes} We now consider Kloosterman sums and correlations in the global field $F$. Recall that we denote by $\psi$ the standard character of $\A/F$. Denote $\psi^{\circ} = \psi(\mathfrak{D}^{-1}\cdot)$. This character is unramified, although, in general, not $F$-invariant. For an integral ideal $\mathfrak{d}\subset O_F$, $\omega \in \A$ with $v(\omega)\geq -2v(\mathfrak{d})$ for every $v<\infty$ we define 
    \begin{equation}\label{eq:Klosterman}Kl(\omega;\mathfrak{d};\psi^{\circ}) = \prod_{v}Kl(\omega_v;v(\mathfrak{d}),\psi_v^{\circ}). \end{equation}  
    Let $\mathfrak{d}_1,\mathfrak{d}_2 \subset O_F$ be integral ideals.    Let $[\mathfrak{d}_1,\mathfrak{d}_2]$ be the least common multiple ideal of $\mathfrak{d}_1,\mathfrak{d}_2$. For $\omega_1,\omega_2 \in \A$ and $\gamma \in \A$ so that $v(\omega_i) \geq -2v(\mathfrak{d}_i)$ and $v(\gamma) \geq -v([\mathfrak{d}_1,\mathfrak{d}_2])$ for all $v< \infty$. For any integral ideal $\mathfrak{d}$ let $R_{\mathfrak{d}} = \{v< \infty \ | \  v(\mathfrak{d})>0\}$. We define the correlation sums
    \begin{equation*}\label{eq:correlationkkloosterman}\mathcal{C}(\omega_1,\omega_2,\gamma; \mathfrak{d}_1,\mathfrak{d}_2) = \prod_{v\in R_{[\mathfrak{d}_1,\mathfrak{d}_2]}}\mathcal{C}({\omega_1}_v,{\omega_2}_v,\gamma_v;v(\mathfrak{d}_1),v(\mathfrak{d}_2)),\end{equation*} 
    where the Local correlations are defined and studied in the previous section. Over a number field we can study correlations of Kloosterman sums along smooth boxes:
    
    \begin{lemma}\label{lem:4.3}Let $\mathfrak{d}_1,\mathfrak{d}_2$ be integral ideals so that $N(\mathfrak{d}_i)\leq N(\mathfrak{q})^A$ for some $A>0$, $X \in \R_{>0}^{S_{\infty}} \subset \A^{\times}$, $\omega_1,\omega_2 \in \A^{\times}$.
    Suppose that $V = \prod_{v}V_v \in C_c^{\infty}(\A)$ satisfies the following properties:\begin{itemize}
        \item for all $v<\infty$ the smooth function $V_v$ is of the form $V_v=\mathbb{1}_{\mathfrak{p}_v^{j_v}}$ for some $j_v \in \Z$ and only finitely many of the $j_v$'s are non-zero. If $v\in R_{[\mathfrak{d}_1,\mathfrak{d}_2]}$, then $$j_v\geq -2v([\mathfrak{d}_1,\mathfrak{d}_2])-\min(v(\omega_1),v(\omega_2)).$$
       
        \item For all $v\in S_{\infty}$ real it holds that $\vert{x^m\partial^n_xV(x)}\vert \ll_{m,n} 1$,
        \item For all $v\in S_{\infty}$ complex it holds that 
        $$\vert{\real(z)^{m_1}\imag(z)^{n_1} \partial_{\real(z)}^{m_1}\partial_{\imag(z)}^{n_2}V(z)}\vert \ll_{m_1,m_2,n_1,n_2} 1. $$
        \end{itemize}
        Then we have \begin{align*} \sum_{\gamma \in F}V(\gamma X^{-1})Kl(\gamma \omega_1;\mathfrak{d}_1,\psi^{\circ})\overline{Kl(\gamma\omega_2;\mathfrak{d}_2,\psi^{\circ}}) \ll \frac{\widehat{V}(0)\vert{X}\vert  }{N([\mathfrak{d}_1,\mathfrak{d}_2])^{1/2}N(J')} N\left((\mathfrak{d}_1,\mathfrak{d}_2,\Omega) \right)^{1/2} \\+   N(\mathfrak{q})^{\epsilon}N([\mathfrak{d}_1,\mathfrak{d}_2])^{1/2}, \end{align*}
        where $J' = \prod_{v<\infty}\mathfrak{p}_v^{j_v}\in \A^{\times}$ and $\Omega=J'(\omega_1-\omega_2)(\mathfrak{d}_1,\mathfrak{d}_2)^2 \in \A.$\end{lemma}
    \begin{proof}  We apply Poisson summation formula to get 
    \begin{align*}&\sum_{\gamma \in F}V(\gamma X^{-1})Kl(\gamma\omega_1;\mathfrak{d}_1,\psi^{\circ})\overline{Kl(\gamma\omega_2;\mathfrak{d}_2,\psi^{\circ})} =
    \\& \frac{\vert{X}\vert }{N([\mathfrak{d}_1,\mathfrak{d}_2]J')} \sum_{\gamma} \widetilde{V}(\gamma)\mathcal{C}(\omega_1 J',\omega_2 J',\gamma;\mathfrak{d}_1,\mathfrak{d}_2),    
    \end{align*}
    where \begin{align*}
        \widetilde{V}_v = \begin{cases}
            \widehat{V_v}(\cdot X_v) & v \in S_{\infty}\\ \vert{\mathfrak{D}}\vert_v^{1/2}\mathbb{1}_{[\mathfrak{d}_1,\mathfrak{d}_2]^{-1}\mathfrak{p}_v^{-j_v}\mathfrak{D}_v^{-1}}&v<\infty.
        \end{cases}
    \end{align*}
    The $0$- contribution in the latter sum is bounded by
    \begin{align*}\frac{\widehat{V}(0)\vert{X}\vert }{N([\mathfrak{d}_1,\mathfrak{d}_2]J')}\prod_{v\in R_{[\mathfrak{d}_1,\mathfrak{d}_2]}}\vert{[\mathfrak{d}_1,\mathfrak{d}_2] (\mathfrak{d}_1,\mathfrak{d}_2,J'(\omega_1-\omega_2))}\vert_v^{-1/2} =\\  \widehat{V}(0)\vert{X}\vert N(J')^{-1}\frac{N(\mathfrak{d}_1,\mathfrak{d}_2,J'(\omega_1-\omega_2))^{1/2}}{N([\mathfrak{d}_1,\mathfrak{d}_2])^{1/2}}. \end{align*}
    whereas the non-zero contribution is, ignoring the cases where the correlation might be $0$, is bounded by 
    \begin{align*}&\ll \frac{\vert{X}\vert  }{N([\mathfrak{d}_1,\mathfrak{d}_2]J')}\sum_{\gamma \neq 0}\vert{\tilde{V}(\gamma)}\vert N(\mathfrak{d}_1,\mathfrak{d}_2,J'(\omega_1-\omega_2),\gamma)^{1/2}N([\mathfrak{d}_1,\mathfrak{d}_2])^{1/2} \\ &\ll N(\mathfrak{q})^{\epsilon}N([\mathfrak{d}_1,\mathfrak{d}_2])^{1/2}.  \end{align*}
    \end{proof}

\section{Automorphic representations and integral representation of $L$-functions}\label{sec:integralrep}
In this section, we recall some facts about automorphic representations and integral representation of $L$-functions. The reference for this section is Cogdell's lecture notes \cite{cogdell-lfunctions}. 

Let $\pi$ be a unitary cuspidal automorphic representation of $\gelle_3(\A)$ (realized in $L_{\operatorname{cusp}}^2(\gelle_3(F)\setminus\gelle_3(\A),\omega)$, where $\omega\in \widehat{F^{\times}\setminus\A^{\times}}$ is the central character) and let \[\pi^{\infty} \to \mathcal{W}(\pi,\psi);\ \varphi\mapsto \left(W_{\varphi}\colon \gelle_3(\A) \in g\mapsto \int_{U_3(F)\setminus U_3(\A)}\varphi(ug)\overline{\psi}(u) \ \mathrm{d}u\right)\]
denote the global Whittaker model of $\pi$ with respect to $\psi$. 
Each $\varphi \in \pi^{\infty}$ has a Fourier-Whittaker expansion 
\[ \varphi(g) = \sum_{\gamma \in \gelle_2(F)\setminus\gelle_2(\A)}W_{\varphi}\left(\begin{pmatrix}
    \gamma&\\&1
\end{pmatrix}g\right).\]
The representation $\pi$ splits as a restricted tensor product $\pi \simeq \otimes_v' \pi_v$, where each $\pi_v$ is a unitary irreducible generic representation of $\gelle_3(F_v)$ and for all but finitely many $v$ the representation $\pi_v$ contains a vector (necessarily unique up to scalar multiplication by \cite{Shalika}) fixed by $\gelle_3(O_{F_v})$. 
The global Whittaker model splits accordingly as a restricted tensor product of local Whittaker models of $\pi_v$, that is, each $W\in \mathcal{W}(\pi,\psi)$ can be written as a finite sum of functions $\prod_vW_v \colon \gelle_3(\A) \ni g \mapsto \prod_vW_v(g_v)$, where $W_v$ lies in the (unique) Whittaker model of $\pi_v$,  denoted by $\mathcal{W}(\pi_v,\psi_v)$, and for all but finitely many places $v$ the function $W_v$ is right $\gelle_3(O_{F_v})$-invariant and $W_v(I)=1$.

Denote by $\iota$ the inverse transposition on $\gelle_3$, that is, $\iota(g) = (g^{-1})^t$. We will also denote by $\iota$ the induced involution on the space of functions of $\gelle_3$. If $V_{\pi} \subset L^2(\gelle_3(F)\setminus\gelle_3(\A),\omega)$ is (the) a realization of $\pi$, then $V_{\pi}^{\iota} = \{\varphi^{\iota},\ \varphi \in V_{\pi}\}$ is a realization of the contragredient representation $\tilde{\pi}(g) = \pi(g^{\iota})$. For $\pi$ unitary, the representation $\tilde{\pi}$ is isomorphic to the unitary dual representation of $\pi$.

Consider the following projections:
\begin{align}
    \mathbb{P}\varphi(g)&= \vert{\det(g)}\vert_{\A}^{-\frac{1}{2}}\int_{F\setminus \A}\int_{F\setminus\A}\varphi\left(\begin{pmatrix}
    1&&u_{13}\\&1&u_{23}\\&&1
\end{pmatrix}g\right)\overline{\psi}(u_{23})\ \mathrm{d}u_{23}\mathrm{d}u_{13} \label{eq:definitionPphi}\\
\tilde{\mathbb{P}}\varphi&= \iota \circ \mathbb{P}\circ \iota. \nonumber
\end{align}

The function $\mathbb{P}\varphi$ is $a_1(F^\times)\leqslant \gelle_2(F)$ invariant. Similarly as $\varphi$, also $\mathbb{P}\varphi$ admits a Fourier-Whittaker expansion in terms of the global Whittaker function of $\varphi$.
\begin{equation}\label{eq:Fourierprojection}\mathbb{P}\varphi(g) = \sum_{\gamma \in F^{\times}}W_{\varphi}(a_1(\gamma)g), \qquad g\in \gelle_3(\A). \end{equation}
For any character $\eta \in \widehat{F^{\times}\setminus \A^{\times}}$ and any $\varphi \in \pi^{\infty}$, we define the complex functions \begin{align*}I(\varphi,\eta,s) &= \int_{F^{\times}\setminus\A^{\times}} \mathbb{P}\varphi(a_1(y))\eta(y)\vert{y}\vert^{s-\frac{1}{2}}\ \mathrm{d}^{\times}y,\\ \tilde{I}(\varphi,\eta,s)&= \int_{F^{\times}\setminus\A^{\times}}\tilde{\mathbb{P}}\varphi(a_1(y))\eta(y)\vert{y}\vert^{s-\frac{1}{2}}\ \mathrm{d}^{\times}y. \end{align*}
Since $\varphi$ is a smooth cusp form the integrals are absolutely convergent and uniform convergent on compacta, implying that $I(\varphi,\eta,s)$ and $\tilde{I}(\varphi,\eta,s)$ are entire functions. Also, they are connected by a functional equation:
\begin{equation}\label{eq:fctequations}I(\varphi,\eta,s) = \tilde{I}(\varphi^{\iota},\overline{\eta},1-s).\end{equation}
We can understand the period integral using the Fourier-Whittaker expansion \eqref{eq:Fourierprojection} and it becomes
\begin{equation*}
    I(\varphi,\eta,s) = \int_{\A^{\times}} W_{\varphi}(a_1(y))\eta(y)\vert{y}\vert^{s-1}\ \mathrm{d}^{\times}y; 
\end{equation*}
the right hand side converges for $\real(s)>1$. If, in addition, $W_{\varphi} = \prod_{v}W^{\varphi}_{v}$, $W_{v}^{\varphi} \in \mathcal{W}(\pi_v,\psi_v)$, then the integral is split as an Euler product ($\real(s)>1$): \begin{equation*}
    I(\varphi,\eta,s) = \prod_{v} \int_{F_v^{\times}}W_{v}^{\varphi}(a_1(y))\eta_v(y)\vert{y}\vert_v^{s-1}\mathrm{d}^{\times}y.
\end{equation*}
For $W \in \mathcal{W}(\pi_v,\psi_v)$ we define the local integral 
\begin{align*}\Psi(W,\eta_v,s) &= \int_{F_v^{\times}}W(a_1(y))\eta_v(y)\vert{y}\vert_v^{s-1}\ \mathrm{d}^{\times}y.
\end{align*}
Let $v<\infty$. Then the local integrals above generate a $\CC[p_v^{\pm s}]$-fractional ideal $\mathcal{I}(\pi_v\times\eta_v)$ in $\CC(p_v^{-s})$. By the properties of the Whittaker model $\mathcal{W}(\pi_v,\psi_v)$, this ideal contains the constant function $1$, and so it has a generator $L(\pi_v\times \eta_v,s) = P(p_v^{-s})^{-1}$, where $P \in \CC[X]$. This is the local $v$-$L$ factor of $\pi \times \eta$. The global arithmetic $L$-function is defined as 
\[ L(\pi\times\eta,s)=\prod_{v<\infty}L(\pi_v\times\eta_v,s), \qquad \real(s)>1\]
Let $\varphi \in \pi$ and suppose that $W_{\varphi} = \prod_{v}W_{v}^{\varphi}$, then the global integral $I(\varphi,\eta,s)$ and the $L$-functions $L(\pi\times\eta,s)$ are \emph{equal up to finitely many factors}, that is there exists a finite set of places $S$ of $F$ so that $S_{\infty}\cap S = \varnothing$ and  
\[I(\varphi,\eta,s) = L(\pi \times \eta,s)\prod_{v\in S}\frac{\Psi(W_v^{\varphi},\eta_v,s)}{L(\pi_v\times\eta_v,s)}\prod_{v\in S_{\infty}}\Psi(W_v^{\varphi},\eta_v,s). \] and note that by construction $\frac{\Psi(W_v^{\varphi},\eta_v,s)}{L(\pi_v\times\eta_v,s)}$, $v\in S$, is an entire function.

\section{Voronoi summation formula}\label{sec:voronoi}
The reference is \cite{IchinoTemplier}. See also \cite{2018VoronoiSF} for a more complete formula.

The formula obtained in \cite{IchinoTemplier} is deduced from an equality between linear functionals on the space of automorphic forms for $\gelle_n$, let $\mathcal{P}\varphi = \mathbb{P}\varphi(1)$. Also, they define 
\begin{equation}\label{def:definitionptilda}\tilde{P}\varphi = \int_{\A}\int_{F\setminus \A}\int_{F\setminus\A}\varphi\left(\begin{pmatrix}
    1&&u_{13}\\&1&u_{23}\\&&1
\end{pmatrix}\begin{pmatrix}
    1&\\x&1\\&&1
\end{pmatrix}\sigma_{23}\right)\psi(u_{23})\ \mathrm{d}u_{23}\mathrm{d}u_{13}\ \mathrm{d}x,\end{equation}
then in \emph{loc. cit.} it is shown that $P\varphi =(\tilde{P}\circ \iota) \varphi= \tilde{P}\varphi^{\iota}$. Inserting \eqref{eq:Fourierprojection} in the definition of $\mathcal{P}$ and in \eqref{def:definitionptilda} one obtains the following equality:
\begin{equation}\label{eq:lefthandsidevoronoi}
    \sum_{\gamma \in F^{\times}}W_{\varphi}(a_1(\gamma)) = \int_{\A}\sum_{\gamma \in F^{\times}}W_{\varphi}^{\iota}\left(a_1(\gamma)\begin{pmatrix}
        1&&\\x&1\\&&1
    \end{pmatrix}\sigma_{23}\right) \ \mathrm{d}x,
\end{equation}
where $W_{\varphi}^{\iota}(g) = W_{\varphi}(\sigma_{13}g^\iota)$. Note that the right-hand side is expressed a priori in terms of $W_{\varphi}^{\overline{\psi}}$, that is, the global Whittaker function of $\varphi^{\iota}$ with respect to $\overline{\psi}$, rather than $\psi$. One can easily compute that by the left $\gelle_3(F)$ invariance of $\varphi$ we have $W_{\varphi}^{\iota}(g) = W_{\varphi^{\iota}}^{\overline{\psi}}(g)$.
Now, different vector choices $\varphi$ will create different formulae. For example, for $\zeta \in \A$, $J_1,J_2\in \A^{\times}$, we have 
\begin{equation*}\mathcal{P}(\varphi^{u_{12}(\zeta)a(J_1,J_2)}) = \sum_{\gamma \in F^{\times}}W_{\varphi}(a(J_1,\gamma J_2))\psi(\gamma\zeta), \end{equation*}
where $\varphi^{g}$ denotes the vector obtained by right translating $\varphi$ by the group element $g$. 
In some cases, one can compute the right-hand side switching integral and sum. If $\varphi \in \pi^{\infty}$ and its global Whittaker function splits accordingly $W_{\varphi} = \prod_{v}W_v^{\varphi}$, then to compute the right-hand side of \eqref{eq:lefthandsidevoronoi} the task is to compute the functions defined in $F_v^{\times}$ by the local integral 
\begin{equation*}
   B_{\pi_v,\zeta,J_1,J_2}[W_v^{\varphi}](y)= \int_{F_v}({W_{v}^{\varphi}})^{\iota}\left(\begin{pmatrix}
        y&&\\x&1\\&&1
    \end{pmatrix}\sigma_{23}\begin{pmatrix}
        (J_1J_2)^{-1}&&\\-\zeta (J_1J_2)^{-1}&J_1^{-1}&\\&&1
    \end{pmatrix}\right) \ \mathrm{d}x,
\end{equation*}
for all places $v$, or over the archimedean places to give enough satisfactory analytic properties of those functions.
In \cite{IchinoTemplier} the above computation is done for the following cases: \begin{itemize}
    \item $v<\infty$, $W_v^{\varphi}$ is $\gelle_3(O_{F_v})$-invariant and $\vert{\zeta}\vert_v \leq 1$ and $\psi_v$ is unramified.
    \item $v< \infty$, $W_v^{\varphi}$ is $\gelle_3(O_{F_v})$-invariant and $\vert{\zeta}\vert_v >1$ and $\psi_v$ is unramified.  
\end{itemize}

 Let $W_v^{\circ}\in \mathcal{W}(\pi_v,\psi_v)$ be the newform (see Remark \ref{rmk:newform}) and let $W^{\circ}_{\fin}=\prod_{v<\infty}W^{\circ}_v$. For each $v\in S_{\infty}$ suppose that $W_v \in \mathcal{W}(\pi_v,\psi_v)$ so that $F_v^{\times}\ni y \mapsto W_v(a_1(y)) \in C_c^{\infty}(F_v^{\times})$. In particular, $W = W_{\fin}^{\circ} \times \prod_{v\in S_{\infty}}W_v \in \mathcal{W}(\pi,\psi)$ in the global Whittaker model of $\pi$. 
 The following version of the Ichino-Templier Voronoi summation formula can be found in \cite[Proposition 4.2]{Qi2024} and deals with the places where $\psi_v$ is ramified.

\begin{proposition}Let $S= S_{\infty}\cup\{v<\infty,\ \pi_v \text{ is ramified}\}$. Let $J \in (\A^{\times})^S$ and $\zeta \in \A^{S}$. Define $R=\{v<\infty  \ |\ \vert{\zeta J_2^{-1}}\vert_v>1\}$.
Then the following equality holds
\begin{align*}
    &\sum_{\gamma \in F^{\times}}W^{\circ}_{\operatorname{fin}}(a(\mathfrak{D}^{-1},\gamma J))\prod_{v \in S_{\infty}}W_v(a_1(\gamma)) \psi(\gamma\zeta) =\\ & \frac{\vert{\zeta}\vert_{R}\chi^\pi_{R}(J(\mathfrak{D}\zeta)^{-1}) \vert{J}\vert^{S\cup R} }{{\chi^{\pi}}^S(\mathfrak{D})(\vert \mathfrak{D}\vert ^S)^{\frac{1}{2}}}\sum_{\gamma \in F^{\times}}K_{R}(\gamma,\zeta, \mathfrak{D}^{-1},J, (W^{\circ}_{R})^{\iota})\prod_{v\in S, v<\infty}B_{\pi_v}[W_v^{\circ}](\gamma)\\&\hspace{6.1cm}\times(W^{\circ R\cup S})^{\iota}(a_1(\mathfrak{D}^{-1},\gamma\mathfrak{D}J^{-1}))\prod_{v\in S_{\infty}}B_{\pi_v}[W_v](\gamma),
\end{align*}
where $K_{R}(\gamma,\zeta,\mathfrak{D}^{-1},J,(W_{R}^{\circ})^{\iota})=\prod_{v\in R}K_v(\gamma,\zeta,\mathfrak{D}^{-1},J,(W_v^{\circ})^{\iota})$
and \begin{align*}
    &K_v(\gamma,\zeta,\mathfrak{D}^{-1},J,(W_v^{\circ})^{\iota}) =\\& 
   \sum_{l=0}^{-v(\zeta(\mathfrak{D}J_1J_2)^{-1})}(W_v^{\circ})^{\iota}(a(\varpi_v^{-l}\mathfrak{D}^{-1}J\zeta^{-1},\gamma\varpi_v^{2l}\mathfrak{D}\zeta ^{-1}))S(\mathfrak{D}^{-1}_v,-\gamma(\mathfrak{D}\zeta^{-1})_v;l,\psi_v)
\end{align*}
in both cases the Kloosterman sum $S(\cdot,\cdot,l;\psi_v)$ (see Section \ref{sec:gausskloosterman}) for $l=0$ must be understood as $1$.
\end{proposition}

Our case of interest is $J=\mathfrak{D}^{-1}I$, for $I \in \mathcal{S}_{\cl}$; and also we want to state the result in terms of the newforms $U_v^{\circ}\in \mathcal{W}(\tilde{\pi}_v,\overline{\psi_v})$ instead of $(W_v^{\circ})^{\iota}$ for $v<\infty$:
\begin{corollary}\label{thm:voronoi1}Let $S= S_{\infty}\cup\{v<\infty,\ \pi_v \text{ is ramified}\}$ and assume that for all $v \in S$ we have $v(\mathfrak{D})=0$. Let $I \in (\A^{\times})^S$ and $\zeta\in \A^S$. Define $R=\{v<\infty \ |\ \vert{\zeta \mathfrak{D}I^{-1}\vert_v>1}\}$. Then the following summation formula holds
    \begin{align*}&\sum_{\gamma \in F^{\times}}W^{\circ}_{\operatorname{fin}}(a(\mathfrak{D}^{-1},\gamma \mathfrak{D}^{-1}I))\prod_{v \in S_{\infty}}W_v(a_1(\gamma)) \psi(\gamma\zeta)=\\&\epsilon(\pi)\frac{\vert{\zeta I\mathfrak{D}^{-1}}\vert_{R}N(\mathfrak{D})^{3/2}C(\pi)^{\frac{1}{2}}}{N(I)}\frac{\chi_{\pi_{R}}(I\mathfrak{D}^{-2}\zeta^{-1})\chi_{\pi}^R(\mathfrak{D}^2)}{{\chi^\pi}^S(\mathfrak{D})}\sum_{\gamma \in F^{\times}}K_{R}(\gamma,\zeta, \mathfrak{D}^{-1},\mathfrak{D}^{-1}I, U^{\circ}_R)\\&\qquad \times U^{\circ R}(a(\mathfrak{D}^{-1},\gamma \mathfrak{D}^2\mathfrak{C}(\pi)I^{-1}))\prod_{v\in S_{\infty}}B_{\pi_v}[W_v](\gamma)
    \end{align*}
\end{corollary}
\begin{proof}
    If $\pi_v$ is unramified then the elements $(W^{\circ}_v)^{\iota}, U_v^{\circ}\in \mathcal{W}(\tilde{\pi}_v,\overline{\psi_v})$ are related by $$(W^{\circ}_v)^{\iota} = \overline{\chi^\pi_{v}}(\mathfrak{D}^{-2})U_v^{\circ}.$$ 
    If $\pi_v$ is ramified (by our assumption, then $\psi_v$ is unramified), then Proposition \ref{prop:computation-v-adic-bessel-transform} shows that \[B_{\pi_v}[W_v^{\circ}](\gamma) = \epsilon(\pi_v,\psi_v)q_v^{\frac{a(\pi_v)}{2}}U_v^{\circ}(a_1(\gamma\varpi_v^{a(\pi_v)})).\]
\end{proof}

\subsection{Archimedean Bessel transforms}Let $F$ be $\R$ or $\CC$. Let $V\in C_c^{\infty}(F^{\times})$. Let $\rho$ denote an admissible irreducible unitary representation of $\gelle_3(F)$.  
Then the $\rho$-Bessel transform $B_{\rho}[V]$ of $V$ is defined by the following uniqueness property (\cite{IchinoTemplier}): 
\begin{equation*}
     \int_{F^{\times}}B_{\rho}[{V}](y)\eta(y)^{-1}\vert{y}\vert_v^{s-1}\mathrm{d}^{\times}y =  \gamma(1-s,\rho\times \eta,\psi)\int_{F^{\times}}V(y)\eta(y)\vert{y}\vert_v^{-s}\ \mathrm{d}^{\times}y,
\end{equation*}
for all $\eta \in \widehat{F^{\times}}$ and all $s\in \CC$ with real part sufficiently big.
In the case where $V$ is $C_c^{\infty}(F^{\times})$ (or say, in the Schwartz space of $F^{\times}$) we can explicitly compute $B_{\rho}[V]$ using Mellin inversion.
For $F=\R$ any character of $\R^\times$ is of the form $y \in \R^{\times} \mapsto \vert{y}\vert^{it} \operatorname{sgn}(y)^{j}  $, therefore if $V \in C_c^{\infty}(\R^{\times})$ we have that 
\[B_{\rho}[V](y) = \sum_{j\in\{0,1\}} \operatorname{sgn}(y)^j\int_{(\sigma)}\gamma(1-s,\rho\operatorname{sgn}^{j},\psi)\vert{y}\vert^{1-s}\int_{\R^{\times}}V(y)\operatorname{sgn}(y)^j\vert{y}\vert^{-s}\ \mathrm{d}s,\]
for $\sigma=\real(s)$ big enough.

Let now $F=\CC$. Any character $\eta$ of $\CC^{\times}$ can be written as $\eta(re^{i\theta}) = r^{it}e^{ik\theta}$ for unique $t\in \R$ and $k\in \Z$. 

Then we have for $\sigma = \real(s)$ big enough:
\begin{equation*}
    B_{\rho}[V](re^{i\theta}) = \sum_{k\in \Z}e^{ik\theta}\int_{(\sigma)}\gamma(1-s,\rho_k,\psi)r^{2-2s}\mathcal{M}(V,k)(-2s) \ \mathrm{d}s,
\end{equation*}
where $\mathcal{M}(V,k)(s) = \frac{1}{2\pi}\int_{0}^{\infty}\int_{-\pi}^{\pi}V(re^{it})r^{s}e^{-ikt}\ \mathrm{d}^{\times}r \mathrm{d}\theta$ and if $\underline{\alpha} = \{\alpha_1,\alpha_2,\alpha_3\}$ denotes the multiset of archimedean paramereters of $\rho$, then $\rho_k$ is the representation with Langlands parameters $\underline{\alpha + \frac{\vert{k}\vert}{2}} = \{\alpha+\frac{\vert{k}\vert}{2},\alpha \in \underline{\alpha}\}.$

The $\gamma$ factors are of the following shape: we denote $$\Gamma_{\R}(s) = \pi^{-s/2}\Gamma(s/2)\quad \text{and}\quad \Gamma_{\CC}(s) = 2\pi^{-s}\Gamma(s).$$ Let $\underline{\alpha}=\{\alpha_1,\alpha_2,\alpha_3\}$ denote the multiset of archimedean parameters of $\rho$. Then we have
\begin{align*}\gamma(1-s,\rho,\psi)=\varepsilon(\rho,\psi)\prod_{\alpha \in \underline{\alpha}}\frac{\Gamma_{F}(s+\overline{\alpha})}{\Gamma_{F}(1-s+\alpha)},\end{align*}
where $\varepsilon(\rho,\psi)$ has absolute value $1$. For $F=\R$ and $j\in \{0,1\}$
the Langlands parameters of $\operatorname{sgn}^\epsilon\rho$ are $\underline{\alpha+j} =\{\alpha+j, \alpha \in \underline{\alpha}\}$.
In particular, for each $k\in \Z$, the function $\gamma(1-s,\rho_k,\psi)$ (respectively, $\gamma(1-s,\rho\operatorname{sgn},\psi)$) is holomorphic for $\real(s)>\max_{\alpha \in \underline{\alpha}}\vert{\real(\alpha)}\vert$.

In the real case by classical techniques of countour shifting, one can show, see \cite[Prosposition 3.5]{kowalski2014fourier} for instance, that \begin{align}\label{eq:decayingpropertiesbessel}     &B_{\rho}[V](y) \ll_{A,V,\rho} (1+\vert{y}\vert)^{-A}\qquad \forall A>0, \nonumber\\ \\&\text{and for } 0< \vert{y}\vert \leq 1\nonumber\\ \nonumber\\&B_{\rho}[V](y) \ll_{\epsilon} \max_{\alpha \in \underline{\alpha}}\vert{y}\vert^{1-\real(\alpha)-\epsilon}\nonumber\end{align}

A consequence of the bounds towards the Ramanujan conjecture (see \cite{ramanujannumberfield}) is that $\vert{\real(\alpha)}\vert < 5/14$ and so for $0<\vert{y}\vert \leq 1$:
\[B_{\rho}[V](y) \ll \vert{y}\vert^{1-Q}\]
for any fixed $Q>5/14$.

We would like to have a family of functions $\omega$ such that the bound \eqref{eq:decayingpropertiesbessel} can be stated uniformly for this family in a similar way as in \cite{nelsonholowinsky}. Also, in the complex case we need to take into account that the sum over $k\in \Z$ is infinite.

\begin{definition}\label{def:inert}
    We call a function $V\in C_c^{\infty}(\mathbb{C}^{\times})$ \emph{inert} if \[V(re^{i\theta}) \neq 0 \Rightarrow r \asymp 1\] and\begin{equation}\label{eq:inert}(r\partial_r)^n\partial_{\theta}^mV(re^{i\theta}) \ll_{n,m} 1.\end{equation}
A real function $V\in C_c^{\infty}(\R^{\times})$ is inert if 
\[V(x) \neq 0 \Rightarrow x \asymp 1\]
and \begin{equation*}
    \vert{(x\partial_x^nV(x)}\vert \ll_n 1
\end{equation*}
\end{definition}

\begin{lemma}
If $V \in C_c^{\infty}(\CC^{\times})$ is inert, then for any $m,n \in \Z_{\geq 0}$ and $A>0$ we have
\begin{equation}\label{eq:decaybesselforinert}(r\partial_r)^n\partial_{\theta}^mB_{\rho}[V](re^{i\theta}) \ll_{m,n,A}  \min(r^{2(1-Q)},r^{-A}).\end{equation}
If $V\in C_c^{\infty}(\R^{\times})$ is inert, then for any $m,n \in \Z_{\geq0 }$ and $A>0$ we have 
\begin{equation}\label{eq:decaybesselinert2}
    (x\partial_x)^nV(x) \ll_{n,A} \min(\vert x\vert^{1-Q},\vert x \vert^{-A}).
\end{equation}
\end{lemma}

\begin{proof}We prove only the complex case, the real one is easier. Since the real part $k$ also varies, we need to be careful in the use of the Stirling approximation formula: it says that for fixed $\sigma$, $\gamma$ will grow at most polynomially as  $\operatorname{im}(s)\to \infty$. The bound is in fact also polynomial in $k$ (once $\real(s)$ is fixed) as the following shows: By Stirling's approximation formula we have 
    \begin{align*}
    \frac{\Gamma_{\CC}\left(s+\overline{\alpha}+\frac{\vert{k}\vert}{2}\right)}{\Gamma_{\CC}\left(1-s+\alpha+\frac{\vert{k}\vert}{2}\right)} & \asymp e^{-2\real(s)}\pi^{1-2s}w_{s,\alpha,k}\left\vert{\frac{s+\alpha+\frac{\vert{k}\vert}{2}}{1-s+\alpha+\frac{\vert{k}\vert}{2}}}\right\vert^{\frac{\vert{k}\vert}{2}}\\&\times\left\vert{\left(1-s+\alpha+\frac{\vert{k}\vert}{2}\right)\left(s+\overline{\alpha}+\frac{\vert{k}\vert}{2}\right)}\right\vert^{\real(s)}\\&\times e^{-\operatorname{im}(s+\overline{\alpha})arg(s,\alpha,k)},
    \end{align*}
    where $w_{s,\alpha,k}=\left(\frac{1-s+\alpha+\frac{\vert{k}\vert}{2}}{s+\overline{\alpha}+\frac{\vert{k}\vert}{2}}\right)^{\frac{1}{2}-\real(\alpha)}$ and $arg(s,\alpha,k) = \operatorname{arg}(s+\overline{\alpha}+\frac{\vert{k}\vert}{2})+\operatorname{arg}(1-s+\alpha+\frac{\vert{k}\vert}{2})$.
  
    We need to make sure $arg(s,\alpha,k)$ is very small whenever $\imag(s)$ or $\vert{k}\vert$ is big. Assuming $\vert{k}\vert > 2(1-\real(s)+\real(\alpha))$, which eventually happens as the right-hand side can be fixed, we have \begin{align}\label{eq:arg}&arg(s,\alpha,k) =\nonumber\\& \arctan\left(\frac{\imag(s-\alpha)}{\real(s)+\real(\alpha)+\frac{\vert{k}\vert}{2}}\right)-\arctan\left(\frac{\imag(s-\alpha)}{1-\real(s)+\real(\alpha)+\frac{\vert{k}\vert}{2}}\right).\end{align}
    \begin{itemize}
        \item In the range $\vert{\imag(s-\alpha)}\vert> \vert{k}\vert$,  we can write \eqref{eq:arg} as 
        \begin{align*}&\sum_{n=0}^{\infty}\frac{(-1)^n}{(2n+1)}\left(\frac{(1-\real(s-\alpha)+\frac{\vert{k}\vert}{2})^{2n+1}-(\real(s+\alpha)+\frac{\vert{k}\vert}{2})^{2n+1}}{(\imag(s-\alpha))^{2n+1}}\right) \\ 
        &\ll \frac{1}{\vert{\imag(s-\alpha)}\vert}\sum_{n=0}^{\infty}\left\vert{\frac{k}{\imag(s-\alpha)}}\right\vert^{2n} = O(\frac{1}{\vert{\imag(s)}\vert}) \end{align*}
        \item The range $\vert{\imag(s-\alpha)}\vert<\frac{\vert{k}\vert}{2}$, we can write \eqref{eq:arg} as 
        \begin{align*}
            &\sum_{n=0}^{\infty}\frac{(-1)^n}{(2n+1)}\left(\frac{(\imag(s-\alpha))^{2n+1}}{(1-\real(s-\alpha)+\frac{\vert{k}\vert}{2})^{2n+1}}-\frac{(\imag(s-\alpha))^{2n+1}}{(\real(s+\alpha)+\frac{\vert{k}\vert}{2})^{2n+1}}\right)\\
            &\ll \frac{1}{\vert{k}\vert}\sum_{n=0}^{\infty}\left\vert{\frac{\imag(s-\alpha)}{k/2}}\right\vert^{2n+1} = O(\frac{1}{\vert{k}\vert}) 
        \end{align*}
        \item Finally in the range $\frac{\vert{k}\vert}{2}\leq \vert{\imag(s-\alpha)}\vert\leq \vert{k}\vert$ we use the fact that $\arctan$ is Lipschitz continuous to say that \eqref{eq:arg} is less or equal to
        \begin{align*}
            &\ll\left\vert{\frac{\imag(s-\alpha)}{\real(s)+\real(\alpha)+\frac{\vert{k}\vert}{2}}-\frac{\imag(s-\alpha)}{1-\real(s)+\real(\alpha)+\frac{\vert{k}\vert}{2}}}\right\vert
            \\& \ll \left\vert{\frac{\imag(s-\alpha)}{k^2}}\right\vert \ll \frac{1}{\vert{\imag(s-\alpha)}\vert}
        \end{align*}
         \end{itemize}
        In any case we see that $e^{-\imag(s-\alpha)arg(s,\alpha,k)} = O(1)$. We deduce that for fixed $\real(s)$ there exists $C>0$ (possibly depending on $\real(s)$ and the set of archimedean parameters of $\pi$) so that 
        \[\gamma(1-s,\rho_k,\psi) \ll \left((1+\vert{k}\vert)(1+\vert{t}\vert)\right)^C. \]
        Shifting the countour of integration to $Q$ we have 
        \begin{align*}
            &B_{\rho}[V](re^{i\theta}) \\ &\ll r^{2-2Q}\sum_{k\in \Z}\int_{-\infty}^{\infty}\left\vert{\gamma(1-Q-it,\rho_k,\psi) \mathcal{M}(V,k)(-2(Q+it))}\right\vert \ \mathrm{d}t \\ &\ll r^{2(1-Q)}\sum_{k\in\Z}\int_{-\infty}^{\infty}\left((1+\vert{k}\vert)(1+\vert{t}\vert)\right)^C\vert{\mathcal{M}(V,k)(-2(Q+it))}\vert \ \mathrm{d}t \ll r^{2(1-Q)},
        \end{align*}
        where the last inequality follows from \eqref{eq:inert}.
        Similarly, for the bound $B_{\pi}[V](re^{i\theta}) \ll_A r^{-A}$ one shifts the countour to $\real(s) = \frac{A}{2}$.  
       
    \end{proof}
        \begin{remark}\label{rmk:complexbesseltransform}
        Zhi Qi has studied extensively the complex Bessel transforms, see for example, \cite{qi2021theory}, \cite{Qi2019}, \cite{Qi2024}. The analysis one can do with the Stirling's approximation is very limited in the complex case, although sufficient for this paper. If one wants to study subconvexity in the $t$-aspect or simultaneously in the $N(\mathfrak{q})$-aspect and $t$-aspect one has certainly to use Zhi Qi's results.  
    \end{remark}

\section{The key identity}\label{sec:keyidentity}In this section we discuss an extension of the key identity used in \cite{nelsonholowinsky}, but first we fix the relevant data: let $\mathfrak{q}\subset O_F$ be a 
prime ideal and  assume that $\mathfrak{q}$ is coprime to each $I\in \mathcal{S}_{\cl}$ (see subsection \ref{subsec:2.choice of fundamental domain}) and to the different ideal of $F$. Let $\alpha_{\mathfrak{q}} \in O_F$ be so that \begin{equation}\label{def:alpha}(\alpha_{\mathfrak{q}})^{-1}\mathfrak{q} \in \mathcal{S}_{\cl}.\end{equation} Then $$N(\mathfrak{q})\leq N(\alpha_{\mathfrak{q}}) \leq  \max_{I\in \mathcal{S}_{\cl}}N(I^{-1})N(\mathfrak{q}),$$ that is $N(\alpha_\mathfrak{q}) \asymp N(\mathfrak{q})$ and $v_{\mathfrak{q}}(\alpha_\mathfrak{q})=1$. By Lemma \ref{lem:choicegenerator} (assuming $N(\mathfrak{q})$ is large enough), we may multiply $\alpha_{\mathfrak{q}}$ by a unit to assume that $$N(\mathfrak{q})^{\frac{1}{d_v}-\epsilon }\ll \vert{\alpha_{\mathfrak{q}}}\vert_v \ll N(\mathfrak{q})^{\frac{1}{d_v}+\epsilon}$$ for each $v\in S_{\infty}$.

Let $\chi\in \widehat{F^{\times}\setminus\A^{\times}}$ be a finite order character of conductor $\mathfrak{q}$. For each $v\in S_{\infty}$ let $\Delta_v \in \R_{>0}$ be a number satisfying $N(\mathfrak{q})^{\epsilon}< \Delta_v < N(\mathfrak{q})^{100}$ and so that $$\prod_{v\in S_{\infty}}\vert{\Delta_v}\vert_v < N(\mathfrak{q})^{1-\epsilon}.$$ Let $\Delta\in F_{\infty}^{\times}\hookrightarrow \A^{\times}$ be the corresponding idèle $\Delta_{\infty} = (\Delta_v)_{v\in S_{\infty}}$ and $\Delta_v = 1$ for $v<\infty$.
Let $\phi \in \mathcal{A}(\gelle_3)_F$ be any automorphic form.
Using the fundamental domain \eqref{eq:fundamentaldomain2} we write:
\begin{align*}
    I_t(\phi,\chi) &= \sum_{I \in \mathcal{S}_{\cl}}\chi(I) \int_{\widehat{O}^{\times}E}\mathbb{P}\phi(a_1(Iyz_t))\chi(y) \ \mathrm{d}^{\times}y.
    \end{align*}
Let $F^{\mathfrak{q}} \subset F^{\times}$ be the subring of elements that are coprime to $\mathfrak{q}$ and are \emph{totally positive}. Let $F^{\mathfrak{q},1}\subset F^{\mathfrak{q}}$ be the subring of elements that are $\equiv 1 \mods{q}$. Then $$F^{\mathfrak{q}}/F^{\mathfrak{q},1} \simeq (O_{v_\mathfrak{q}}/\mathfrak{q})^{\times} \simeq (O_F/\mathfrak{q})^{\times}.$$ We write     
  \begin{align*}
      I_t(\phi,\chi)&= \sum_{I \in \mathcal{S}_{\cl}}\chi(I)\sum_{a \in F^{\mathfrak{q}}/F^{\mathfrak{q},1}}\chi(\iota_\mathfrak{q}(a))\int_{\widehat{O}^{\times}_1(\mathfrak{q})E}\mathbb{P}\phi(a_1(I\iota_{\mathfrak{q}}(a)yz_t))\chi(y_{\infty}) \ \mathrm{d}^{\times}y.
\end{align*}
For $\eta \in \widehat{F^{\mathfrak{q}}/F^{\mathfrak{q},1}}$, we denote by $I_t(\phi,\eta,\chi|_{\cl})$ the following sum.
\begin{equation*}
    I_t(\phi,\eta,\chi|_{\cl}) = \sum_{I \in \mathcal{S}_{\cl}}\chi(I)\sum_{a \in F^{\mathfrak{q}}/F^{\mathfrak{q},1}}\eta(a)\int_{\widehat{O}^{\times}_1(\mathfrak{q})E}\mathbb{P}\phi(a_1(I\iota_{\mathfrak{q}}(a)yz_t))\chi(y_{\infty}) \ \mathrm{d}^{\times}y.
\end{equation*}
Consider the following first-moment-looking sum.
\begin{equation}\label{eq:Ffirst}
    \sum_{\delta \in F}\sum_{\eta \in \widehat{F^{\mathfrak{q}}/F^{\mathfrak{q},1}}} \eta(\delta)\overline{\chi}(\iota_{\mathfrak{q}}(\delta))I_t(\phi,\eta,\chi|_{\cl})V(\delta \Delta^{-1}),
\end{equation}
where we choose $V_0 = \prod_v{V_0}_v \in C_c^{\infty}(\A)$ so that for $v< \infty$:\begin{equation*}{V_0}_v=\begin{cases}
     \mathbb{1}_{O_v}  &v(\mathfrak{C}(\pi))=0,\\
     \mathbb{1}_{O_v^{\times}} &v(\mathfrak{C}(\pi))>0\\
    \end{cases}\end{equation*}
   while for $v\in S_{\infty}$ we select ${V_0}_v\in C_c^{\infty}(F_v)$ so that it satisfies the following properties:
   \begin{itemize}\item ${V_0}_v \text{ is positive real valued}$,
   \item $
   {V_0}_v(u)\neq 0 \Rightarrow \vert{u}\vert_v \asymp 1 \text{ and } u>0 \text{ if } v \text{ real}$,
   \item If $v$ is real, then for all $m,n \in \Z$: $\Vert{x^n\partial_x^m{V_0}_v}\Vert_{\infty}\ll_{m,n} 1$.
   \item If $v$ is complex, then for all $m_1,m_2,n_1,n_2 \in \Z$: $$\Vert{\real(z)^{m_1}\imag(z)^{n_1}\partial_{\real(z)}^{m_2}\partial^{n_2}_{\imag(z)}{V_0}_v}\Vert_\infty\ll_{m_1,m_2,n_1,n_2} 1.$$
   \end{itemize}
   In addition, we normalize $V_0$ so that $\widehat{V_0}(0) = 1$. Since $\vert{\Delta}\vert < N(\mathfrak{q})^{1-\epsilon}$, we have, with this choice of $V_0$, that: \[V_0(\delta\Delta^{-1})\neq0 \Rightarrow \delta \in F^{\mathfrak{q}}\cap O_F. \] 
    We apply Poisson summation formula to the $\delta$ sum in \eqref{eq:Ffirst}:
\begin{align*}
    &\sum_{a \in F^{\mathfrak{q}}/F^{\mathfrak{q},1}}\eta(a)\overline{\chi}(\iota_{\mathfrak{q}}(a))\sum_{\delta \in F^{q,1}}V_0(\delta a\Delta^{-1}) \nonumber=\\&\sum_{a\in F^{\mathfrak{q}}/F^{\mathfrak{q},1}}\eta(a)\overline{\chi}(\iota_{\mathfrak{q}}(a))\sum_{\delta \in F}\prod_{\substack{v<\infty\\v\neq v_{\mathfrak{q}}}}\mathbb{1}_{O_v}(\delta)\times \mathbb{1}_{-a+\varpi_{\mathfrak{q}}O_{v_\mathfrak{q}}}(\delta)\times \prod_{v\in S_{\infty}}{V_0}_v(\delta \Delta^{-1})
    \nonumber=\\ &\frac{\vert{\Delta}\vert}{N(\mathfrak{q})}\sum_{\substack{a\in F^{\mathfrak{q}}/F^{\mathfrak{q},1}}}\eta(a)\overline{\chi}(\iota_{\mathfrak{q}}(a))\psi_{v_{\mathfrak{q}}}\left(a\frac{\delta }{\varpi_{\mathfrak{q}}}\right)\sum_{\delta \in F}\widehat{V_0}\left(\delta\varpi_{\mathfrak{q}}  \Delta \right).
\end{align*}
Hence, we see that \eqref{eq:Ffirst} is also equal to 
\begin{equation}\label{eq:dualF}
\frac{\vert{\Delta}\vert}{N(\mathfrak{q})}\sum_{\delta \in F}\widehat{V_0}\left(\delta  \varpi_{\mathfrak{q}}\Delta\right)\sum_{\substack{a\in F^{\mathfrak{q}}/F^{\mathfrak{q},1}}}\eta(a)\overline{\chi}(\iota_{\mathfrak{q}}(a))\psi_{v_{\mathfrak{q}}}\left(a\frac{\delta }{\varpi_{\mathfrak{q}}}\right)I_t(\phi,\eta,\chi|_{\cl})
\end{equation}
We now take the sum over $\eta\in \widehat{F^{\mathfrak{q}}/F^{\mathfrak{q},1}}$ in both \eqref{eq:Ffirst} and \eqref{eq:dualF}, opening $I_t(\phi,\eta,\chi|_{\cl})$, and we get the following equality. 
\begin{align*}
    &\sum_{\delta \in F}V_0(\delta \Delta^{-1})\chi(\iota_{\mathfrak{q}}(\delta^{-1}))\sum_{I \in \mathcal{S}_{\cl}}\chi(I)\int_{\widehat{O}^{\times}_1(\mathfrak{q})E}\mathbb{P}\phi(a_1(I\iota_\mathfrak{q}(\delta^{-1})yz_t)\chi(y_{\infty})\ \mathrm{d}^{\times}y = \\&\frac{\vert{\Delta}\vert}{N(\mathfrak{q})^{1/2}}\sum_{\delta \in F}\widehat{V_0}\left(\delta\varpi_\mathfrak{q}\Delta\right) 
    \sum_{a\in (F^{\mathfrak{q}}/F^{\mathfrak{q},1}) 
    }\psi_{v_{\mathfrak{q}}}\left(\frac{a\delta}{\varpi_{\mathfrak{q}}}\right)\chi(\iota_{\mathfrak{q}}(a^{-1}))\\&\times \sum_{I\in \mathcal{S}_{\cl}}\chi(I)\int_{\widehat{O}^{\times}_1(\mathfrak{q})E}\mathbb{P}\phi(a_1(I\iota_\mathfrak{q}(a^{-1})yz_t))\chi(y_{\infty})\ \mathrm{d}^{\times}y
\end{align*}
Notice that the $\delta$-sum on the right-hand side runs over all the $\delta \in F\cap \widehat{O}$. If, in addition, $\delta \in F\cap \mathfrak{q}O_{v_{\mathfrak{q}}}$, then $\psi_{v_{\mathfrak{q}}}\left(\frac{a\delta }{\varpi_{\mathfrak{q}}} \right)=1$ and the inner sums on right hand side form again   
$ I_t(\phi,\chi)$. 
By repeated integration by parts, we have
\begin{align*}
    \widehat{V_0}_v(x\Delta_v) &\ll \min(\frac{1}{\vert{x\Delta_v}\vert^m}\widehat{\partial_x^mV_0}(x\Delta_v),\Vert{{V_0}_v}\Vert_1) \qquad v \text{ real},\\
    \widehat{V_0}_v(z\Delta_v) &\ll \min(\frac{1}{\vert{2\real(z)\Delta_v}\vert^m}\frac{1}{\vert{2\imag(z)\Delta_v}\vert^n}\widehat{\partial^m_{\real(z)}\partial^n_{\imag(z)}{V_0}_v}(z\Delta_v),\Vert{{V_0}_v}\Vert_1)\quad v \text{ complex}.
\end{align*}
In particular, $\widehat{V_0}(x\Delta)$ is negligible small whenever $\vert{x}\vert_v\gg \Delta_v^{-1}N(\mathfrak{q})^{\epsilon}$ for some $v \in S_{\infty}$.
Also, for $v<\infty$ we have 
\begin{equation*}
   \widehat{V_0}_v =\begin{cases} \vert \mathfrak{D}\vert_v^{1/2}
    \mathbb{1}_{\mathfrak{D}^{-1}O_v} &v(\mathfrak{C}(\pi))=0,\\
    -\vert \mathfrak{D}\vert_v^{1/2}\mathbb{1}_{\varpi_{v}^{-(v(\mathfrak{D})+1)}O_v^{\times}}
    &v(\mathfrak{C}(\pi))>0.\end{cases}
\end{equation*}
 Therefore, if $\delta \in F$ is so that $\widehat{V_0}\left(\delta\varpi_\mathfrak{q}\Delta\right)\neq 0$, then $\delta\in \prod_{v:v(\mathfrak{D})>0}\mathfrak{p}_v^{-(v(\mathfrak{D})+1)} \cap \mathfrak{q}O_{v_\mathfrak{q}}$. If in addition $\delta$ is non-zero, we deduce that $N(\delta) \gg N(\mathfrak{q})$ and so $\vert{\delta}\vert_v \gg \Delta_v^{-1}N(\mathfrak{q})^{\epsilon}$ for some $v\in S_{\infty}$, that is ${V_0}_v(\delta\Delta_v)$ is negligible small and so is $\widehat{V_0}\left(\delta \varpi_{\mathfrak{q}}\Delta\right)$.
We deduce the following \emph{key identity}: 
\begin{align}\label{eq:keyidentity}
&N(\mathfrak{q})^{-\frac{1}{2}}I_t(\phi,\chi)= 
    \sum_{I \in \mathcal{S}_{\cl}}\int_{\widehat{O}^{\times}_1(\mathfrak{q})E}\mathcal{F}_t^{\phi,\chi}(y,I) - \mathcal{O}_t^{\phi,\chi}(y,I)\   \mathrm{d}^{\times}y + O(N(\mathfrak{q})^{-100}),
\end{align}
where 
\begin{align}\label{eq:ftphi}
    \mathcal{F}^{\phi,\chi}_t(y,I)&= \frac{N(\mathfrak{q})^{\frac{1}{2}}}{\vert{\Delta}\vert}\chi(I)\sum_{\delta \in F}\overline{\chi}(\iota_{\mathfrak{q}}(\delta))V_0(\delta \Delta^{-1})\mathbb{P}\phi(a_1(I\iota_{\mathfrak{q}}(\delta^{-1})yz_t))\chi(y_{\infty}) 
\end{align} 
and 
\begin{align}\label{eq:otphi}
     \mathcal{O}^{\phi,\chi}_t(y,I)&=\frac{1}{N(\mathfrak{qD})^{\frac{1}{2}}}\sum_{\substack{\delta \in F^{\times}\\\forall v \in S_{\infty}:\\\vert{\delta}\vert_v\ll N(\mathfrak{q})^{\epsilon}\vert{\Delta}\vert_v^{-1} }}\widehat{V_0}\left(\delta\varpi_\mathfrak{q}\Delta\right)\sum_{a\in F^{\mathfrak{q}}/F^{\mathfrak{q},1}}\psi_{v_{\mathfrak{q}}}\left(\frac{a\delta}{\varpi_{\mathfrak{q}}}\right)\chi(\iota_{\mathfrak{q}}(a^{-1})) \nonumber\\&\qquad\times\mathbb{P}\phi(a_1(I\iota_\mathfrak{q}(a^{-1})yz_t))\chi(y_{\infty}),
\end{align}
We define $a_{\delta} = \frac{1}{\vert{\Delta}\vert}\overline{\chi}(\iota_{q}(\delta))V_0(\delta\Delta^{-1})$.
\begin{lemma}
    We have $\vert{a_{\delta}}\vert \ll N(\mathfrak{q})^{\epsilon}\vert{\Delta}\vert^{-1}$ and $\sum_{\delta \in F}\vert{a}_{\delta}\vert \asymp 1$.
\end{lemma}
\begin{proof}
    The first statement is clear. For the second, we have 
    \[\sum_{\delta}\vert{a_{\delta}}\vert = \frac{1}{\Delta}\sum_{\delta}V_0(\delta \Delta^{-1}) = \sum_{\delta}\widehat{V_0}(\delta\Delta) = \widehat{V_0}(0) + \sum_{\delta\neq 0}\widehat{V_0}(\delta\Delta). \]
    The tail $\sum_{\delta\neq 0}\widehat{V_0}(\delta\Delta)$ is negligible small, since for $\delta \neq 0$ so that $\widehat{V_0}(\delta\Delta )\neq0$, we have $\vert{\delta\Delta}\vert_v\gg N(\mathfrak{q})^{\epsilon}$ for some $v\in S_{\infty}$.
\end{proof}

\begin{remark}[Heuristic]\label{rmk:heuristic}Note that the set of $\delta$ contributing in the $\mathcal{F}_t$ basically (and it does if one wants to choose the archimedean functions more precisely) injects  into $(O/\mathfrak{q})^{\times} = F^{\mathfrak{q}}/F^{\mathfrak{q},1}$. Let $\phi \in \pi^\infty$ and consider the right translation $\phi^{u_{12}(\zeta_\mathfrak{q})}$, where we recall Section \ref{sec:Notation and basic set-up} for the definitions of $\zeta_\mathfrak{q}$ and $u_{12}(\zeta_\mathfrak{q})$. We write $\mathbb{P}\phi^{u_{12}(\zeta_\mathfrak{q})}$ using the Fourier-Whittaker expansion  
\[\mathbb{P}\phi^{u_{12}(\zeta_\mathfrak{q})}(a_1(I\iota_\mathfrak{q}(\delta^{-1})yz_t)) = \frac{1}{t^{1/2}}\sum_{\gamma \in F^\times}\psi(\gamma\delta^{-1}\zeta_{\mathfrak{q}})W_{\phi}(a_1(\gamma I\iota_\mathfrak{q}(\delta^{-1})yz_t)). \]
A trivial bound for $\mathbb{P}\phi^{u_{12}(\zeta_\mathfrak{q})}(a_1(I\iota_\mathfrak{q}(\delta^{-1})yz_t))$ can be obtained by Cauchy-Schwarz (a Rankin-Selberg-type bound). To obtain better bounds, we can exploit the oscillation coming from $\psi$. To do this, we use the fact that $\delta$ has size $N(\delta) \asymp \vert\Delta\vert < N(\mathfrak{q})$ and additive reciprocity to approximate $\psi(\gamma\delta^{-1}\zeta_\mathfrak{q})\sim \psi(\gamma\alpha_\mathfrak{q}^{-1}\zeta_\delta)$, for some $\alpha_\mathfrak{q} \in \mathfrak{q}$, so that a use of Voronoi summation formula leads to savings.

From \eqref{eq:keyidentity} one can see alternatively that
\[\mathcal{O}_t^{\phi,\chi} = \sum_{\delta}\sum_{\substack{\eta\in \widehat{F^\mathfrak{q}/F^{\mathfrak{q},1}}\\\eta\neq \chi|_{F^{\mathfrak{q}}/F^{\mathfrak{q},1}}}}\eta(\delta)\overline{\chi}(\iota_\mathfrak{q}(\delta))I_t(\phi,\eta,\chi|_{\cl})V_0(\delta\Delta^{-1})+O(N(\mathfrak{q})^{-100}).\]
However, to have power of $N(\mathfrak{q})$ cancellation for $\mathcal O_t^{\phi,\chi}$, we need $\vert \Delta \vert > N(\mathfrak{q})^{1+\kappa}$ for some $\kappa>0$, whereas for the $\mathcal{F}_{t}^{\phi,\chi}$ side we actually need $\vert{\Delta}\vert < N(\mathfrak{q})^{1-\kappa}$. This will be achieved by introducing an amplification procedure, this is done in section \ref{sec:amplification}.

\end{remark}

Finally, given a parameter $T>0$, an inert function $g\in C_c^{\infty}(\R_{>0})$ and an automorphic form $\phi \in \mathcal{A}_{\gelle_3/F}$ we define 
\begin{align*}\mathcal{F}^{\phi,\chi}_T &= \int_{\R_{>0}}g\left(\frac{t}{T}\right) \sum_{I\in \mathcal{S}_{\cl}}\int_{\widehat{O}_1^{\times}(\mathfrak{q})E}\mathcal{F}_t^{\phi,\chi}(y,I)\ \mathrm{d}^{\times}y\ \mathrm{d}^{\times}t\\
\mathcal{O}_T^{\phi,\chi}&=\int_{\R_{>0}}g\left(\frac{t}{T}\right) \sum_{I\in \mathcal{S}_{\cl}}\int_{\widehat{O}_1^{\times}(\mathfrak{q})E}\mathcal{O}_t^{\phi,\chi}(y,I)\ \mathrm{d}^{\times}y\ \mathrm{d}^{\times}t.
\end{align*}

\section{Approximate functional equation: choice of test vector}\label{sec:approximatefunctionalequation}
Let, in this section and for the rest of the paper, $\phi_0 \in \pi^{\infty}$ denote an global newform so that its global whittaker functions $W_{\phi_0}$ splits as a product $$W_{\phi_0}(g) = \prod_{v<\infty}W^{\circ}_{v}(g_v) \times \prod_{v\in S_{\infty}}W_v(g_v),$$ where for all $v\in S_{\infty}$ the function $y\in F_v^{\times}\mapsto W_v(a_1(y))$ is inert (see Definition \ref{def:inert}). Recall that for $v<\infty$ we denote by $W^{\circ}_v \in \mathcal{W}(\pi_v,\psi_v)$ the new-vector (see Subsection \ref{subsect:newforms}).  

In addition, we assume that for all $v\in S_\infty$ \begin{equation}\label{eq:assumption-on-archimdean}
    \int_{F_v^\times}W_v(a_1(y))\vert y \vert_v^{-1/2}\ \mathrm{d}^\times y = \int_{F_v^\times}W_v(a_1(y^{-1}))\vert y\vert_v^{-1/2}\ \mathrm{d}^\times y = 1. 
\end{equation}

\subsection{Shift of newforms and definitive choice of vector in $\pi$}\label{subsect:snca}
Recall that $\zeta_\mathfrak{q}\in \A$ is the adèle so that for any place $v$ of $F$, $$(\zeta_\mathfrak{q})_v=\begin{cases}
    \varpi_\mathfrak{q}^{-1}& v=v_\mathfrak{q}\\0&\text{otherwise},
\end{cases},$$
that $$u_{12}(\zeta_\mathfrak{q}) = \begin{pmatrix}
    1&\zeta_\mathfrak{q}&\\&1&\\&&1
\end{pmatrix}\in \gelle_3(\A)$$
and that $\phi_0^{u_{12}(\zeta_\mathfrak{q})}$ is the right-translation of $\phi_0$ by $u_{12}(\zeta_\mathfrak{q})$.

The next Lemma is the analog of \cite[Lemma 11.8]{sparse-equidistribution-venkatesh}.

\begin{lemma}
    \label{prop:choiceofphi} 
Let $\phi_0$ be as in the beginning of the section and denote $W_{\phi_0}$ its associated Whittaker function.
Let $\mathfrak{f} \subset O_F$ be an integral ideal and $\eta \in \widehat{F^{\times}\setminus\A^{\times}}$ be a finite order character with finite conductor $\mathfrak{f}$. Then for any $s\in \CC$, we have
        \[ I(\phi_0^{u_{12}(\zeta_{\mathfrak{f}})},\eta,s) = G(\eta,\psi(\zeta_{\mathfrak{f}}\cdot)) L^{(\mathfrak{f})}(\pi \otimes \eta,s)\prod_{v\in S_{\infty}} \int_{F_{v}^{\times}}W_{v}(a_1(y))\vert{y}\vert_{v}^{s-1}\ \mathrm{d}^{\times}y,  \]
        where $G(\eta,\psi(\zeta_{\mathfrak{f}}\cdot) = \prod_{v|\mathfrak{f}}G(\eta_v,\psi_v(\varpi_v^{-v(\mathfrak{f})}\cdot))$ and $\vert{G(\eta,\psi(\zeta_{\mathfrak{f}}\cdot))}\vert  =  N(\mathfrak{f})^{-1/2}$.
    In particular,
    \[L^{(\mathfrak{f})}(\pi\otimes \eta,\tfrac{1}{2}) = G(\eta,\psi(\zeta_\mathfrak{f}\cdot))^{-1}I(\phi_0^{u_{12}(\zeta_f)},\eta,\tfrac{1}{2}).\]
\end{lemma}
\begin{proof}
    
    For $\operatorname{re}(s)>1$ we have using Lemma \ref{lem:shiftnewform} and the properties of non-archimedean new-vectors \begin{align*}
    I(\phi_0^{u_{12}(\zeta_{\mathfrak{f}})},\eta,s) =&  
          \prod_{v<\infty} \int_{F_v^{\times}}\psi_v({\zeta_{\mathfrak{f}}y})W^{\circ}_{v}(a_{1}(y))\eta_v(y)\vert{y}\vert_v^{s-1}\mathrm{d}^{\times}y\  \\&\times\prod_{v\in S_{\infty}}\int_{F_v^{\times}}W^{\circ}_{v}(a_{1}(y))\eta_v(y)\vert{y}\vert_v^{s-1}\mathrm{d}^{\times}y
         \\ =&L^{(\mathfrak{f})}(\pi\times\eta,s)G(\eta,\psi(\zeta_{\mathfrak{f}}\cdot))\prod_{v\in S_{\infty}}\int_{F_{v}^{\times}}W_{v}(a_{1}(y))\vert{y}\vert_v^{s-1}\mathrm{d}^{\times}y.
    \end{align*}
    Both left and right-hand sides have meromorphic continuation for $s\in \CC$, so they agree everywhere. We conclude using Lemma \ref{lem:gausssum} and recalling the normalization \eqref{eq:assumption-on-archimdean}.
    \end{proof}

\subsection{Approximate functional equation}
The next step is to state an approximate functional equation. In our context, this amounts in approximating the integral in the central direction. Recall that we write for $\phi \in \pi^{\infty}$, $t>0$ and $\eta\in \widehat{F^{\times}\setminus\A^{\times}}$, \[I_t(\phi,\eta) = \int_{F^{\times}\setminus\A^{(1)}}\mathbb{P}\phi(a_1(yz_t))\eta(y)  \mathrm{d}^{\times}y, \]
so that 
\[I(\phi,\eta,1/2) = \int_{0}^{\infty}I_t(\phi,\eta)\ \mathrm{d}^{\times}t\]
and more generally
\[I(\phi,\eta,s+1/2) = \int_0^{\infty}I_t(\phi,\eta)t^s\ \mathrm{d}^{\times}t.\]
In particular, by an approximate functional equation, we mean a statement of the following form:
\begin{align*}I(\phi,\eta,1/2) &\sim \int_{\substack{t>0\\ t \asymp N(\mathfrak{f})^{-3/2}}}I_{t}(\phi,\eta) \ \mathrm{d}^{\times}y\ \mathrm{d}^{\times}t, \end{align*}
for any reasonable $\phi$. The sign $\sim$ is to be qualitatively interpreted as an approximation, but without a precise meaning yet.
We do this only on $\phi = \phi_0^{u_{12}(\zeta_{\mathfrak{f}})},$
where $\mathfrak{f}\subset O_F$ is the finite conductor of $\eta$.

The next Lemma is classical.

\begin{lemma}\label{lem:decaying}
    Let $h\colon (0,\infty)\to \CC$ be a smooth function such that $h(y)\neq 0 \Rightarrow y \ll 1$ and $(y\partial_y)^mh(y)\ll_m 1$ for any $m\in \Z_{\geq 0}$. 
    Then for any $\sigma>0$ and any polynomial $P \in \CC[X]$ it holds that 
    \[\int_{(\sigma)}\vert{\mathcal{M}(h,s)P(s)}\vert \ \mathrm{d}s \ll_{\sigma,\deg P} 1, \]
    where the implicit constant may depend on the degree of $P$ and on $\sigma$. Similarly, if $h(y)\neq 0 \Rightarrow y\gg 1$ and $(y\partial_y)^mh(y)\ll_m 1$, for any $m\in \Z_{\geq 0}$, then for any $\sigma < 0$ and any polynomial $P\in \CC[X]$ it holds that 
    \[\int_{(\sigma)}\vert{\mathcal{M}(h,s)P(s)}\vert \ \mathrm{d}s \ll_{\sigma,\deg P} 1, \]
    where the implicit constant may depend on the degree of $P$ and on $\sigma$. 
\end{lemma}
\begin{proof} Suppose $P$ is non-zero, otherwise the statement is trivially verified.
     By integration by parts we have, for $s\neq 0$ and $m\in \Z_{\geq 0}$ that 
    $\mathcal{M}(h,s) = \frac{1}{s^m}\mathcal{M}((y\partial_y)^mh,s),$ hence  
    \begin{align*}
        &\int_{(\sigma)}\vert{\mathcal{M}(h,s)P(s)}\vert\ \mathrm{d}s\\ &\ll \int_0^{\infty}\vert{(y\partial_y)^{\deg P + 2}h(y)}y^{\sigma-1}\vert \int_{(\sigma)}\left\vert{\frac{P(s)}{s^{\deg P + 2}}}\right\vert\  \mathrm{d}s\ \mathrm{d}y\ll 1.
    \end{align*}
    The second statement follows by the first with the change of variables $y\to y^{-1}$, or also by direct computation as above.
\end{proof}

 The following is an analog of \cite[Lemma 11.9]{sparse-equidistribution-venkatesh} for $\gelle_3$ periods, and it is sometimes referred to as the geometric version of the approximate functional equation, even though it does not use directly the functional equation of $L(\pi\otimes\eta,s)$, but rather the convexity bound (which of course is classically deduced from the functional equation, at least on one side).
\begin{lemma}[Geometric approximate functional equation]\label{lem:convexitybound} 
Let $\mathfrak{f}\subset O_F$ be an integral ideal and $\eta \in \widehat{F^{\times}\setminus\A^{\times}}$ a finite order character of conductor $\mathfrak{f}$. Let $\kappa \in (0,\frac{1}{2})$. Then there exists a smooth function $h_0 \colon (0,\infty)\to [0,1]$ such that \begin{align*}h_0(t) \neq 0 &\Rightarrow N(\mathfrak{f})^{-(3/2+\kappa)}\ll t\ll N(\mathfrak{f})^{-(3/2-\kappa)}\\
\forall j\in \Z_{\geq 0}&: \vert t^j h_0^{(j)}(t)\vert\ll_j 1. \end{align*} and that
\begin{equation*}\label{eq:geometricapprox}
        I(\phi_0^{u_{12}(\zeta_{\mathfrak{f}})},\eta,1/2) = \int_{0}^{\infty}h_0(t)I_t(\phi_0^{u_{12}(\zeta_{\mathfrak{f}})},\eta)\ \mathrm{d}^{\times}t + O(N(\mathfrak{f})^{1/4-\kappa/2+\epsilon}) 
    \end{equation*}
\end{lemma}

\begin{proof}
 Let $h\colon (0,\infty)\to [0,1]$ be a smooth function so that $h(x) = 0$ for $x \gg 1$ and $h(x) = 1 $ for $x \in (0,1]$. In addition, we assume that $(y\partial_yh)^m \ll_m 1$ for any $m\in \Z_{\geq 0}$. For any $M> 0$ denote $h_{M}(t) = h\left(\frac{t}{M}\right)$ ($t>0$). Denote for the proof, $\phi = \phi_0^{u_{12}(\zeta^\mathfrak{f})}$. Let $A = A_{\mathfrak{f}}>0$ and $B=B_{\mathfrak{f}}>0$ be two quantities depending on $\mathfrak{f}$ that we will choose later.
    We have, by Parseval, for any $\sigma >0$
    \begin{equation*}\label{eq:parseval}
        \int_{0}^{\infty}h_A(t)I_t(\phi,\eta,1/2)\ \mathrm{d}^{\times}t = \frac{1}{2\pi i}\int_{(\sigma)}A^{s}\mathcal{M}(h,s)I(\phi,\eta,-s+1/2) \ \mathrm{d}s,
    \end{equation*}
    At $\real(s) = 1/2$ we have the convexity bound 
     \[L(\pi \times \eta, -s+\tfrac{1}{2}) \ll (1+\vert{s}\vert)^{\frac{3d}{2}+\epsilon} C(\pi\times \eta)^{1/2+\epsilon} \leq (1+\vert s \vert)^{\frac{3d}{2}+\epsilon}C(\pi)^{1/2+\epsilon}N(\mathfrak{f})^{3/2+\epsilon},  \]
    which together with Lemmas \ref{prop:choiceofphi} and \ref{lem:decaying} gives
    \begin{equation*}
        \int_{(1/2)}A^{s}\mathcal{M}(h,s)I(\phi,\eta,-s+1/2) \ \mathrm{d}s \ll_{\epsilon,h} A^{1/2}C(\pi)^{1/2+\epsilon}N(\mathfrak{f})^{1+\epsilon}.
    \end{equation*}
     Hence, choosing $A = N(\mathfrak{f})^{-3/2-\kappa}$, we get 
     \begin{equation*}
          \int_{(1/2)}A^{s}\mathcal{M}(h,s)I(\phi,\eta,-s+1/2) \ \mathrm{d}s \ll_{\epsilon} C(\pi)^{1/2+\epsilon}N(\mathfrak{f})^{1/4-\frac{\kappa}{2}+\epsilon}.
     \end{equation*}for $\kappa >0$ to be decided, we have a subconvex bound. 
    Similarly, we have for $\sigma <0$
    \begin{equation*}
         \int_{0}^{\infty}(1-h_B(t))I_t(\phi,\eta,1/2)\ \mathrm{d}^{\times}t = \frac{1}{2\pi i}\int_{(\sigma)}B^{s}\mathcal{M}(1-h,s)I(\phi,\eta,-s+1/2)\ \mathrm{d}s,
    \end{equation*}
     The convexity bound for $\real(s) = -1/2$ reads
    \[ L(\pi \times \eta,s-\tfrac{1}{2})  \ll (1+\vert{s}\vert)^{\epsilon} C(\pi\times \eta)^{\epsilon}, \] which together with Lemmas \ref{prop:choiceofphi} and \ref{lem:decaying} gives $$\int_{(-1/2)}B^{s}\mathcal{M}(1-h,s)I(\phi,\eta,-s+1/2)\ \mathrm{d}s\ll B^{-1/2}C(\pi)^\epsilon N(\mathfrak{ f})^{-1/2+\epsilon}.$$  We can choose $B= N(\mathfrak{f})^{-3/2+\kappa}$ and obtain 
     \begin{equation*}
         \int_{0}^{\infty}(1-h_B(t))I_t(\phi,\eta,1/2)\ \mathrm{d}^{\times}t \ll C(\pi)^{\epsilon}N(\mathfrak{f})^{1/4-\kappa/2+\epsilon}
    \end{equation*}
    Hence, we choose $h_0\coloneqq h_B- h_A$ with $B= N(\mathfrak{f})^{-3/2+\kappa}$ and $A=N(\mathfrak{f})^{-3/2-\kappa}$.
    
\end{proof}

Using the functional equation directly, we can further shrink the range we are considering at the cost of having two sums. The proof is virtually the same as \cite[Theorem 16]{MR4370537}.
For simplicity, we will assume that $\mathfrak{C}(\pi)$ and $\mathfrak{f}$ are coprime; in particular, this implies that $$L_v(\pi\times \eta,s) = 1 = L_v(\tilde{\pi}\times \overline{\eta},s),$$ whenever $v|\mathfrak{f}$. Recall from subsection \ref{subsect:newforms}, that for any $v< \infty$ we denote by $U_v^{\circ} \in \mathcal{W}(\pi_v,\overline{\psi}_v)$ the newform of $\tilde{\pi}_v$.

\begin{proposition}[Approximate functional equation]\label{prop:approximatefunctionalequation}
Let $\mathfrak{f} \subset O_F$ be an integral ideal and $\eta \in \widehat{F^{\times}\setminus\A^{\times}}$ be a finite order character of conductor $\mathfrak{f}$.
Let $\frac{1}{2}>\kappa > \epsilon > 0$. Suppose that $\mathfrak{C}(\pi)$ and $\mathfrak{f}$ are coprime. 

For $v\in S_{\infty}$ let $U_v \in \mathcal{W}(\tilde{\pi}_v,\overline{\psi}_v)$ be so that $U_v(a_1(y))= W_v(a_1(y^{-1}))$. Let $\prod_{v<\infty}U^{,\circ}_v \times \prod_{v\in S_{\infty}}U_v = U_{\phi_0^{\vee}}$ be the corresponding global Whittaker function in $\mathcal{W}(\tilde{\pi},\overline{\psi})$ and $\phi_0^{\vee}(g) = \sum_{\gamma \in U_2(F)\setminus \gelle_3(F)}U_{\phi_0^{\vee}}(\gamma g)$ the corresponding vector in $\tilde{\pi}^{\infty}$. 

Then there exist smooth functions $h_0\colon (0,\infty)\to [0,1]$ such that 
\begin{align*}h_0(t) \neq 0 &\Rightarrow N(\mathfrak{f})^{-(3/2+\kappa)}\ll t\ll N(\mathfrak{f})^{-(3/2-\kappa)}\\
\forall j\in \Z_{\geq 0}&: \vert t^j h_0^{(j)}(t)\vert\ll_j 1. \end{align*} 
so that
\begin{align*}\label{eq:approximatefunctionalequation}
    I(\phi_0^{u_{12}(\zeta_\mathfrak f)},\eta,1/2) \leq \ & \int_{0}^{\infty}h_0(t)I_t(\phi_0^{u_{12}(\zeta_\mathfrak f)},\eta)\ \mathrm{d}^{\times}t+ \int_{0}^{\infty}k_0(t)I_t((\phi_0^{\vee})^{u_{12}(\zeta_\mathfrak f)},\overline{\eta})\ \mathrm{d}^{\times}t \\&+ O(N(\mathfrak{f})^{1/4-\kappa/2+\epsilon}),
\end{align*}
where \[k_0(t) = \frac{1}{2\pi i}\int_{(\sigma)}\mathcal M(h_0,s)\gamma_{\infty}\left(\frac{1}{2}-s,\pi \times \eta\right)(C(\pi)N(\mathfrak{f})^3t)^{s} \]
satisfies $t^jk_0^{(j)}(t) \ll 1$ and $k_0(t) \ll_{\sigma} (N(\mathfrak{f})^{\frac{3}{2}}t)^\sigma$ for every $\sigma >0$. 
\end{proposition}
\begin{proof} Let $B= N(\mathfrak{f})^{-(3/2-\kappa)}$ and $A  = N(\mathfrak{f})^{-(3/2+\kappa)}$ and let $h = h_B-h_A$ be defined as in the proof of Lemma \ref{lem:convexitybound}.
    Let $D= N(\mathfrak{f})^{-3/2}$, then $B>D>A$. We set $h_0= h_B- h_D$ and $h_1 = h_D-h_A$, so that $h = h_0+h_1$. Write $\phi = \phi_0^{u_{12}(\zeta_\mathfrak{f})}$. Then by \eqref{eq:fctequations}
    \begin{align*}\int_{0}^{\infty}h_1(t)I_t(\phi,\eta)\ \mathrm{d}^{\times}t &= \frac{1}{2\pi i}\int_{(\sigma)}\mathcal{M}(h_1,s)I(\phi,\eta,-s+1/2) \ \mathrm{d}s \\ &=\frac{1}{2\pi i}\int_{(\sigma)}\mathcal{M}(h_1,s)\tilde{I}(\phi^{\iota},\overline{\eta},s+1/2) \ \mathrm{d}s.\end{align*}
    For $\real(s)>1/2$ we have 
    \begin{align*}
        &\tilde{I}(\phi^{\iota},\overline{\eta},s+1/2)= \prod_v\tilde{\Psi}_v(s+1/2,\rho(\sigma_{23})({W}^{\phi}_v)^{\iota},\overline{\eta_v})\\
        \\&= \theta(C(\pi)N(\mathfrak{f})^3)^{s} \prod_{v\in S_{\infty}}\gamma_{v}(1/2-s,\pi\times\eta) I((\phi_0^{\vee})^{u_{12}(\zeta_\mathfrak f)},\overline{\eta},s+1/2),
    \end{align*}
    where $\theta =\frac{G(\eta,\psi(\zeta_\mathfrak{f}\cdot))}{G(\overline{\eta},\overline{\psi}(\zeta_\mathfrak{f}\cdot))}\varepsilon(\pi\eta)$ is a number of absolute value $1$. The identity connecting the left-hand side and the last line is valid for any $s$. In particular, we get 
  \[\int_0^{\infty}h_1(t)I_t(\phi,\eta)\ \mathrm{d}^{\times}t = \int_{0}^{\infty}k_0(t)I_t((\phi_0^{\vee})^{u_{12}(\zeta_\mathfrak f)},\overline{\eta})\ \mathrm{d}^{\times}t \]
  with
   \begin{align*}k_0(t)=
   \frac{\theta}{2\pi i}\int_{(\sigma)}\mathcal{M}(h_1,s)\gamma_{\infty}\left(\frac{1}{2}-s,\pi \times \eta\right)(C(\pi)N(\mathfrak{f})^3t)^s \ \mathrm{d}s.\end{align*}
  We now study the properties of $k_0(t)$. Let $h_{1,D^{-1}}(t) = h_1(Dt)$. We have \[k_0(t) = \frac{1}{2\pi i}\int_{(\sigma)}\mathcal{M}(h_{1,D^{-1}},s)\gamma_{\infty}\left(\frac{1}{2}-s,\pi \times \eta\right)(C(\pi)DN(\mathfrak{f})^3t)^{s}\ \mathrm{d}s.\]
  Notice that $h_{1,D^{-1}}(y)  = h(y)-h(\frac{D}{A}y)$, hence $h_{1,D^{-1}}(y)\neq0 \Rightarrow y\ll 1$ and  $(y\partial_y)^mh_{2,D^{-1}}(y) \ll_m 1$ for any non-negative integer $m$. Hence, for any $\sigma>0$ we see that 
  \begin{align*}k_2(t) &\ll (C(\pi)DN(\mathfrak{f})^3t)^{\sigma}\int_{(\sigma)}\left\vert{\mathcal{M}(h_{1,D^{-1}},s)\gamma_{\infty}(\frac{1}{2}-s,\pi \times \eta)}\right\vert \ \mathrm{d}s\\ &\ll_{\sigma} (N(\mathfrak{f})^{\frac{3}{2}}t)^{\sigma}.\end{align*} Where in the last line we used Lemma \ref{lem:decaying}.  that $\gamma_{\infty}(1/2-s,\pi\times\eta)$ is polynomially bounded for fixed real part $\sigma$. Also, $\gamma_{\infty}(1/2-s,\pi\times\eta)$ does not have poles for $\real(s) = 0$, hence we see that $k_1(t) \ll 1$ for every $t >0$. 
\end{proof}

\subsection{Rankin-Selberg bound and dyadic partition of unity}
The classical Rankin-Selberg bound is the following statement:
\begin{equation}\label{eq:classicrankinselberg}
    \sum_{N(\mathfrak{a})N(\mathfrak{b})^2\ll X}\left\vert{W_{\fin}^{\circ}\begin{pmatrix}\mathfrak{ab}\mathfrak{D}^{-2}&&\\&\mathfrak{b}\mathfrak{D}^{-1}&\\&&1\end{pmatrix}}\right\vert^2N(\mathfrak{a}\mathfrak{b})^2 \ll X^{1+\epsilon}.
\end{equation}

Similarly as \cite[(2.5)]{nelsonholowinsky} we have
\begin{lemma}\label{lem:rankinweneed}
    For $M,N>0$ we have
    \begin{equation}\label{eq:Rankinselbergseparate}\sum_{N(\mathfrak{f})\leq M}\sum_{N(\mathfrak{b})\leq N}\vert{W_{\operatorname{fin}}^{\circ}(a(\mathfrak{aD}^{-1},\mathfrak{bD}^{-1})}\vert^2N(\mathfrak{ab})^2 \ll (MN)^{1+\epsilon} \end{equation}
\end{lemma}
\begin{proof} The proof follows \emph{loc. cit.} Denote $W_{\fin}^{\circ}(a(\mathfrak{aD^{-1}},\mathfrak{bD}^{-1}))N(\mathfrak{ab}) = \lambda(\mathfrak{a},\mathfrak{b})$.
 The left-hand side of \eqref{eq:Rankinselbergseparate} is bounded by 
  \[ \ll (MN)^{\epsilon}\sup_{M'\leq M,N'\leq N}M'^{-1}\sum_{\substack{M'/2\leq N(\mathfrak{a})\leq M'\\N'/2\leq N(\mathfrak{b})\leq N'}}N(\mathfrak{a})\vert{\lambda(\mathfrak{a},\mathfrak{b})}\vert^2.\]
  Similarly as in \cite[Lemma 2]{Munshi1} we use the Hecke relations  
  \[\lambda(\mathfrak{m},\mathfrak{n}) = \chi^{\pi}_{\fin}(\mathfrak{d})\sum_{\mathfrak{d}|(\mathfrak{m},\mathfrak{n})}\lambda(\mathfrak{md}^{-1},1)\lambda(1,\mathfrak{nd}^{-1}). \]
  to see that 
  \begin{align*}
      \sum_{\substack{N(\mathfrak{b})N(\mathfrak{a})^2\leq X}}N(\mathfrak{a})\vert{\lambda(\mathfrak{a},\mathfrak{b})}\vert^2 &\ll \sum_{N(\mathfrak{a})\leq X^{1/2}}N(\mathfrak{a})\sum_{\mathfrak{d}|\mathfrak{a}}\vert{\lambda(\mathfrak{ad}^{-1},1)}\vert^2\sum_{N(\mathfrak{b})\leq \frac{X}{N(\mathfrak{d})N(\mathfrak{a})^2}}\vert{\lambda(1,\mathfrak{b})}\vert^2\\&\ll X^{1+\epsilon}\sum_{N(\mathfrak{d})\leq X^{1/2}}\frac{1}{N(\mathfrak{d})^2}\sum_{N(\mathfrak{a})\ll \frac{X^{1/2}}{N(\mathfrak{d})}}\frac{\vert{\lambda(\mathfrak{a},1)}\vert^2}{N(\mathfrak{a})}\ll X^{1+\epsilon}.
  \end{align*}
  Hence, combining this last bound with the first one, we prove the claim.
\end{proof}
From this bound we can deduce that for any integral ideal $\mathfrak{f}$, and $\eta\in \widehat{F^{\times}\setminus \A^{\times}}$ a finite order character, $0<t\ll 1$ and any $\epsilon >0$ we have 
\begin{equation*}\label{eq:7rankinselberg}
    I_t(\phi,\eta) \ll_{\epsilon} N(\mathfrak{f})^{-1/2}t^{-1/2-\epsilon}, \qquad \phi = \phi_0^{u_{12}(\zeta_\mathfrak{f})} \text{ or } \phi = (\phi_0^{\vee})^{u_{12}(\zeta_\mathfrak{f})}.
\end{equation*}
In fact, opening $I_t(\phi^{u_{12}(\zeta_\mathfrak{f})},\eta)$ and applying  Cauchy-Schwarz inequality we get \begin{align*}\vert{I_t}(\phi,\eta)\vert^2&\ll t^{-1}N(\mathfrak{f})\sum_{I \in \mathcal{S}_{\cl}}\sum_{\gamma \in F^{\times}}\vert{W_{\fin}^{\circ}(a_1(\gamma I))}\vert^2\left\vert{\int_{\widehat{O}_1^\times(\mathfrak{f})E}W_{\infty}(a_1(\gamma z_{N(I)^{-1}t} y))\eta(y_{\infty}) \mathrm{d}^{\times}y}\right\vert
\\&\sum_{I\in \mathcal{S}_{\cl}}\sum_{\substack{\gamma \in F^{\times}\\(\gamma) I\subset O_F}}\left\vert{\int_{\widehat{O}_1^\times(\mathfrak{f})E}W_{\infty}(a_1(\gamma z_{N(I)^{-1}t}y))\eta(y_\infty)  \ \mathrm{d}^{\times}y}\right\vert.\end{align*}
If the integrand $W_{\infty}(a_1(\gamma z_{N(I)^{-1}t}y))\neq 0$, then for all $v\in S_{\infty}$ it holds that $\vert{\gamma}\vert_v \asymp (N(I)t^{-1})^{\frac{1}{d_v}}$. By trivially estimating the contribution from the archimedean integral we get then 
\begin{equation*}
    \vert{I_t(\phi,\eta)}\vert^2 \ll t^{-2-\epsilon}N(\mathfrak{f})^{-1}\sum_{\substack{\mathfrak{a}\subset O_F\\N(\mathfrak{a})\asymp t^{-1}}}\vert{W^{\circ}_{\fin}(a_1(\mathfrak{a})}\vert^2  \ll N(\mathfrak{f})^{-1}t^{-1-\epsilon}.
\end{equation*}

\begin{proposition}[Dyadic partition of unity]Let $\mathfrak{f} \subset O_F$ be an integral ideal and $\eta \in \widehat{F^{\times}\setminus\A^{\times}}$ be a finite order character of conductor $\mathfrak{f}$.
Then there exists an inert function $g\in C_c^{\infty}(\R_{>0})$ so that 
\begin{align*}&I(\phi_0^{u_{12}(\zeta_\mathfrak{f})},\eta,1/2) \ll \log(N(\mathfrak{f}))\\&\sup_{N(\mathfrak{f})^{-\frac{3}{2}-\epsilon}\ll T\ll N(\mathfrak{f})^{-\frac{3}{2}+\kappa}}\left\vert\int_{0}^{\infty}g\left(\frac{t}{T}\right)h_0(t)I_t(\phi_0^{u_{12}(\zeta_f)},\eta)\ \mathrm{d}^{\times}t \right\vert+ \\&\sup_{N(\mathfrak{f})^{-\frac{3}{2}-\epsilon}\ll T\ll N(\mathfrak{f})^{-\frac{3}{2}+\kappa}} \left\vert \int_{0}^{\infty}g\left(\frac{t}{T}\right)k_0(t)I_{t}((\phi_0^{\vee})^{u_{12}(\zeta_\mathfrak{f})},\overline{\eta})\ \mathrm{d}^{\times}t\right\vert +O(N(\mathfrak{f})^{1/4-\kappa/2+\epsilon})  
\end{align*}
\end{proposition}
\begin{proof} We apply a dyadic partition of unity with an inert function $g\colon (0,\infty)\to \R$.

Let $0<T< N(\mathfrak f)^{-3/2-\epsilon}$, then for any $\sigma >0$
\[\int_{0}^{\infty}g\left(\frac{t}{T}\right)k_0(t)I_t((\phi_0^{\vee})^{u_{12}(\zeta_\mathfrak{f})},\overline \eta)\ \mathrm{d}^{\times}t\ll_{\sigma} (N(\mathfrak{f})^{3/2}T)^{\sigma}T^{-1/2+\epsilon}N(\mathfrak{f})^{-1/2}, \]
which is, choosing $\sigma$ very big, as small as we want. Let $T>N(\mathfrak{q})^{-3/2+\kappa}$. Then we have
\[\int_{0}^{\infty}g\left(\frac{t}{T}\right)k_0(t)I_t((\phi_0^{\vee})^{u_{12}(\zeta_\mathfrak{f})},\overline \eta) \ \mathrm{d}^{\times}t \ll T^{-1/2+\epsilon}N(\mathfrak{f})^{-1/2} \leq N(\mathfrak{f})^{3/4-\frac{\kappa}{2}+\epsilon}. \]
\end{proof}
Next, consider our case of interest where $\mathfrak{f} = \mathfrak{q}$ is a prime ideal and $\eta = \chi$.
Using the notation of Section \ref{sec:keyidentity} we set
\begin{align*}\mathcal{F}_t^{\chi}(y,I) = \mathcal{F
}_t^{\phi_0^{u_{12}(\zeta_\mathfrak q)},\chi}(y,I),\ \mathcal{F}_T^{\chi} = \mathcal{F
}_T^{\phi_0^{u_{12}(\zeta_\mathfrak q)},\chi}, \\\ \mathcal O_t^{\chi}(y,I) = \mathcal{O}_t^{\phi_0^{u_{12}(\zeta_\mathfrak q)},\chi}(y,I), \ \mathcal{O}_T^{\chi} = \mathcal{O}_T^{\phi_0^{u_{12}(\zeta_\mathfrak q)},\chi}. \end{align*}
and similarly for the dual function 
\begin{align*}\mathcal{F}_t^{\vee,\chi}(y,I) = \mathcal{F
}_t^{(\phi^\vee_0){u_{12}(\zeta_\mathfrak q)},\chi}(y,I),\ \mathcal{F}_T^{\vee,\chi} = \mathcal{F
}_T^{(\phi^\vee_0)^{u_{12}(\zeta_\mathfrak q)},\chi},\\ \ \mathcal O_t^{\vee,\chi}(y,I) = \mathcal{O}_t^{(\phi^\vee_0)^{u_{12}(\zeta_\mathfrak q)},\chi}(y,I), \ \mathcal{O}_T^{\vee,\chi} = \mathcal{O}_T^{(\phi^{\vee}_0)^{u_{12}(\zeta_\mathfrak q)},\chi}. \end{align*}
We say that $T>0$ is in the significant range if \begin{equation}\label{eq:significant-range}N(\mathfrak{q})^{-\frac{3}{2}-\epsilon}\ll T\ll N(\mathfrak{q})^{-\frac{3}{2}+\kappa}.\end{equation}
To recap, in this section we proved that we can reformulate the subconvexity bound in terms of $\mathcal{F}_T^{\chi}$, $\mathcal{F}_T^{\vee,\chi}$, $\mathcal{O}_T^{\chi}$ and $\mathcal{O}_T^{\vee,\chi}$.
\begin{proposition}\label{prop:splitFandO}
  Let $\mathfrak{q} \subset O_F$ be a prime ideal coprime to $\mathfrak{C}(\pi)$ and $\chi \in \widehat {F^{\times}\setminus A^\times}$ a finite order character of conductor $\mathfrak{q}$. Then 
\begin{align*}
    L\left(\pi \otimes \chi,\frac{1}{2}\right) \ll N(\mathfrak{q})\bigg(&\operatorname{sup}_{N(\mathfrak{   q})^{-\frac{3}{2}-\epsilon}\ll T\ll N(\mathfrak{q})^{-\frac{3}{2}+\kappa}}\vert \mathcal{F}_T^{\chi} \vert + \vert \mathcal{O}_T^{\chi} \vert + \vert\mathcal{F}_T^{\vee,\chi}\vert + \vert \mathcal{O}_T^{\vee,\chi}\vert\bigg) . 
\end{align*}
\end{proposition}

\section{Amplification}\label{sec:amplification}
To optimize a later use of the Cauchy-Schwartz inequality, as in \cite[Section 3.3]{nelsonholowinsky}, we introduce two sets of prime ideals: let $K = (K_v)_v, L= (L_v)_v \in \R_{>0}^{d_{\infty}}\subset \A^{\times}$ be so that \[N(\mathfrak{q})^{\epsilon} \leq \vert{K}\vert,\vert{L}\vert \leq N(\mathfrak{q})^{1-\epsilon}.\]
We define
\begin{align*}
    &\mathcal{K} = \{ \mathfrak{k} \in  \operatorname{spec}(O_F)\ |\ \mathfrak{k} \text{ principal}, N(\mathfrak{k})\asymp \vert{K}\vert\}\\   &\mathcal{L} = \{\mathfrak{l} \in \operatorname{spec}(O_F)\ | \ \mathfrak{l} \text{ principal},  N(l)\asymp \vert{L}\vert\}.
\end{align*} 
Assume that either $\vert{K}\vert \gg N(\mathfrak{q})^{\epsilon}\vert{L}\vert$ or $\vert{L}\vert\gg N(\mathfrak{q})^{\epsilon}\vert{K}\vert$, and so $\mathcal{L}\cap \mathcal{K} = \varnothing$. 

For each $\mathfrak{k}\in \mathcal{K}$ and $\mathfrak{l} \in \mathcal{L}$ we use Lemma \ref{lem:choicegenerator} to choose generators $k,l$ so that 
\begin{align*}\vert{K_v}\vert_v N(\mathfrak{q})^{-\epsilon}&\ll\vert{k}\vert_v \ll \vert{K_v}\vert_v N(\mathfrak{q})^{\epsilon}, \\
\vert{L_v}\vert_v N(\mathfrak{q})^{-\epsilon}&\ll\vert{l}\vert_v \ll \vert{L_v}\vert_vN(\mathfrak{q})^{\epsilon}, \qquad \forall v \in S_{\infty}.\end{align*}

Let $(b_k)_{k\in F^{\times}}$ and $(c_l)_{l\in F^{\times}}$ be sequences in $\CC$ so that if $b_k \neq 0$ (resp. $c_l \neq 0$), then $k$ is one of the generators of one of the $\mathfrak{k} \in \mathcal{K}$ we have chosen above (resp. $l$ is one of the generators we fixed above). Also, assume that $b_k \ll N(\mathfrak{q})^{\epsilon}\vert{K}\vert^{-1}, c_l\ll N(\mathfrak{q})^{\epsilon}\vert{L}\vert^{-1}$ and
\[\sum_{k}b_k\overline{\chi}(\iota_{\mathfrak{q}}(k)) =  1 = \sum_{l}c_l\chi(\iota_{\mathfrak{q}}(l)).\]
We will study the amplified versions of $\mathcal{F}^{\chi}_t$ \eqref{eq:ftphi} and $\mathcal{O}^{\chi}_t$ \eqref{eq:otphi}, for $y\in \widehat{O}^\times_1(\mathfrak{q}) E$ (see \eqref{eq:definitionE}) and $I \in \mathcal{S}_{\cl}$ (see Subsection \ref{subsec:2.choice of fundamental domain}) we have
\begin{align}
    \mathcal{F}_t^{\chi}(y,I) 
    &={N(\mathfrak{q})^{\frac{1}{2}}}\sum_{\delta,k,l}a_{\delta}b_kc_l\mathbb{P}\phi_0^{u_{12}(\zeta_{\mathfrak{q}})}(a_1(I\iota_{\mathfrak{q}}(k(l\delta)^{-1})yz_t)\chi(y_{\infty})  \\
    \mathcal{O}^{\chi}_t(y,I)&=\frac{1}{N(\mathfrak{q})^{\frac{1}{2}}}\sum_{\substack{\delta,k,l\\ \forall v \in S_{\infty}:\\\vert{\delta}\vert_v \ll \vert{\Delta}\vert_vN(\mathfrak{q})^{\epsilon}}}b_kc_l\widehat{V}\left(\delta\varpi_\mathfrak{q}\Delta\right)\sum_{a\in F^{\mathfrak{q}}/F^{\mathfrak{q,1}}}\psi_{v_{\mathfrak{q}}}\left(\frac{a\delta}{\varpi_{\mathfrak{q}}}\right)\chi(\iota_{\mathfrak{q}}(k(al)^{-1}))\nonumber\\&\times\mathbb{P}\phi_0^{u_{12}(\zeta_{\mathfrak{q}})}(a_1(I\iota_{\mathfrak{q}}(k(l\delta)^{-1})yz_t)\chi(y_{\infty}) 
\end{align}
and prove the following two propositions.

\begin{proposition}\label{prop:boundF}
     Let $T$ be in the significant range \eqref{eq:significant-range} and suppose that for all $v\in S_{\infty}$ we have \begin{equation}\label{eq:choicedelta_v}K_vT^{-\frac{1}{d_v}}(\vert \alpha_\mathfrak{q}\vert_v\Delta_vL_v)^{-1} \ll N(\mathfrak{q})^{-4\epsilon}.\end{equation} Then 
    \[\mathcal{F}^{\chi}_T \ll N(\mathfrak{q})^{-1/2+\epsilon}\vert{\Delta L}\vert^{3/2}T^{1/2}\left(\frac{1}{\vert{LK}\vert} + \frac{1}{\vert\Delta\vert^{1/2} \vert L\vert}+\frac{T^{-1}}{\vert\Delta\vert^{5/2}\vert L\vert ^2}\right)^{1/2}. \]
    
\end{proposition}

\begin{proposition}\label{prop:boundO}
     Let $T$ be in the significant range \eqref{eq:significant-range} and suppose that 
     \[ \frac{\vert{KL}\vert}{\vert \Delta \vert} \ll N(\mathfrak{q})^{-3\epsilon}.\] Then we have
     \[\vert{\mathcal{O}_T^{\chi}}\vert \ll N(\mathfrak{q})^{-1/2+\epsilon}T^{-1/2}\vert{\Delta KL}\vert^{-1/2}.\]
\end{proposition}

The bounds are totally analogous to \cite[Proposition 1, Proposition 2]{nelsonholowinsky}, with a different normalization, in \emph{loc. cit.} $N$ corresponds to our $T^{-1}$. In \emph{loc. cit.}, it is explained that the parameters $$\vert{L}\vert = N(\mathfrak{q})^{2/18},\quad \vert K\vert = N(\mathfrak{q})^{5/18},\quad \vert \Delta \vert = \frac{\vert K \vert T^{-1}}{N(\mathfrak{q})^{1+4\epsilon}\vert L \vert }$$ optimize the bounds and one has
    \begin{equation}\label{eq:finalequation}\mathcal{F}_T^{\chi} \ll N(\mathfrak{q})^{-1+\epsilon} N(\mathfrak{q})^{\frac{3}{4}-\frac{\kappa}{2}},\qquad \mathcal{O}_T^{\chi} \ll N(\mathfrak{q})^{-1+\epsilon}N(\mathfrak{q})^{\frac{3}{4}+(\frac{1}{4}-\frac{5}{18})} = N(\mathfrak{q})^{-1+\frac{3}{4}-\frac{1}{36}+\epsilon}. \end{equation}
    Hence, applying Proposition \ref{prop:splitFandO} we get the desired result.

\section{Bounds for $\mathcal{F}^{\chi}_T$}\label{sec:Bounds for FT}

The goal of this section is to prove the following proposition.

\begin{proposition}\label{prop:boundf} Let $T$ be in the significant range \eqref{eq:significant-range}.
Suppose that for all $v\in S_{\infty}$ we have \begin{equation}\label{eq:choicedelta_v2}K_vT^{-\frac{1}{d_v}}(\vert \alpha_\mathfrak{q}\vert_v\Delta_vL_v)^{-1} \ll 1,\end{equation}
then for any $I\in \mathcal{S}_{\cl},$ $y \in \widehat{O}_1^{\times}(\mathfrak{q})E$ and $t\asymp T$ we have the following. 
\[\vert{\mathcal{F}_t^\chi(y,I)}\vert^2 \ll N(\mathfrak{q})^{1+\epsilon}\vert{\Delta L}\vert^3t\left(\frac{1}{\vert{LK}\vert} + \frac{1}{\vert\Delta\vert^{1/2} \vert L\vert}+\frac{t^{-1}}{\vert\Delta\vert^{5/2}\vert L\vert ^2}\right).  \]
\end{proposition}

From the above Proposition, we deduce Proposition \ref{prop:boundF} by trivially estimates using the triangle inequality.

To prove Proposition \ref{prop:boundF} we fix $I\in \mathcal{S}_{\cl}$, $o(y)\in \widehat{O}_1^{\times}(\mathfrak{q})$, $y_{\infty}\in E$ (see \eqref{eq:definitionE}) and $t>0$ and we suppose that $\Delta, K$ and $L$ satisfy \eqref{eq:choicedelta_v}. Note that the case of interest is when $t  \asymp T \asymp N(\mathfrak{q})^{\frac{3}{2}}$ and $T^{-\frac{1}{d_{v}}} \asymp N(\mathfrak{q})^{\frac{3}{2d_{v}}}$. In particular, in this case, we require that $\Delta_v\gg K_vL_v^{-1}N(\mathfrak{q})^{\frac{3}{2d_{v}}}$ for each $v\in S_{\infty}$.

\subsection{The conductor lowering mechanism}
We expand $\mathbb{P}\phi$ using the Fourier expansion: for $k,l,\delta \in F^\times$ we have
\begin{align}\label{eq:10.1.1}
    &\mathbb{P}\phi_0^{u_{12}(\zeta_{\mathfrak{q}})}\left(a_1(Io(y)y_{\infty}z_t\iota_{\mathfrak{q}}(k(l\delta)^{-1}))\right) =     \nonumber\\    &\frac{1}{t^{\frac{1}{2}}}\sum_{\gamma \in F^{\times}}W_{\phi_0}\left(a_1(\gamma Io(y)y_{\infty}z_t\iota_{\mathfrak{q}}(k(l\delta^{-1}))\right)\psi(\gamma Io(y)y_{\infty}z_t\iota_\mathfrak{q}( k(l\delta)^{-1})\zeta_{\mathfrak{q}}).
\end{align}
By the properties of the selected Whittaker function, only those $\gamma \in I^{-1}O_F \subset O_F$\ contribute to the sum.\footnote{Recall the specific choice of the $I\in \mathcal{S}_{\cl}$, that is, if $I_v\neq 1$, then $v(I)<0$.}
We have
\begin{align*}\psi(\gamma Io(y)y_{\infty}z_t \iota_{\mathfrak{q}}(k(l\delta)^{-1})\zeta_{\mathfrak{q}})=
\psi_{v_{\mathfrak{q}}}(\gamma k(l\delta)^{-1}\alpha_{\mathfrak{q}}^{-1}).
\end{align*}
Also, for such a $\gamma \in O_F$ we have that:
\begin{align*}
    1 &= \psi(\gamma k(l\delta\alpha_{\mathfrak{q}})^{-1}) \\&= \psi_{\infty}(\gamma k(l\delta\alpha_{\mathfrak{q}})^{-1})\prod_{\substack{v<\infty\\v|l\delta \alpha_{\mathfrak{q}}\\v\neq v_{\mathfrak{q}}}}\psi_v(\gamma k(l\delta\alpha_{\mathfrak{q}})^{-1}) \times \psi_{v_{\mathfrak{q}}}(\gamma k(l\delta\alpha_{\mathfrak{q}})^{-1})
\end{align*}
We rewrite the sum on the right-hand side of \eqref{eq:10.1.1} as 
\begin{align*}
    &\frac{1}{t^{\frac{1}{2}}}\sum_{\gamma \in F^{\times}}W_{\fin}^{\phi_0}(a_1(\gamma I))\psi(-\gamma k \zeta_{k,l,\delta,\alpha_{\mathfrak{q}}})\prod_{v \in S_{\infty}} W^{\phi_0}_v(a_1(\gamma y_v (tN(I))^{\frac{1}{d_v}}) )\psi_v\left(-\gamma k(l\delta\alpha_{\mathfrak{q}})^{-1}\right),
\end{align*}
where $\zeta_{k,l,\delta,\alpha_{\mathfrak{q}}} \in \A_{\fin}$ is defined as follows:
\begin{align*}
    (\zeta_{k,l,\delta,\alpha_{\mathfrak{q}}})_v &= \begin{cases}
        0 & \text{if }v\nmid l\delta\alpha_{\mathfrak{q}} \text{ or } v= v_{\mathfrak{q}}\\
        k(l\delta \alpha_{\mathfrak{q}})^{-1}&\text{otherwise.}
    \end{cases}
\end{align*}
From this we obtain the identity
\begin{align*}\label{eq:fty} 
\mathcal{F}_t^{\chi}({y,I})=& (t^{-1}N(\mathfrak{q}))^{1/2}\sum_{\delta,l,k}a_{\delta}b_kc_l \sum_{\gamma \in F^{\times}}W^{\phi_0}_{\fin}(a_1(\gamma I))\psi(-\gamma \zeta_{k,l,\delta,\alpha_\mathfrak{q}})\\&\times \prod_{v\in S_{\infty}} W_v((N(I)t)^{\frac{1}{d_v}}y_v\gamma) \psi_v\left(-\gamma k(l\delta\alpha_{\mathfrak{q}})^{-1}\right).
\end{align*}

\begin{lemma} For each $v\in S_{\infty}$, $I \in \mathcal{S}_{\cl}$, $y\in E$, the function \begin{equation*}
    W_v^{k,l,\delta}(x) \coloneqq W^{\phi_0}_v(x)\psi_v(-xk(N(I)t)^{-\frac{1}{d_v}}(y_vl\delta\alpha_{\mathfrak{q}})^{-1}). 
\end{equation*} defined above is \emph{inert} in the sense of Definition \ref{def:inert}. 
\end{lemma}
\begin{proof}
    Since we have chosen $W_v$ to be inert it suffices to check that $\psi_v$ is not oscillating on the interval where $W_v$ is not $0$. Suppose $v$ is complex, then $\psi_v(z) = e^{-2\pi i\operatorname{tr}_{\CC/\R}(z)} = e^{-4\pi i\real(z)}$. Write $z \in \CC$, $z = \vert z\vert_v^{1/2}e^{i\theta_z}$. Then we have 
    \begin{align*}
        &\psi_v(-zkt^{-\frac{1}{d_v}}(y_vl\delta\alpha_{\mathfrak{q}})^{-1}) = e^{4\pi i \vert zk{y}^{-1}{(l\delta\alpha_\mathfrak{q})}^{-1}\vert_v^{1/2} t^{-\frac{1}{2d_\infty}}\cos(\theta_{z}-\theta_y-\theta_{\varpi_\mathfrak{q}}-\theta_{\delta})}.
    \end{align*}
    Taking derivatives and recalling the bound \eqref{eq:choicedelta_v} we see that the function is not oscillating.
\end{proof}

\subsection{Application of Voronoi summation formula}

The next step is the Voronoi summation formula. Let $\delta,k,l$ be so that $a_{\delta},b_k,c_l$ are non-zero. Consider the sum
\begin{align}\label{eq:10.2.1}
    &\sum_{\gamma \in F^{\times}}W^{\phi_0}_{\fin}(a_1(\gamma I))\psi(-\gamma \zeta_{k,l,\delta,\alpha_{\mathfrak{q}}})\prod_{v \in S_{\infty}}W^{k,l,\delta}_{v}(t^{\frac{1}{d_v}}y_v\gamma) = \nonumber\\
    &\sum_{\gamma \in F^{\times}}W^{\circ}_{\fin}\left(a(\mathfrak{D}^{-1},\gamma \mathfrak{D}^{-1}I)\right)\psi(-\gamma \zeta_{k,l,\delta,\alpha_{\mathfrak{q}}})\prod_{v \in S_{\infty}}W^{k,l,\delta}_{v}(t^{\frac{1}{d_v}}y_v\gamma)
\end{align}

We apply Voronoi summation formula Proposition~\ref{thm:voronoi1} to \eqref{eq:10.2.1} and get,\footnote{By construction $(l\delta,\mathfrak{C}(\pi))=1$.}  with $$R=\{v<\infty,\ \vert{\zeta_{k,l,\delta,\alpha_{\mathfrak{q}}} \mathfrak{D} I^{-1}}\vert_v>1\},$$ up to terms of absolute value $1$ that we ignore,
\begin{align}\label{eq:10.2.2}
   \frac{\vert{\zeta_{k,l,\delta,\alpha_\mathfrak{q}} I\mathfrak{D}^{-1}}\vert_{R}N(\mathfrak{D})^{3/2}C(\pi)^{\frac{1}{2}}}{N(I)}\sum_{\gamma \in F^{\times}}K_{R}(\gamma,\zeta, \mathfrak{D}^{-1},\mathfrak{D}^{-1}I, U^{\circ}_R)U^{\circ R}(a(\mathfrak{D}^{-1},\gamma I^{-1}\mathfrak{D}^2\mathfrak{C}(\pi))) \nonumber\\
   \times\prod_{v\in S_{\infty}}B_{\pi_v}[W^{k,l,\delta}_v(t^{\frac{1}{d_v}}y_v\bullet)](\gamma). \end{align}To unfold the above sum, we define the following objects \begin{itemize}
       \item  $\mathfrak{u} = (\alpha_{\mathfrak{q}})\mathfrak{q}^{-1}\subset O_F$ (see \eqref{def:alpha}),
       \item  $\textswab{d}$ the ideal generated by $\delta$,
       \item $\mathfrak{z}_{k,l,\delta}=\prod_{v\in R}\mathfrak{p}_v^{-v(\zeta_{k,l,\delta,\alpha_{\mathfrak{q}}}\mathfrak{D}I^{-1})}=\textswab{d}\mathfrak{ul}\left(\textswab{d}\mathfrak{ul},\mathfrak{kD}I^{-1}\right)^{-1}$.\footnote{We are suppressing from the notation the dependency on $\mathfrak{u}$, $\mathfrak{D}$ and $I$ since they have absolutely bounded norm.}
       \item $J=\prod_{v\in R}\varpi_v^{v(\alpha_{\mathfrak{q}}I\mathfrak{D}^{-1})} = (l\frac{\delta}{(\delta,k)})^{-1}\mathfrak{z}_{k,l,\delta} \in \A^{\times}$.
   \end{itemize}   
   
   Then the sum \eqref{eq:10.2.2} can be written as follows (in what follows $M$ is a constant that depends at most polynomially on $I,C(\pi)$ and $N(\mathfrak{D})$ and it might change from line to line). 
    \begin{align*}
        &M\cdot N(\mathfrak{z}_{k,l,\delta})\sum_{\substack{\mathfrak{d}\subset O_F\\\mathfrak{d}|\mathfrak{z}_{k,l,\delta}}}N(\mathfrak{d})^{1/2}\sum_{\gamma \in F^{\times}}Kl(\gamma\zeta_{k,l,\delta,\alpha_{\mathfrak{q}}}^{-1}\mathfrak{D}^2;\mathfrak{d},\psi^{\circ})\\
    &\times U_{\fin}^{\circ}\left(a\left(\frac{\mathfrak{z}_{k,l,\delta}}{\mathfrak{d}}\mathfrak{D}^{-1},\gamma (\frac{\mathfrak{z}_{k,l,\delta}}{\mathfrak{d}})^{-2}\mathfrak{z}_{k,l,\delta}^3I^{-1}\mathfrak{D}^2\mathfrak{C}(\pi)\right)\right)\prod_{v\in S_{\infty}}B_{\pi_v}[W_v^{k,l,\delta}(t^{\frac{1}{d_v}}y_v\bullet)](\gamma)
\end{align*}
where for $\mathfrak{d}\subset O_F$: \begin{align*}N(\mathfrak{d})^{1/2}Kl(\gamma\zeta_{k,l,\delta,\alpha_{\mathfrak{q}}}^{-1}\mathfrak{D}^2;\mathfrak{d},\psi^{\circ}) &\coloneqq N(\mathfrak{d})^{1/2}\prod_{v(\mathfrak{d})>0}Kl(\gamma(\zeta_{k,l,\delta,\alpha_{\mathfrak{q}}}^{-1}\mathfrak{D}^2)_v;{v(\mathfrak{d})},\psi^{\circ}_v)  \\&= \prod_{v(\mathfrak{d})>0}S(\varpi_v^{-v(\mathfrak{D})},\gamma(\zeta_{k,l,\delta,\alpha_{\mathfrak{q}}}^{-1})_v\varpi_v^{v(\mathfrak{D})};{v(\mathfrak{d})},\psi_v),
\end{align*}
Recall that $\psi^{\circ} = \psi(\mathfrak{D}^{-1}\cdot)$.
From now on, to make notation less heavy, we write $Kl_v(x;{\mathfrak{d}},\psi^{\circ}),$ omitting the $v$'s in the arguments $x$ and $\mathfrak{d}$. 
After changing the variables $\mathfrak{d}\leftrightarrow\frac{\mathfrak{z}_{k,l,\delta}}{\mathfrak{d}}$ we get 
\begin{align*}
        &M\cdot N(\mathfrak{z}_{k,l,\delta})\sum_{\substack{\mathfrak{d}\subset O_F\\\mathfrak{d}|\mathfrak{z}_{k,l,\delta}}}\frac{N(\mathfrak{z}_{k,l,\delta})^{1/2}}{N(\mathfrak{d})^{1/2}}\left(\frac{N(\mathfrak{d})}{N(\mathfrak{z}_{k,l,\delta})^3}\right)^{1/2}\sum_{\gamma \in F^{\times}}\frac{1}{N(\gamma)^{1/2}}Kl(\gamma\zeta_{k,l,\delta,\alpha_{\mathfrak{q}}}^{-1}\mathfrak{D}^2;\frac{\mathfrak{z}_{k,l,\delta}}{\mathfrak{d}},\psi^{\circ})\nonumber\\
    &\times U_{\fin}^{\circ}(a(\mathfrak{d}\mathfrak{D}^{-1},\gamma \mathfrak{d}^{-2}\mathfrak{z}_{k,l,\delta}^3I^{-1}\mathfrak{D}^2\mathfrak{C}(\pi))\left(N(\gamma\mathfrak{d}^{-2}\mathfrak{z}_{k,l,\delta}^3)N(\mathfrak{d})\right)^{1/2}\prod_{v\in S_{\infty}}B_{\pi_v}[W_v^{k,l,\delta}(t^{\frac{1}{d_v}}y_v\bullet)]\left(\gamma\right).
\end{align*}
Let us denote $(\delta,k) = k$ if $k|\delta$ and $1$ if $k\nmid \delta$. We get finally, with the change of variables $\gamma \leftrightarrow \gamma(\frac{l\delta}{(\delta,k)})^{3}$:
      \begin{align*}
           &M\cdot N(\mathfrak{z}_{k,l,\delta})\sum_{\substack{\mathfrak{d}\subset O_F\\\mathfrak{d}|\mathfrak{z}_{k,l,\delta}}}\frac{N(\mathfrak{z}_{k,l,\delta})^{1/2}}{N(\mathfrak{d})^{1/2}}\left(\frac{N(\mathfrak{d})}{N(J)^3}\right)^{1/2}\sum_{\gamma \in F^{\times}}\frac{1}{N(\gamma)^{1/2}}Kl\left(\gamma\frac{(\delta,k)^3}{k(l\delta)^{2}}\alpha_{\mathfrak{q}}\mathfrak{D}^2;\frac{\mathfrak{z}_{k,l,\delta}}{\mathfrak{d}},\psi^{\circ}\right)\nonumber\\
    &\times U_{\fin}^{\circ}(a(\mathfrak{d}\mathfrak{D}^{-1},\gamma \mathfrak{d}^{-2}J^3I^{-1}\mathfrak{D}^2\mathfrak{C}(\pi))N(\gamma\mathfrak{d}^{-1}J^3)^{1/2}\prod_{v\in S_{\infty}}B_{\pi_v}[W_v^{k,l,\delta}(t^{\frac{1}{d_v}}y_v\bullet)]\left(\gamma(\frac{l\delta}{(\delta,k)})^{-3}\right).
      \end{align*}
By the properties of the newform $U_{v}^{\circ}$, $v<\infty$, the sums above, given $\mathfrak{d}|\mathfrak{z}_{k,l,\delta}$, are actually restricted to those $\gamma$ that satisfy \begin{equation*}
    v(\gamma) \geq v(\mathfrak{d}^{2}J^{-3}I\mathfrak{D}^{-3}\mathfrak{C}(\pi)^{-1})\qquad \forall v<\infty,\end{equation*}
We define so, for each $v < \infty$ and $\mathfrak{d}|\mathfrak{z}_{k,l,\delta}$, the function $f_v^{\mathfrak{d}} =\mathbb{1}_{\mathfrak{d}^{-2}J^{3}I\mathfrak{D}^{-3}\mathfrak{C}(\pi)^{-1}O_v}$.
Finally, we have \begin{align*}
    &B_{\pi_v}[W_v^{k,l,\delta}(t^{\frac{1}{d_v}}y_v\bullet)](x) = 
 \vert{t^{\frac{1}{d_{v}}}y_v}\vert_v B_{\pi_v}[W^{k,l,\delta}_{v}](t^{-\frac{1}{d_v}}y_v^{-1}x)\end{align*}
 and by \eqref{eq:decaybesselforinert} the latter is negligible small whenever $\vert{x}\vert_v \gg N(\mathfrak{q})^{\frac{\epsilon}{d_{\infty}}}T^{\frac{1}{d_v}}$.
After the above manipulations, we have 
\begin{align}\label{eq:10.2.3}
    \mathcal{F}^{\chi}_t(y,I) &=  M\cdot (tN(\mathfrak{q}))^{1/2}\sum_{\substack{\gamma,\delta \in F^{\times}\\\mathfrak{d}\subset O_F}}\sum_{\substack{k,l\\ \mathfrak{d}|\mathfrak{z}_{k,l,\delta}}}\frac{a_{\delta}b_kc_l}{N(\gamma)^{1/2}} \nonumber\\& \times U_{\fin}^{\circ}(a(\mathfrak{d}\mathfrak{D}^{-1},\gamma \mathfrak{d}^{-2}J^3I^{-1}\mathfrak{D}^2\mathfrak{C}(\pi))N(\gamma\mathfrak{d}^{-1}J^3)^{1/2} \nonumber\\&\times N(\mathfrak{z}_{k,l,\delta})^{3/2}Kl\left(\gamma\frac{(\delta,k)^3}{k(l\delta)^{2}}\alpha_{\mathfrak{q}}\mathfrak{D}^2;\frac{\mathfrak{z}_{k,l,\delta}}{\mathfrak{d}},\psi^{\circ}\right)\prod_{v\in S_{\infty}}B_{\pi_v}[W_v^{k,l,\delta}]\left(t^{-\frac{1}{d_v}}y_v^{-1}\gamma\left(\frac{l\delta}{(\delta,k)}\right)^{-3}\right).
\end{align}

\subsection{Cauchy--Schwarz}
Applying the Cauchy--Schwarz inequality to \eqref{eq:10.2.3}, we have

\begin{align*}\label{eq:FTcauchyschwartz}
    &\vert{\mathcal{F}^{\chi}_t(y,I)}\vert^2 \ll tN(\mathfrak{q})\nonumber\\&  \times\sum_{\substack{\gamma,\delta,\mathfrak{d}\\\\ N(\mathfrak{d})\ll \vert{\Delta L}\vert^3\\\forall v\in S_{\infty}:\\\vert{\gamma(\delta l)^{-3}}\vert_v\ll N(\mathfrak{q})^{\epsilon/d_{\infty}}T^{\frac{1}{d_v}}\\N(\gamma\mathfrak{d}^{-2})\ll N(\mathfrak{q})^{\epsilon}N(\delta l)^3T}} \vert{a_{\delta}}\vert\left\vert{U_{\fin}^{\circ}(a(\mathfrak{d}\mathfrak{D}^{-1},\gamma \mathfrak{d}^{-2}J^3I^{-1}\mathfrak{D}^2\mathfrak{C}(\pi))}\right\vert^2N(\mathfrak{d}) N(\gamma\mathfrak{d}^{-2})\nonumber\\ &
    \times \sum_{\mathfrak{d}\subset O_F}\sum_{\gamma,\delta}\frac{\vert{a_{\delta}}\vert}{N(\gamma)}\left\vert{\sum_{\substack{k,l\\\mathfrak{d}|\mathfrak{z}_{k,l,\delta}}}N(\mathfrak{z}_{k,l,\delta})^{3/2}b_kc_lC(\gamma,k,l,\delta,\mathfrak{d})}\right\vert^2 + O(N(\mathfrak{q})^{o(1)}).
\end{align*}
With \[C(\gamma,k,l,\delta,\mathfrak{d})=Kl\left(\gamma \frac{(\delta,k)^3}{k({l\delta})^{2}}\alpha_{\mathfrak{q}}\mathfrak{D}^2;\frac{\mathfrak{z}_{k,l,\delta}}{\mathfrak{d}};\psi^{\circ}\right)V_{k,l,\delta}\left(\gamma z_{t^{-1}}y_{\infty}^{-1}\tilde{T}_{\delta,k,l}\right),\]
where $\tilde{T}_{\delta,k,l}=\left(\left(\frac{l\delta}{(\delta,k)}\right)^{3}\right)_{v\in S_{\infty}} \subset F_{\infty}^{\times} \subset \A^{\times}$ is an idèle, trivial in the finite places, and
\[V_{k,l,\delta}(x) = \prod_{v<\infty}f_v^{\mathfrak{d}}(x_v)\times \prod_{v\in S_{\infty}}B_{\pi_v}[W_v^{k,l,\delta}](x_v), \qquad x \in \A^{\times}.\] 
By the Rankin--Selberg bound \eqref{eq:classicrankinselberg} and Lemma \ref{lem:countingunits} we have 
\begin{align*}\sum_{\substack{\gamma,\delta,\mathfrak{d}\\\\ N(\mathfrak{d})\ll \vert{\Delta L}\vert^3\\\forall v\in S_{\infty}:\\\vert{\gamma(\delta l)^{-3}}\vert_v\ll N(\mathfrak{q})^{\epsilon/d_{\infty}}T^{\frac{1}{d_v}}\\N(\gamma\mathfrak{d}^{-2})\ll N(\mathfrak{q})^{\epsilon}N(\delta l)^3T}} \vert{a_{\delta}}\vert\left\vert{U_{\fin}^{\circ}(a(\mathfrak{d}\mathfrak{D}^{-1},\gamma \mathfrak{d}^{-2}J^3I^{-1}\mathfrak{D}^2\mathfrak{C}(\pi))}\right\vert^2N(\gamma\mathfrak{d}^{-2})N(\mathfrak{d}) 
\ll N(\mathfrak{q})^{\epsilon}.
\end{align*}
Hence,
\begin{equation}\label{eq:10.3.1}\vert \mathcal{F}_t^\chi(y,I)\vert ^2 \ll tN(\mathfrak{q})^{1+\epsilon}\mathcal{F}_2 + O(N(\mathfrak{q})^{o(1)}),  \end{equation}
where
\begin{equation*}
    \mathcal{F}_2 = \sum_{\mathfrak{d}\subset O_F}\sum_{\gamma,\delta}\frac{\vert{a_{\delta}}\vert}{N(\gamma)}\left\vert{\sum_{\substack{k,l\\\mathfrak{d}|\mathfrak{z}_{k,l,\delta}}}N(\mathfrak{z}_{k,l,\delta})^{3/2}b_kc_lC(\gamma,k,l,\delta,\mathfrak{d})}\right\vert^2.
\end{equation*}
 Denote for each $v\in S_{\infty}$: \[\tilde{T}_v= (t^{\frac{1}{d_v}}y_v)\delta^3(\frac{l_1l_2}{(\delta,k_1)(\delta,k_2)})^{3/2} \] and let $\tilde{T} = (\tilde{T}_v)_{v\in S_{\infty}}, \tilde{T_1} = \tilde{T}_{\delta,k_1,l_1}, \tilde{T_2} = \tilde{T}_{\delta,k_2,l_2}\in F_{\infty}^{\times} \subset \A^{\times}$ be the corresponding idèle.
Let \[ U(x) = \frac{1}{\vert{x}\vert_{\infty}}V_{k_1,l_1,\delta}\left(\frac{x z_{t^{-1}}y_{\infty}^{-1}\tilde{T}}{\tilde{T}_1}\right)\overline{V_{k_2,l_2,\delta}\left(\frac{xz_{t^{-1}}y_{\infty}^{-1}\tilde{T}}{\tilde{T}_2}\right)}.
\]
We apply a partition of unity for each $v\in S_{\infty}$, to write:
\begin{align*}U(x)=\prod_{v<\infty}f_v^{\mathfrak{d}}(x_v)\sum_{(X_v)_v\in 2^{\Z^{d_{\infty}}}} \prod_{v\in S_{\infty}}\min(X_v^{1-2\theta},X_v^{-17})U_{X_v}\left(\frac{x_v}{X_v}\right),\end{align*}
where each $U_{X_v} \in C_c^{\infty}(F_v^{\times})$ is inert.
For each $X = (X_v)_v$ we denote \[U_X(x) = \prod_{v<\infty}f_v^{\mathfrak{d}}(x_v)\prod_{v\in S_{\infty}}U_{X_v}(x_v)\in C_c^{\infty}(\A).\]
Opening the square and applying the partition of unity we see that $\mathcal{F}_2$ is bounded by 
\begin{align}\label{eq:finalboundF} &\vert{\tilde{T}^{-1}}\vert N(\mathfrak{q})^{\epsilon}\sum_{X}\prod_{v\in S_{\infty}}\min(\vert{X}\vert_v^{1-2\theta},\vert{X}\vert_v^{-17})\sum_{\substack{\gamma,\delta,\mathfrak{d}}}\vert{a_{\delta}}\vert\times \nonumber\nonumber\\&\sum_{\substack{k_1,k_2,l_1,l_2\\\mathfrak{d}|(\mathfrak{z}_{k_1,l_1,\delta},\mathfrak{z}_{k_2,l_2,\delta})}}N(\mathfrak{z}_{k_1,l_1,\delta}\mathfrak{z}_{k_2,l_2,\delta})^{3/2}b_{k_1}\overline{b_{k_2}}c_{l_1}\overline{c_{l_2}}\mathscr{C}(\delta,\mathfrak{d},k_1,k_2,l_1,l_2,U_X(\frac{\cdot}{\tilde{T}X})),
\end{align}
where
\begin{align*}&\mathscr{C}\left(\delta,\mathfrak{d},k_1,k_2,l_1,l_2,U_X(\frac{\cdot}{\tilde{T}X} )\right) =\\ &\sum_{\gamma \in F^{\times}}U_X\left(\frac{\gamma}{\tilde{T}X}\right)Kl\left(\gamma\frac{(\delta,k_1)^3}{k_1(l_1\delta)^2}\alpha_{\mathfrak{q}}\mathfrak{D}^2;\frac{\mathfrak{z}_{k_1,l_1,\delta}}{\mathfrak{d}},\psi^{\circ}\right)\overline{Kl\left(\gamma\frac{(\delta,k_2)^3}{k_2\left({l_2\delta}\right)^{2}}\alpha_{\mathfrak{q}}\mathfrak{D}^2;\frac{\mathfrak{z}_{k_2,l_2,\delta}}{\mathfrak{d}},\psi^{\circ}\right)}.\end{align*}
and $X$ runs through the vectors $(X_v)_{v\in S_{\infty}}$ with $X_v = 2^k$ for some $k\in \Z$.
By Lemma \ref{lem:4.3} we have 
\begin{align*}\mathscr{C}&\left(\delta,\mathfrak{d},k_1,k_2,l_1,l_2,U_{X}(\frac{\cdot}{\tilde{T}X})\right) \\&\ll  \frac{\widehat{V}(0)\vert{X}\vert  }{N([\frac{\mathfrak{z}_{k_1,l_1,\delta}}{\mathfrak{d}},\frac{\mathfrak{z}_{k_2,l_2,\delta}}{\mathfrak{d}}])^{1/2}N(\mathfrak{d})^2} N\left((\frac{\mathfrak{z}_{k_1,l_1,\delta}}{\mathfrak{d}},\frac{\mathfrak{z}_{k_1,l_1,\delta}}{\mathfrak{d}},\Omega)\right)^{1/2}  +   N(\mathfrak{q})^{\epsilon}N\left([\frac{\mathfrak{z}_{k_1,l_1,\delta}}{\mathfrak{d}},\frac{\mathfrak{z}_{k_2,l_2,\delta}}{\mathfrak{d}}]\right)^{1/2},
\end{align*}
where, up to multiplication by elements of $\widehat O^{\times}$ of norm $1$, we have \begin{align*}\Omega &= \mathfrak{d}^{2}\prod_{v\in R}\mathfrak{p}_v^{v(J^{-3}I\mathfrak{D}^{-3})}\left(\frac{\alpha_{\mathfrak{q}}(\delta,k_1)^3}{k_1(l_1\delta)^2}-\frac{\alpha_{\mathfrak{q}}(\delta,k_2)^3}{k_2(l_2\delta)^2}\right)\left(\frac{(\mathfrak{z}_{k_1,l_1,\delta},\mathfrak{z}_{k_2,l_2,\delta})}{\mathfrak{d}}\right)^2 
\\&= \prod_{v\in R}\mathfrak{p}^{v(\alpha_{\mathfrak{q}}^{-3}I^{-2})}\left(\frac{\alpha_{\mathfrak{q}}(\delta,k_1)^3}{k_1(l_1\delta)^2}-\frac{\alpha_{\mathfrak{q}}(\delta,k_2)^3}{k_2(l_2\delta)^2}\right)\left(\frac{\delta^2 \mathfrak{u}^2(\mathfrak{l}_1,\mathfrak{l}_2)^2}{(\textswab{d},[\mathfrak{k}_1,\mathfrak{k}_2])^2(\textswab{d}\mathfrak{u},I^{-1}\mathfrak{D})^2}\right)\\ &=\mathfrak{D}^{-2}\left(\frac{(l_1,l_2)}{(\delta,[k_1,k_2])}\right)^2\left(\frac{(\delta,k_1)^3}{k_1l_1^2}-\frac{(\delta,k_2)^3}{k_2l_2^2}\right). \end{align*}
By \eqref{eq:decaybesselforinert} and \eqref{eq:decaybesselinert2} the sum \eqref{eq:finalboundF} is dominated by the contribution of $X=(1)_{S_\infty}$. Combining the bounds on $\mathcal{F}_2$ and \eqref{eq:10.3.1}, we have 
\begin{align*}
    &\vert{\mathcal{F}^{\chi}_{t}(y,I)}\vert^2 \ll tN(\mathfrak{q})^{1+\epsilon} \\ &\bigg( \sum_{\delta, k_1,k_2,l_1,l_2}\vert{a_{\delta}b_{k_1}b_{k_2}c_{l_1}c_{l_2}}\vert N(\mathfrak{z}_{k_1,l_1,\delta}\mathfrak{z}_{k_2,l_2,\delta})^{3/2}\\& \sum_{\mathfrak{d}|(\mathfrak{z}_{k_1,l_2,\delta},\mathfrak{z}_{k_2,l_2,\delta})}N(\mathfrak{d})^{-3/2}\frac{N\left((\frac{\mathfrak{z}_{k_1,l_1,\delta}}{\mathfrak{d}},\frac{\mathfrak{z}_{k_2,l_2,\delta}}{\mathfrak{d}},\Omega)\right)^{1/2}}{(\frac{N(\delta)}{N(\delta,k_1,k_2)})^{1/2}N([\mathfrak{l}_1,\mathfrak{l}_2])^{1/2}} +  \frac{1}{\vert{\tilde{T}}\vert N(\mathfrak{d})^{1/2}}N\left(\frac{\delta}{(\delta,k_1,k_2)}[\mathfrak{l}_1,\mathfrak{l}_{2}]\right)^{1/2}\bigg)
\end{align*}

\subsection{Conclusion}
We first bound the summand \begin{equation}\label{eq:finalFpart1} \sum_{\delta, k_1,k_2,l_1,l_2}\vert a_\delta b_{k_1}b_{k_2}c_{l_1}c_{l_2}\vert N(\mathfrak{z}_{k_1,l_1,\delta}\mathfrak{z}_{k_2,l_2,\delta})^{3/2}\sum_{\mathfrak{d}|(\mathfrak{z}_{k_1,l_2,\delta},\mathfrak{z}_{k_2,l_2,\delta})}N(\mathfrak{d})^{-3/2}\frac{N\left((\frac{\mathfrak{z}_{k_1,l_1,\delta}}{\mathfrak{d}},\frac{\mathfrak{z}_{k_2,l_2,\delta}}{\mathfrak{d}},\Omega)\right)^{1/2}}{(\frac{N(\delta)}{N(\delta,k_1,k_2)})^{1/2}N([\mathfrak{l}_1,\mathfrak{l}_2])^{1/2}}.\end{equation}

 We claim that $\Omega =0$ if and only if $k_1=k_2$ and $l_1=l_2$. Suppose that $k_1,k_2\nmid \delta$, then $\Omega = 0$ if and only if $k_1^{-1}l_1^{-2}-k_2^{-1}l_2^{-2}=0$ and the latter happens if and only if $k_1l_1^2 = k_2l_2^2$, and according to the assumptions on $k_i,l_i$ this happens if and only if $k_1 = k_2$ and $l_1= l_2$.   
Suppose that $k_1|\delta$ and $k_2\nmid \delta$. Then $\Omega = 0$ if and only if $k_1^2l_1^{-2}-k_2^{-1}l_2^{-2}=0$ and this can never happen. Suppose $k_1,k_2|\delta$, then $\Omega = 0$ if and only if $k_1^2l_1^{-2}-k_2^2l_2^{-2} =0$ and once again this happens if and only if $k_1 = k_2$ and $l_1 = l_2$.

We decompose \eqref{eq:finalFpart1} into two cases:
\[ \Omega =0,\quad \Omega \neq 0. \]First, we consider the case $\Omega = 0$, hence $k_1=k_2=k$ and $l_1=l_2=l$. We consider two subcases $k|\delta$ and $k\nmid \delta$, respectively. The first subcase contribution is bounded
\begin{equation}\label{eq:10.52}
    \ll N(\mathfrak{q})^{\epsilon}\sum_{k,l,\delta}\vert a_\delta b_{k_1}b_{k_2}c_{l_1}c_{l_2}\vert\frac{N(\frac{\delta}{k}l)^{7/2}}{N(\delta/k)^{1/2}N(l)^{1/2}} \ll \frac{N(\mathfrak{q})^{\epsilon}}{\vert{\Delta K^2L^2}\vert}\frac{\vert{\Delta^3 L^4}\vert}{\vert{K}\vert^{2}}(\vert{\Delta K^{-1}}\vert+1),
\end{equation}
whereas the second by 
\begin{equation}\label{eq:10.53}
    \ll N(\mathfrak{q})^{\epsilon}\sum_{k,l,\delta}\vert a_\delta b_{k_1}b_{k_2}c_{l_1}c_{l_2}\vert\frac{N(\delta l)^{7/2}}{N(\delta)^{1/2}N(l)^{1/2}} \ll \frac{N(\mathfrak{q})^{\epsilon}}{\vert{\Delta K^2L^2}\vert}\vert{\Delta L}\vert^{4}\vert{K}\vert,
\end{equation}
which dominates the first one.

For $\Omega \neq 0$ we have $k_1\neq k_2$ or $l_1\neq l_2$, similarly as in \eqref{eq:10.52}, the contribution if $k=k_1=k_2|\delta$ is less than the other contribution. By the divisor bound, the contribution of this case to \eqref{eq:finalFpart1} is then bounded by 
\begin{equation}\label{eq:10.54}
    \ll \frac{N(\mathfrak{q})^{\epsilon}}{\vert{\Delta K^2L^2}\vert}\vert{\Delta L}\vert^3 \frac{\vert{\Delta K^2 L^2}\vert}{\vert{\Delta^{1/2}L}\vert} = N(\mathfrak{q})^{\epsilon}\vert\Delta\vert^{5/2}\vert L\vert^2.
\end{equation}
Whereas the second summand contributes
\begin{align}\label{eq:10.55}
\sum_{\delta,k_1,k_2,l_1,l_2}\vert a_\delta b_{k_1}b_{k_2}c_{l_1}c_{l_2}\vert N(\mathfrak{z}_{k_1,l_1,\delta}\mathfrak{z}_{k_2,l_2,\delta})^{3/2}\sum_{\mathfrak{d}|(\mathfrak{z}_{k_1,l_1,\delta},\mathfrak{z}_{k_2,l_2,\delta})}\frac{ N(\delta[\mathfrak{l}_1,\mathfrak{l}_2]/(\delta,k_1,k_2))^{1/2}}{\vert{\tilde{T}}\vert N(\mathfrak{d})^{1/2}} \nonumber\\ \ll N(\mathfrak{q})^{\epsilon}\frac{1}{\vert{\Delta K^2L^2}\vert}t^{-1}\vert{\Delta^{1/2} L}\vert\vert{\Delta K ^2L^2}\vert = t^{-1}N(\mathfrak{q})^{\epsilon}\vert\Delta\vert^{1/2}\vert L\vert.
\end{align}

Combining \eqref{eq:10.52}, \eqref{eq:10.53}, \eqref{eq:10.54} and \eqref{eq:10.55} we obtain Proposition \ref{prop:boundf}. 

\section{Bounds for $\mathcal{O}_T^\chi$}\label{sec:Bounds for OT}
\begin{proposition}\label{prop:boundot} Suppose that $\frac{\vert{KL}\vert}{\vert\Delta\vert} \ll N(\mathfrak{q})^{-3\epsilon}$. Then for any $I\in \mathcal{S}_{\cl}$, $y \in \widehat{O}_1^{\times}(\mathfrak{q})E$ and $t\asymp T$, we have the following \[\mathcal{O}_t^{\chi}(y,I)\ll N(\mathfrak{q})^{1/2+\epsilon}t^{-1/2}\vert\Delta KL\vert^{-1}.  \]
\end{proposition}
Again, by trivially estimating, we get Proposition \ref{prop:boundO}.

We recall that
\begin{align*}
    \mathcal{O}^{\chi}_t(y,I)&=\frac{1}{N(\mathfrak{q})^{\frac{1}{2}}}\sum_{\substack{\delta,k,l\\ \forall v \in S_{\infty}:\\\vert{\delta}\vert_v \ll \vert{\Delta}\vert_vN(\mathfrak{q})^{\epsilon}}}b_kc_l\widehat{V_0}\left(\delta\varpi_\mathfrak{q}\Delta\right)\sum_{a\in F^{\mathfrak{q}}/F^{\mathfrak{q,1}}}\psi_{v_{\mathfrak{q}}}\left(\frac{a\delta}{\varpi_{\mathfrak{q}}}\right)\chi(\iota_{\mathfrak{q}}(k(al)^{-1}))\nonumber\\
    &\times \mathbb{P}\phi_0^{u_{12}(\zeta_{\mathfrak{q}})}(a_1(I\iota_{\mathfrak{q}}(k(la)^{-1})yz_t)\chi(y_{\infty}), 
\end{align*}

To prove Proposition \ref{prop:boundot} we fix $I\in \mathcal{S}_{\cl}$, $o(y) \in \widehat{O}_1^{\times}(\mathfrak{q})$, $y_{\infty} \in E$ and $t>0$. 
We express $\mathbb{P}\phi_0^{u_{12}(\zeta_{\mathfrak{q}})}$ using the Fourier-Whittaker expansion, and we see that 

\begin{align*}
      \mathcal{O}^{\chi}_t(y,I) &= \frac{1}{(tN(\mathfrak{q}))^{\frac{1}{2}}}\sum_{\substack{\delta,k,l\\ \forall v \in S_{\infty}:\\\vert{\delta}\vert_v \ll \vert{\Delta}\vert_vN(\mathfrak{q})^{\epsilon}}}b_kc_l\widehat{V_0}\left(\delta\varpi_\mathfrak{q}\Delta\right)\sum_{a\in F^{\mathfrak{q}}/F^{\mathfrak{q,1}}}\psi_{v_{\mathfrak{q}}}\left(\frac{a\delta}{\varpi_{\mathfrak{q}}}\right)\chi(\iota_{\mathfrak{q}}(k(al)^{-1}))\nonumber\\
    &\times \sum_{\gamma \in F^{\times}}W_{\phi_0}\left(a_1(\gamma Iyz_t\iota_{\mathfrak{q}}(k(l\delta^{-1}))\right)\psi_{v_{\mathfrak{q}}}(\gamma k(la)^{-1}\alpha_{\mathfrak{q}}^{-1})) \\ &= \frac{1}{t^{\frac{1}{2}}}\sum_{\substack{\delta,k,l\\ \forall v \in S_{\infty}:\\\vert{\delta}\vert_v \ll \vert{\Delta}\vert_v^{-1}N(\mathfrak{q})^{\epsilon}}}\chi(\iota_{\mathfrak{q}}(\delta kl^{-1}))b_kc_l\widehat{V_0}\left(\delta\varpi_\mathfrak{q}\Delta\right) \sum_{\gamma \in F^{\times}}Kl_{\chi_{\mathfrak{q}}}(kl^{-1}\delta\gamma;\mathfrak{q})W_{\phi_0}\left(a_1(\gamma Iy_{\infty}z_t)\right).
\end{align*} Here $Kl_{\chi_{q}}(\cdot\ ;\mathfrak{q})$ is a Kloosterman sum, and in the notation of Section \ref{sec:gausskloosterman}, equation \eqref{eq:Klosterman}, we have
\[Kl_{\chi_{\mathfrak{q}}}(kl^{-1}\delta \gamma;\mathfrak{q})= Kl_{\chi_{\mathfrak{q}}}(kl^{-1}\delta\gamma\varpi_{\mathfrak{q}}^{-2};1,\psi_{v_{\mathfrak{q}}}).\]
We apply Cauchy--Schwarz, and using the Rankin--Selberg bound \eqref{eq:classicrankinselberg} we get 
\begin{align*}
    \vert{O_t^{\chi}(y,I)}\vert^2 &\ll \frac{N(\mathfrak{q})^{\epsilon}}{t} \sum_{\substack{\gamma \in F^{\times}\\ \gamma I \subset O_F  }}\frac{1}{N(\gamma)}\left\vert W_{\infty}(a_1(\gamma z_{N(I)^{-1}t}y_\infty))\right\vert^2\\&\times\left\vert \sum_{\substack{\delta,k,l\\ \forall v \in S_{\infty}:\\\vert{\delta}\vert_v \ll \vert{\Delta}\vert_v^{-1}N(\mathfrak{q})^{\epsilon}}}\chi(\iota_{\mathfrak{q}}(\delta kl^{-1}))b_kc_l\widehat{V}\left(\delta\varpi_\mathfrak{q}\Delta\right)Kl_{\chi_{\mathfrak{q}}}(kl^{-1}\delta\gamma;\mathfrak{q})\right\vert^2,
\end{align*}
where $W_{\infty} = \prod_v W_v$.
We open the square and consider $\delta_1,\delta_2,k_1,k_2,l_1$ and $l_2$ to be fixed. We see that it suffices to bound sums of the shape 
\begin{equation}\label{eq:11.1}
    \sum_{\substack{\gamma \in F^{\times}\\ \gamma I \subset O_F}}\frac{\left\vert W_{\infty}\left(a_1(\gamma z_{N(I)^{-1}t}y_{\infty})\right)\right\vert^2}{N(I)t^{-1}}Kl_{\chi_{\mathfrak{q}}}(k_1l_1^{-1}\delta_1\gamma;q)\overline{Kl_{\chi_{\mathfrak{q}}}(k_2l_2^{-1}\delta_2\gamma;\mathfrak{q})},
\end{equation}
where we observed that if $\gamma$ is such that $W_{\infty}(a_1(\gamma z_{N(I)^{-1}t}y_{\infty}) \neq 0$, then $$\vert{\gamma}\vert_v \asymp \vert (N(I)t^{-1})^{\frac{1}{d_v}}y_v \vert_v$$ for all $v\in S_{\infty}$, and so $N(\gamma) \asymp  N(I)t^{-1}$.
Let $$W(x) = \prod_{v<\infty}\mathbb{1}_{I_v}(x_v) \times \prod_{v\in S_{\infty}}\vert{W_v}(a_1(x_v))\vert^2 \in C_c^{\infty}(\A).$$ By Poisson summation \eqref{eq:11.1} equals  \begin{align}\label{eq:11.2}
    \frac{1}{N(\mathfrak{q})}\sum_{\substack{\gamma \in F^{\times}}}\widehat{W}(\gamma z_{N(I)t^{-1}}y_{\infty}^{-1}) \mathcal{C}_{\eta}(k_1l_1^{-1}\delta_1,k_2l_2^{-1}\delta_2,\gamma;1,\psi_{v_{\mathfrak{q}}}),
\end{align}
where $\mathcal{C}_{\eta}(k_1l_1^{-1}\delta_1,k_2l_2^{-1}\delta_2,\gamma;1,\psi_{v_{\mathfrak{q}}})$ is defined in Section \ref{sec:gausskloosterman}. For each $v\in S_{\infty}$ the Fourier transform $\widehat{W}_v$ satisfies $\widehat{W}_v(x) \ll_A (1+\vert{x}\vert)^{-A}$ for every $A>0$. If $\gamma\neq 0$ contributes to \eqref{eq:11.2}, then $\gamma \in I\mathfrak{D}^{-1}O_F$ and so, then $N(\gamma)\gg 1$. In particular, there is $v\in S_{\infty}$ such that $\vert\gamma\vert_v \gg 1$ and $\widehat{W}_{\infty}(\gamma z_{N(I)t^{-1}}y_{\infty})$ is negligible small. Using Lemma \ref{lem:correlationkloosterman} we see that \eqref{eq:11.2} is bounded by 
\[\frac{1}{N(\mathfrak{q})}\vert \mathcal{C}_{\eta}(k_1l_1^{-1}\delta_1,k_2l_2^{-1}\delta_2,\gamma;1,\psi_{v_{\mathfrak{q}}}) \vert + O_A(N(\mathfrak{q})^{-A}) \ll \frac{1}{N(\mathfrak{q})}N((k_1l_1^{-1}\delta_1-k_2l_2^{-1}\delta_2,\mathfrak{q}).  \]
 To conclude, we have 
\begin{align*}\vert O_t^{\chi}(y,I)\vert ^2 &\ll \frac {N(\mathfrak{q})^{\epsilon}}{t}\sum_{k_1,k_2,l_1,l_2,\delta_1,\delta_2} b_{k_1}\overline{b_{k_2}}c_{l_1}\overline{c_{l_2}}\widehat{V}\left(\delta_1\varpi_\mathfrak{q} \Delta\right)\overline{\widehat{V}\left(\delta_2\varpi_\mathfrak{q} \Delta\right)}\frac{N((k_1l_1^{-1}h_1-k_2l_2^{-1}h_1,\mathfrak{q})}{N(\mathfrak{q})
} \\&\ll \frac {N(\mathfrak{q})^2N(\mathfrak{q})^{\epsilon}}{t\vert\Delta\vert^2}(N(\mathfrak{q})^{-1}+\frac{\vert\Delta\vert}{N(\mathfrak{q})\vert KL\vert} ).\end{align*}
Since $\vert{\Delta}KL\vert \leq N(\mathfrak{q})^{1-3\epsilon}$ we see that the second term dominates the first one and this concludes the proof of Proposition \ref{prop:boundot}.

\section{Appendix: $v$-adic Bessel transform for new-vectors}\label{appendix}

Let $v$ be a finite place of $F$ and suppose that $\psi_v$ is unramified.\footnote{This assumption is not necessary for the computations, but it slightly simplifies matters.} Let $a(\pi_v)$ be the exponent conductor of $\pi_v$. Recall that running the argument of \cite{IchinoTemplier} one is faced with the following Bessel transforms of the functions $W_v \in \mathcal{W}(\pi_v,\psi_v)$:
\[B_{\pi_v}[W_v]\colon F_v^{\times}\ni y \mapsto \int_{F_v}(W_v)^{\iota}\left(\begin{pmatrix}
    y&&\\&1&\\&&1
\end{pmatrix}\begin{pmatrix}1&&\\x_v&1&\\&&1\end{pmatrix}\sigma_{23}\right)\ \mathrm{d}x_v, \]
where we recall that $(W_v)^\iota(g) = W_v(\sigma_{13}g^{-t})$.
In this Appendix, we compute the Bessel transform of the new-vectors $W_v^{\circ} \in \mathcal{W}(\pi_v,\psi_v)$.

\begin{proposition}\label{prop:computation-v-adic-bessel-transform}For the new-vector $W_v^\circ \in \mathcal{W}(\pi_v,\psi_v)$ the following identity holds
    \[B_{\pi_v}[W^{\circ}_v](y) = \epsilon(\pi_v,\psi_v)q_v^{\frac{a(\pi_v)}{2}}U^{\circ}_v(a_1(y\varpi_v^{a(\pi_v)})),\]
    where $U^\circ_v \in \mathcal{W}(\tilde{\pi}_v,\overline{\psi_v})$ is the new-vector of the contragredient representation.
\end{proposition}

In \emph{loc. cit.} the authors show that $B_{\pi}[W_v]$ is uniquely characterized by the following identity. 
\[\frac{\int_{F_v^{\times}}B_{\pi_v}[W_v](y)\overline{\eta}(y)\vert{y}\vert_v^{s-1}\ \mathrm{d}^{\times}y}{L(\overline{\eta}\tilde{\pi_v},s)} = \varepsilon(\eta\pi_v,1-s,\psi_v)\frac{\int_{F_v^{\times}}W_v(a_1(y))\eta(y)\vert{y}\vert_v^{-s}\ \mathrm{d}^{\times}y}{L(\eta\pi_v,1-s)}, \]
for every $\eta \in \widehat{F_v^{\times}}$. We restrict our attention to $\eta \in \widehat{O_v^{\times}}$ since every character $\eta \in \widehat{F_v^{\times}}$ can be written as $\eta = \eta_0\vert{\cdot}\vert_v^{it}$ for some $\eta_0 \in \widehat{O_v^{\times}}$ and $t\in \R$.

The identity has to be understood by analytic continuation, since the integrals generally only converge in some region.

The $\varepsilon$-factor is of the shape
\[\varepsilon(\eta\pi_v,s,\psi_v) = \varepsilon(\eta\pi_v,\psi_v)q_v^{a(\eta\pi_v)(1/2-s)},\]
for $\varepsilon(\eta\pi_v,\psi_v)$ a complex number of absolute value $1$.

If $W^{\circ}_v$ is the local new-vector (unique up to scalar multiple) in the Whittaker model $\mathcal{W}(\pi_v,\psi_v)$, then when it converges, the integral on the right hand side is 
\begin{align*}
    \int_{F_v^{\times}}W^{\circ}_v(a_1(y))\eta(y)\vert{y}\vert_v^{-s}\ \mathrm{d}^{\times}y = \begin{cases}
        L(\pi_v,1-s)& \eta = 1,\\
        0 & \text{otherwise.}
    \end{cases}
\end{align*}
Recall that the $L$-factor is defined so that 
\[e(W,\eta,s)\colon \CC\ni s\mapsto \frac{\int_{F_v^{\times}}W(a_1(y))\eta(y)\vert{y}\vert_v^{s-1}\ \mathrm{d}^{\times}y}{L(\eta\pi_v,s)} \]
has an entire continuation for all $W\in \mathcal{W}(\pi_v,\psi_v)$ (more precisely it is an element of $\CC[q_v^{\pm s}]$).
By the above computation and by the the identity theorem we deduce that when $W^{\circ}_v$ is the new-vector we have \begin{align*}
    e(W^{\circ}_v,\eta,s) = \begin{cases}
        1 & \eta =1 \\ 
        0 & \text{otherwise.}
    \end{cases}
\end{align*}
Consider the new-vector $U^{\circ}_v\in \mathcal{W}(\tilde{\pi}_v,\overline{\psi_v}) $  and consider the vector $U_v = (U^{\circ}_v)^{(a_1(\varpi_v^{a(\pi_v)}))}$ obtained by right translating $U_v^{\circ}$. Then, whenever it converges, we have 
\begin{align*}\int_{F_v^{\times}}U_v(a_1(y))\overline{\eta}(y)  \vert{y}\vert_v^{s-1} \ \mathrm{d}^{\times}y &= q_v^{a(\pi_v)(s-1)}\overline{\eta}(\varpi_v^{-a(\pi_v)})\int_{F_v^{\times}}U^{\circ}_v(a_1(y))\overline{\eta}(y)\vert{y}\vert_v^{s-1}\ \mathrm{d}^{\times}y\\ &=\begin{cases}
    q_v^{a(\pi_v)(s-1)}L(\tilde{\pi_v},s)&\eta =1, \\
    0 & \text{otherwise.}
\end{cases}\end{align*}
We deduce that \begin{align*}\int_{F_v^{\times}}B_{\pi}[W^{\circ}_v](y)\eta(y) \vert{y}\vert^{s-1}\ \mathrm{d}^{\times}y &= \delta_{\eta=1}\varepsilon(\pi_v,\psi_v)q_v^{a(\pi_v)(s-1/2)}L(\tilde{\pi_v},s)\\ &=\varepsilon(\pi_v,\psi_v)q_v^{\frac{a(\pi_v)}{2}}\int_{F_v^{\times}}U^{\circ}_v(a_1(y\varpi_v^{a(\pi)}))\eta(y)\vert{y}\vert^{s-1}\ \mathrm{d}^{\times}y\end{align*}
for all $\eta \in \widehat{O_{F_v}^{\times}}$. Both expressions converge absolutely for $\real(s)>1$. Taking the inverse Mellin transform we deduce that for $W^{\circ}_v$ the new-vector we have \[B_{\pi_v}[W^{\circ}_v](y) = \epsilon(\pi_v,\psi_v)q_v^{\frac{a(\pi_v)}{2}}U^{\circ}_v(a_1(y\varpi_v^{a(\pi_v)})).\]

\bibliographystyle{alphaurl}
\bibliography{references}
\end{document}